\documentclass[opre,nonblindrev]{informs3} % current default for manuscript submission

\usepackage{amssymb}
\usepackage{color}
\usepackage{rotating}
\usepackage{multirow}
\usepackage{latexsym}
\usepackage{secdot}
\usepackage{setspace}
\usepackage{wasysym}
\usepackage[us,12hr]{datetime}
\usepackage{bm}
\usepackage{pdfsync}
\usepackage{rccol}
\usepackage{mathrsfs}
\usepackage{epsfig}
\usepackage[margin=1in,paper=letterpaper]{geometry}
\usepackage{url}
\usepackage{multibib}
\newcites{appendix}{Online Appendix References}
\usepackage{enumitem}
\usepackage{stackrel}
\usepackage{soul}
\usepackage{appendix}
\usepackage{titlesec}

\newtheorem{thm}{Theorem}[section]

\newtheorem{lem}[thm]{Lemma}
\newtheorem{prop}[thm]{Proposition}

{\theoremstyle{THkey}}

\OneAndAHalfSpacedXI % current default line spacing
\usepackage{natbib}
 \bibpunct[, ]{(}{)}{,}{a}{}{,}%
 \def\bibfont{\small}%
 \def\bibsep{\smallskipamount}%

\EquationsNumberedThrough    % Default: (1), (2), ...
\MANUSCRIPTNO{} 

\newcommand{\ts}{{\,}}
\newcommand{\Acal}{{\mathcal A}}
\newcommand{\Fcal}{{\mathcal F}}
\newcommand{\Jcal}{{\mathcal J}}
\newcommand{\Lcal}{{\mathcal L}}
\newcommand{\Mcal}{{\mathcal M}}
\newcommand{\Ncal}{{\mathcal N}}
\newcommand{\Pcal}{{\mathcal P}}
\newcommand{\Qcal}{{\mathcal Q}}
\newcommand{\Tcal}{{\mathcal T}}
\newcommand{\Ucal}{{\mathcal U}}
\newcommand{\Vcal}{{\mathcal V}}
\newcommand{\Xcal}{{\mathcal X}}
\newcommand{\Ycal}{{\mathcal Y}}
\newcommand{{\xhat}}{{\widehat x}}
\newcommand{{\yhat}}{{\widehat y}}
\newcommand{{\zhat}}{{\widehat z}}
\newcommand{\xtilde}{{\widetilde x}}
\newcommand{\ytilde}{{\widetilde y}}
\newcommand{{\pbar}}{{\overline p}}
\newcommand{{\sbar}}{{\overline s}}
\newcommand{{\ubar}}{{\overline u}}
\newcommand{{\xbar}}{{\overline x}}
\newcommand{{\ybar}}{{\overline y}}
\newcommand{{\zbar}}{{\overline z}}
\newcommand{{\Cbar}}{{\overline C}}
\newcommand{{\Pbar}}{{\overline P}}
\newcommand{{\Rbar}}{{\overline R}}
\newcommand{{\alphabar}}{{\overline \alpha}}
\newcommand{{\betabar}}{{\overline \beta}}
\newcommand{{\nubar}}{{\overline \nu}}
\newcommand{{\thetabar}}{{\overline \theta}}
\newcommand{{\zetabar}}{{\overline \zeta}}
\newcommand{{\Deltabar}}{{\overline \Delta}}
\newcommand{{\Brm}}{\text{\rm B}}
\newcommand{{\Grm}}{\text{\rm G}}
\newcommand{{\Nrm}}{\text{\rm N}}
\newcommand{{\Urm}}{\text{\rm U}}
\newcommand{{\Xrm}}{\text{\rm X}}
\newcommand{{\Yrm}}{\text{\rm Y}}
\newcommand{{\Fs}}{{\text{\sf F}}}
\newcommand{{\Ss}}{{\text{\sf S}}}
\newcommand{{\sigmahat}}{{\widehat \sigma}}
\newcommand{\avec}{{\bm a}}
\newcommand{\cvec}{{\bm c}}
\newcommand{\evec}{{\bm e}}

\newcommand{\uvec}{{\bm u}}
\newcommand{\xvec}{{\bm x}}
\newcommand{\yvec}{{\bm y}}
\newcommand{\wvec}{{\bm w}}
\newcommand{\zvec}{{\bm z}}

\newcommand{\xvechat}{\widehat{\bm x}}
\newcommand{\yvechat}{\widehat{\bm y}}
\newcommand{\zvechat}{\widehat{\bm z}}

\newcommand{\xvectilde}{\widetilde{\bm x}}
\newcommand{\yvectilde}{\widetilde{\bm y}}
\newcommand{\zvectilde}{\widetilde{\bm z}}
\newcommand{\pvecbar}{\overline{\bm p}}
\newcommand{\uvecbar}{\overline{\bm u}}
\newcommand{\wvecbar}{\overline{\bm w}}
\newcommand{\xvecbar}{\overline{\bm x}}
\newcommand{\yvecbar}{\overline{\bm y}}
\newcommand{\zvecbar}{\overline{\bm z}}

\newcommand{\alphavecbar}{\overline{\bm \alpha}}
\newcommand{\Prmvecbar}{\overline{\bf P}}
\newcommand{{\Prmbar}}{{\overline{\mathrm{P}}}}

\newcommand{\avail}{{\text{\sf avail}}}
\newcommand{\availp}{{\text{\sf avail}}^{\text{\sf ep}}}
\newcommand{\apx}{{\text{\sf apx}}}
\newcommand{\apxp}{{\text{\sf apx}}^{\text{\sf ep}}}
\newcommand{\thetabarp}{{\thetabar}^{\text{\sf ep}}}
\newcommand{\opt}{{\text{\sf opt}}}
\newcommand{\lp}{{\text{\sf LP}}}
\newcommand{\reve}{\overline{\text{\sf rev}}}
\newcommand{\capa}{\overline{\text{\sf cap}}}
\newcommand{\reg}{\text{\sf R}}
\newcommand{\dis}{\text{\sf D}}
\renewcommand{\qed}{\hfill \mbox{\raggedright \rule{0.1in}{0.1in}}}
\newcommand{\ind}[1]{{\bf 1}_{(#1)}}
\newcommand{\ru}[1]{\lceil #1 \rceil}

\begin{document}
%%%%%%%%%%%%%%%%

\RUNAUTHOR{\normalfont Li, Rusmevichientong, Topaloglu, \today ;}
\RUNTITLE{\normalfont{Inter-Temporal Price Constraints in Dynamic Pricing}}

\TITLE{\large 
Inter-Temporal Price Constraints in Dynamic Pricing: Performance Guarantees Under Price Monotonicity and Promotion Fatigue \vspace{-4mm}}

\ARTICLEAUTHORS{
\AUTHOR{\mbox{\normalsize Weiyuan Li$^1$, Paat Rusmevichientong$^2$, Huseyin Topaloglu$^1$}}
\AFF{
\scriptsize $^1$School of Operations Research and Information Engineering, Cornell Tech, New York, NY 10044, USA
\\
\scriptsize $^2$Marshall School of Business, University of Southern California, Los Angeles, CA 90089, USA}
\vspace{-0.5mm}
\EMAIL{\scriptsize wl425@cornell.edu, rusmevic@marshall.usc.edu, topaloglu@orie.cornell.edu}
\\
\vspace{-1mm}
\normalsize September 22, 2026
}

\ABSTRACT{%
We study dynamic pricing problems under inter-temporal price constraints. We have resources with limited capacities. At each time period, we decide which products to make available and what prices to charge for the available products. The sale probability for a product depends on its price.~If we make a sale for a product, then we collect a revenue reflecting the price and consume the capacities of a combination of resources. We work with two types of inter-temporal constraints. In price monotonicity, the prices charged for a product at different time periods have to be monotone. In promotion fatigue, we can discount a product at most once over each time interval of a fixed length. Computing the optimal policy is intractable. We use fluid approximations to construct policies. Traditionally, policies from fluid approximations make randomized decisions at each time period by following an optimal solution to the fluid approximation, but such  randomized decisions easily violate price monotonicity or promotion fatigue constraints. We develop policies that sample price paths according to an optimal solution to the fluid approximation, while satisfying the inter-temporal constraints. Letting $c_{\min}$ be the smallest initial capacity of a resource and $L$ be the maximum number of resources used by a product, our policies have a performance guarantee of $\max\Big\{ \frac{1}{8L}, \ts \frac{1}{2} - \sqrt{\frac{\log c_{\min}}{2 \ts c_{\min}}} - \frac{L}{c_{\min}}\Big\}$. Thus, treating the number of resources used by a product as a constant, our policies have a constant-factor performance guarantee.~Under large resource capacities, our policies are guaranteed to obtain at least half of the optimal total expected revenue. The latter performance guarantee is tight in the sense that no policy can, in general, obtain more than half of the optimal objective value of the fluid approximation even under large resource capacities.
In our policies, we focus on two of the price paths sampled for each product. We also give a policy with an ex-post performance guarantee of $1 - \sqrt{\frac{2\log c_{\min}}{ c_{\min}}} - \frac{L + \Deltabar}{c_{\min}}$, where the parameter $\Deltabar$ depends on the difference between the total expected capacity consumptions of the two price paths. The latter policy is asymptotically optimal under large resource capacities, as long as the parameter $\Deltabar$ scales more slowly than $c_{\min}$. Lastly, we unify our approach to open the path for extensions to other inter-temporal price constraints.}

%\KEYWORDS{submodular optimization, approximate dynamic programming, online retail} 
%\HISTORY{This paper was first submitted on April 12, 1922 and has been with the authors for 83 years for 65 revisions.}

\maketitle

\vspace{-9mm}

\section{Introduction}

\vspace{-1mm}

Inter-temporal price constraints that link the prices charged for a product at different time periods in the selling horizon frequently appear in dynamic pricing applications. Fashion retailers may charge prices for a product that are decreasing over time, so that early purchasers pay a premium. On the other hand, providers of hospitality services, such as airlines and hotels, may charge prices for a trip or stay itinerary that are increasing over time, so that early purchasers enjoy a discount. Supermarkets concerned about customers being accustomed to discounts may limit the number of price promotions on a product over a certain interval of time. Retailers may set a promotion budget for a product, in which case, the total number of discounted units sold over a certain duration of time may need to be no larger than the promotion budget. Even without inter-temporal price constraints, dynamic pricing problems are challenging. When choosing the price for a product, one has to carefully balance between charging a high price, in which case, the revenue from the immediate sale would be larger but there would be a lower likelihood that the product is sold, against charging a low price, in which case, there would be a lower likelihood that the product goes unsold but the revenue from the immediate sale would be smaller. This tradeoff is further complicated by the fact that customers arriving over different portions of the selling horizon may have different price sensitivities as, for example, early purchasers of hospitality services, who tend to be leisure travelers, may have higher price sensitivities than late purchasers, who tend to be  business travelers.~When these complications are coupled with inter-temporal price constraints, finding good  policies in dynamic pricing applications becomes a challenging task.

\vspace{-0.15mm}

In this paper, we study dynamic pricing problems over a network of resources under \mbox{inter-temporal} price constraints. We have a set of resources with limited capacities. At each time period, we decide which products to make available and what prices to charge for the available products. The probability of making a sale for a product depends on its price. If we make a sale for a product, then we collect a revenue reflecting the price for the product and consume the capacities of a combination of resources that depends on the sold product. We work with two types of inter-temporal price constraints. In price monotonicity, the prices charged for a product at different time periods have to be monotone over time. In promotion fatigue, we can discount a product once over every time interval of a fixed length. Surprisingly, dynamic pricing problems with such inter-temporal price constraints are hardly studied.  Computing the optimal policy is intractable, so we turn to fluid approximations to construct policies with performance guarantees. 

\vspace{-0.15mm}

Policies from fluid approximations traditionally make randomized decisions at each time period by following an optimal solution to the fluid approximation, but making randomized decisions at each time period can easily result in price paths that violate price monotonicity or promotion fatigue constraints. We give fluid approximations under price monotonicity or promotion fatigue constraints, where we impose stochastic dominance only in the distribution of the prices charged at each time period or ensure that only the expected number of times we discount each product over every time interval of a fixed length is at most one. Using an optimal solution to these fluid approximations, however, we show that we can sample price paths each satisfying the inter-temporal constraints with probability one. Focusing on the sampled price paths, we construct policies with performance guarantees under both price monotonicity and promotion fatigue constraints. We unify our approach to facilitate extensions to other inter-temporal price constraints.

\vspace{-0.15mm}

%In other words, we impose price monotonicity constraint only in distribution in the fluid approximation.

{\bf \underline{Main Contributions\phantom{p}\!\!\!}:} We construct fluid approximations under price monotonicity or promotion fatigue constraints. We show that we can use an optimal solution to these fluid approximations to sample price paths satisfying the inter-temporal price constraints. Using these price paths, we develop approximate policies with performance guarantees.

{\underline{\it Fluid Approximations Under Price Monotonicity Constraints}.} We begin by focusing on price monotonicity constraints. We give a fluid approximation under price monotonicity constraints. In our fluid approximation, we only impose the constraint that the distribution of the price for a product at one time period stochastically dominates the one at another time period. 
Thus, it is not immediately clear whether we can use the fluid approximation to obtain price paths that satisfy price monotonicity constraints with probability one. Using an approach akin to common random numbers, we show that we can use an optimal solution to the fluid approximation to sample price paths that satisfy price monotonicity constraints, while ensuring that the marginal distribution of the price for each product at each time period matches that in the optimal solution to the fluid approximation. Letting $n$ be the number of possible price levels for a product and $T$ be the number of time periods in the selling horizon, the number of possible sampled price paths is $O(nT)$. We show that we can re-construct an optimal solution to the fluid approximation by using at most two of the sampled price paths for each product, which becomes critical in our approximate policy. Our approach for going from the fluid approximation to price paths that satisfy price monotonicity constraints with probability one appears to be novel and can unlock other constraints. 

{\underline{\it Approximate Policy and Performance Guarantee}.} Letting $c_{\min}$ be the smallest initial capacity of a resource and $L$ be the maximum number of resources used by a product, we give an approximate policy with a performance guarantee of $\max\Big\{ \frac{1}{8L}, \ts \frac{1}{2} - \sqrt{\frac{\log c_{\min}}{2 \ts c_{\min}}} - \frac{L}{c_{\min}}\Big\}$. Thus, treating the number of resources used by a product as a constant, our approximate policy has a \mbox{constant-factor} performance guarantee.~Under large resource capacities, our approximate policy is guaranteed to obtain at least half of the optimal total expected revenue. The optimal objective value of the fluid approximation gives an upper bound on the optimal total expected revenue. We establish the performance guarantee by comparing the total expected revenue of our approximate policy with the upper bound. The performance guarantee for our approximate policy is tight. In particular, we show that no policy can, in general, obtain more than half of the optimal objective value of the fluid approximation even under large resource capacities. In our approximate policy, we follow one of the two possible price paths for each product that we discuss at the end of the previous paragraph. We give a well-defined rule for selecting the right price path for each product. It is remarkable that we achieve our performance guarantees by following pre-fixed price paths, along with making each product available for purchase through randomized decisions.

{\underline{\it Ex-Post Performance Guarantee}.}  To deepen the understanding of our results, we  give another approximate policy with an ex-post performance guarantee of $1 - \sqrt{\frac{2\log c_{\min}}{ c_{\min}}} - \frac{L + \Deltabar}{c_{\min}}$, where the parameter $\Deltabar$ depends on the difference between the total expected capacity consumptions of the two price paths for each product. The latter policy is asymptotically optimal under large resource capacities, as long as the parameter $\Deltabar$ scales more slowly than $c_{\min}$. While the ex-post performance guarantee is interesting by itself, it also pinpoints the critical problem features to drive asymptotic optimality and the reasons for not necessarily achieving asymptotic optimality in general. 

{\underline{\it Promotion Fatigue Constraints}.} We move on to giving approximate policies under promotion fatigue constraints. Our development deviates from the earlier one at several points. We give a fluid approximation under promotion fatigue constraints, where we only impose that the expected number of times that each product is offered at the promoted price over every interval of a fixed length is at most one. In other words, we impose promotion fatigue constraints only in expectation in the fluid approximation. We show that we can use an optimal solution to the fluid approximation to sample price paths that satisfy promotion fatigue constraints with probability one. The approach that we use to sample price paths under promotion fatigue constraints is dramatically different from the one under price monotonicity constraints. Nevertheless, we show that we can still use at most two of the sampled price paths for each product to re-construct an optimal solution to the fluid approximation. Using the two price paths, we construct approximate policies that attain the same performance guarantees as under price monotonicity constraints. 

{\underline{\it Unifying Our Approach}.} The fluid approximation under price monotonicity constraints imposes a stochastic dominance relationship on the prices charged at different time periods, whereas the fluid approximation under promotion fatigue constraints imposes promotion fatigue constraints in expectation. Furthermore, our approach for sampling price paths from the two fluid approximations is different as well. Nevertheless, we give a unifying framework for the two inter-temporal price constraints. In particular, we show that if we can use a polytope with integer extreme points to represent the feasible price paths under some inter-temporal constraints, then we can give a fluid approximation under the corresponding inter-temporal constraints and develop approximate policies with performance guarantees. We establish that we indeed can represent the feasible price paths under both price monotonicity and promotion fatigue constraints by using a polytope with integer extreme points. There can be other inter-temporal price constraints that admit similar characterizations, in which case, we can make extensions to such constraints.

{\underline{\it Computational Experiments}.} We test the performance of our approximate policies on synthetically generated datasets, as well as on datasets based on the bookings at an urban hotel, where we focus on imposing monotonicity constraints on the prices. The performance of our approximate policies is dramatically better than what is indicated by the theoretical performance guarantees. In particular, the total expected revenues from our approximate policies are within a few percentage points of the upper bounds on the optimal total expected revenue provided by the fluid approximation, even when the capacities of the resources are not dramatically large.

{\underline{\it Discussion of Our Model and Results}.} Inter-temporal price constraints are ubiquitous in practice, but there is, to our knowledge, hardly any work on practical policies under such constraints. Our work is an attempt to close this critical gap. Fluid approximations have been the workhorse for constructing practical policies for dynamic pricing problems, but the policies from fluid approximations make randomized decisions at each time period according to an optimal solution to the fluid approximation. Once we start flipping coins to decide what price to charge for a product at different time periods, it becomes difficult to satisfy inter-temporal price constraints. Our approach for sampling price paths that satisfy the inter-temporal price constraints, while ensuring that the marginal distribution of the price for each product at each time period matches that in the optimal solution to the fluid approximation, appears to be novel.

In our model, if we make a sale for a product, then we consume the capacities of a combination of resources. The probability of getting a demand for a product depends on the price of the product. In this sense, our model is a natural extension of the standard network revenue management setting to pricing decisions. The inter-temporal price constraints link the prices for a particular product at different time periods. In other words, the inter-temporal price constraints do not link the prices for different products at different time periods. Such constraints allow us to capture a variety of important applications, but there are natural applications that involve inter-temporal price constraints across  products. Nevertheless, it is surprising that there is hardly any work that addresses the class of problems that we work on and it is remarkable that we can obtain performance guarantees by focusing only on a pre-fixed price path for each product.

{\bf \underline{Related Literature\phantom{p}\!\!\!}:} There is work on developing fluid approximations for dynamic pricing problems over a network of resources and extracting policies with performance guarantees, but this work does not consider inter-temporal price constraints; see \cite{GaRy94}, \cite{GaRy97}, \cite{MaMe06}, \cite{ErTo11}, \cite{MaDa22}. Because there are no inter-temporal price constraints, it is not difficult to extract randomized policies from such fluid approximations. Moreover, the asymptotic regimes in these papers require the resource capacities and length of the selling horizon to grow large with the same rate, whereas our asymptotic regime only requires the resource capacities to grow large. By solving the fluid approximation multiple times over the selling horizon, one can get even better performance guarantees in the same regime; see \cite{Ja14} and \cite{WaWa22}.

Related work studies restrictions on the timing or cost of price adjustments, under possibly unknown demand functions that need to be learned. In two connected papers, \cite{FeGa95} study the optimal timing of a single price change, whereas \cite{FeGa00} find the optimal timing of multiple price changes among a given menu of allowable price paths. \cite{ChJa16} design heuristics that require a small number of price changes. \cite{MaLe21} study price and assortment optimization under a pre-fixed calendar of offered prices or assortments and give performance guarantees. \cite{AhRy26} focus on a pricing problem with a limited number of price change opportunities, considering only a single product, but a rich demand model that allows the demand to be dependent on past sales. A dual stream of work to ours studies multi-product static pricing problems with constraints linking the prices of different products; see \cite{RuVa06}, \cite{KeLe13}, \cite{DaTo17}, \cite{HaSu19}, \cite{SuGa21}.

Going back to dynamic pricing problems but with a single product, focusing on the case where the demand function is unknown but the price for the product has to be monotonically decreasing over time,  \cite{JiLi22} give regret bounds when the demand has a parametric form. Considering the case of making markdown decisions with stationary demand functions, \cite{ChJa24} give policies with logarithmic loss in the scaling factor when the initial inventory of the product and length of the selling horizon are both scaled with the same rate.  Contrasting our work with the existing literature, we can work with non-stationary demand functions and multiple products that consume capacities of overlapping sets of resources. Thus, the pricing decisions are coupled through both the inter-temporal price constraints and overlapping sets of resources used by the products. Our approximate policies follow a pre-determined price path for each product, only changing the availability of the products in real-time, but otherwise adhering to a pre-determined price path when the product is offered. Moreover, our approach for using an optimal solution to the fluid approximation to construct price paths that satisfy the inter-temporal price constraints does not seem to have appeared in the existing literature. Thus, while our fluid approximations impose inter-temporal price constraints only in the distributional or expectation sense, we can indeed use the fluid approximations to satisfy the inter-temporal price constraints with probability one and obtain performance guarantees for the resulting approximate policies.

%A distinct but related literature also exists around single-product pricing papers with unknown demand and limited pricing opportunities; see \cite{BiCh25}

{\bf \underline{Organization}:} In Section \ref{sec:form}, we formulate the dynamic pricing problem under price monotonicity constraints. In Section \ref{sec:fluid}, we construct the corresponding fluid approximation. In Section \ref{sec:paths}, we show that we can sample price paths that satisfy price monotonicity constraints, while ensuring that the marginal distributions of the prices match those in an optimal solution to the fluid approximation. In Section \ref{sec:policy}, we develop our approximate policy and its performance guarantee. In Section \ref{sec:expost_policy}, we construct our approximate policy with the ex-post performance guarantee. In Section \ref{sec:fatigue}, we give the corresponding results under promotion fatigue constraints. In Section \ref{sec:gen_const}, we unify our approach to open the path for extensions to other inter-temporal price constraints. In Section \ref{sec:exp}, we give computational experiments. In Section \ref{sec:conc}, we conclude.

\newpage 

\section{Pricing Under Price Monotonicity Constraints}
\label{sec:form}

\vspace{-0mm}

The set of resources is $\Lcal$. The initial capacity of resource $i$ is $c_i$. The set of products is $\Jcal$. We capture the resources used by product $j$ by the vector $\avec_j = (a_{ij} : i \in \Lcal) \in \{0,1\}^{|\Lcal|}$, where $a_{ij} = 1$ if and only if product $j$ uses resource $i$. The set of price levels for a product is $\Ncal = \{1,\ldots,n\}$.~The price corresponding to price level $\ell$ for product $j$ is $r_j^\ell$. Thus, if we charge price level $\ell$ for product $j$ and make a sale for this product, then we obtain a revenue of $r_j^\ell$. We index the price levels in increasing order of revenues, so $r_j^1 \leq r_j^2 \leq \ldots \leq r_j^n$. The set of time periods in the selling horizon is $\Tcal = \{1,\ldots,T\}$. There is at most one customer arrival at each time period. The customer arriving at time period $t$ is interested in purchasing product $j$ with probability $\theta_{jt}$. If we charge price level $\ell$ for product $j$, then the customer purchases the product with probability $\gamma_{jt}^\ell$. In this case, setting $\lambda_{jt}^\ell = \theta_{jt} \ts \gamma_{jt}^\ell$, if we charge price level $\ell$ for product $j$ at time period $t$, then we make a sale for the product with probability~$\lambda_{jt}^\ell$.~At~each time period, we decide which products to make available for purchase and what prices to charge for the available products. The prices that we charge for each product have to be monotonically increasing over the time periods. If a product is made available at a time period, then its price must be at least as large as the largest price charged at the previous time periods. Our goal is to maximize the total expected revenue over the selling horizon, while ensuring that the prices charged for each product are monotonically increasing. 

We give a dynamic program to compute the optimal policy. 
%To ensure that the prices that we charge for a product are monotone over time periods, we keep track of the current lower bound on the price that we can charge for each product. After we make the pricing decision for a product a a particular time period, we ensure that the lower bound on the price at the subsequent time period is at least as large as the price charged at the current time period, as well as the current lower bound on the price 
To capture the state of the system at the beginning of a generic time period, we keep the remaining capacities of the resources by using the vector $\wvec = (w_i : i \in \Lcal) \in \mathbb Z_+^{|\Lcal|}$, where $w_i$ is the remaining capacity of resource~$i$. In addition to keeping track of the remaining resource capacities, to ensure that the prices for a product are monotonically increasing over time, we represent the lower bounds on the prices that we can charge for the products  at the current time period by using the vector \mbox{$\zvec = (z_j^\ell : j \in \Jcal,~\ell \in \Ncal) \in \{0,1\}^{|\Jcal| \times n}$}, where $z_j^\ell = 1$ if and only if price level $\ell$ is the lower bound on the price for product~$j$.~If $z_j^\ell = 1$, then we must either charge price level $\ell$ or larger price levels for product~$j$ or not make the product available for purchase at all. In our dynamic program, we use the pair $(\wvec,\zvec)$ as the state of the system. To capture the decisions at a generic time period, we use the vector \mbox{$\xvec = (x_j^\ell : j \in \Jcal,~\ell \in \Ncal ) \in \{0,1\}^{|\Jcal|\times n}$}, where $x_j^\ell = 1$ if and only if we charge price level $\ell$ for product~$j$.~We have $\sum_{\ell \in \Ncal} x_j^\ell \leq 1$  for all $j \in \Jcal$ so that we can charge at most one price level for a product at a particular time period.~Having $\sum_{\ell \in \Ncal} x_j^\ell =0$ implies that we do not make product~$j$ available for purchase. Also, we use the vector $\yvec=  (y_j^\ell : j \in \Jcal,~\ell \in \Ncal) \in \{0,1\}^{|\Jcal| \times n}$ to represent the lower bounds on the prices that we can charge for the products after the pricing decisions at the current time period, where $y_j^\ell=1$ if and only if price level $\ell$ is the lower bound on the price for product $j$. We use the pair $(\xvec,\yvec)$ as the decisions in our dynamic program. If the state of the system at time period $t$ is $(\wvec,\zvec)$, then the set of feasible decisions is given by 
\begin{align}
\Fcal_t(\wvec,\zvec) = \Bigg\{& (\xvec,\yvec) \in \{0,1\}^{2 \times |\Jcal| \times n} : \sum_{j \in \Jcal} \sum_{\ell \in \Ncal} a_{ij} \ts \lambda_{jt}^\ell \ts x_j^\ell \leq w_i~~\forall \ts i \in \Lcal,~~~\sum_{\ell \in \Ncal} x_j^\ell \leq 1~~\forall \ts j \in \Jcal,
\nonumber
\\
&
z_j^\ell + \sum_{k=1}^{\ell-1} x_j^k \leq 1~~\forall \ts j \in \Jcal,~\ell \in \Ncal,~~~
y_j^\ell = x_j^\ell + \Big( 1 - \sum_{k \in \Ncal} x_j^k \Big) \ts z_j^\ell ~~\forall \ts j \in \Jcal,~\ell \in \Ncal
 \Bigg\}.
 \label{eqn:feas_monotone}
\end{align}

By the first constraint, if we charge a price for a product that generates demand for the product and the product uses the capacity of resource $i$, then we must have remaining capacity for the resource.  The second constraint ensures that we choose at most one price level for product $j$. If we have $\sum_{\ell \in \Ncal} x_j^\ell < 1$, then we do not make product $j$ available for purchase. In the third constraint, if the current lower bound for the price of product $j$ is price level $\ell$, then we cannot charge a price smaller than price level $\ell$ for product $j$. In the fourth constraint, if we charge price level $\ell$ for product~$j$ or we do not make product $j$ available for purchase and price level $\ell$ is the current lower bound on the price for product $j$, then price level $\ell$ is the lower bound on the price of product~$j$ after making the pricing decisions at the current time period. The initial resource capacities are given by the vector $\cvec = (c_i : i \in \Lcal)$. Letting $\evec^1 = (e_j^\ell : j \in \Jcal,~\ell \in \Ncal) \in \{0,1\}^{|\Jcal| \times n}$ be such that $e_j^\ell = 1$ if and only if $\ell = 1$, the initial state is $(\cvec,\evec^1)$. Using the boundary condition that $J_{T+1} = 0$, we can compute the optimal policy through the dynamic program 
\begin{align}
J_t(\wvec,\zvec)  = \! \!\!\max_{(\xvec,\yvec) \in \Fcal_t(\wvec,\zvec)} \! \Bigg\{ \sum_{j \in \Jcal} \sum_{\ell \in \Ncal} \lambda_{jt}^\ell \ts x_j^\ell \Big\{ r_j^\ell + J_{t+1}(\wvec - \avec_j , \yvec) \Big\} + \Big(1 - \sum_{j \in \Jcal} \sum_{\ell \in \Ncal} \lambda_{jt}^\ell \ts x_j^\ell \Big) \ts J_{t+1}(\wvec,\yvec) \Bigg\}. \!\!
\label{eqn:dp}
\end{align}
Noting that the initial state of the system corresponds to the vector $(\cvec,\evec^1)$, the optimal total expected revenue is given by $\opt = J_1(\cvec,\evec^1)$.

We explicitly model the possibility of not making a product available for purchase. In dynamic pricing problems, it is common to assume the presence of a large enough price such that if we charge the large price for a product, then there is no demand for the product. Charging the large price for a product is equivalent to not making the product available for purchase. In our formulation, because we ensure that the price of a product is monotonically increasing over the time periods, charging the large price simply for not making a product available for purchase at a particular time period possibly creates confusion as to whether we have to continue charging the large price at the subsequent time periods. By explicitly modeling the possibility of not making a product available and avoiding the assumption of the presence of a large enough price to shut off demand, it is clear that if we do not make a product available at a particular time period, then we can still offer the product at a subsequent time period, as long as the price that we charge for the product is at least as large as the prices charged at the previous time periods. 

%Also, we represent the lower bound on the price of a product at a subsequent time period as a decision. To maximize the total expected revenue, it is optimal to set the lower bound at the subsequent time period as small as possible. Therefore, even though we represent the lower bound at a subsequent time period as a decision, it is deterministically fixed by the pricing decision and lower bound at the current time period.

\section{Fluid Approximation}
\label{sec:fluid}

\vspace{-2mm}

The dynamic programming formulation in the previous section involves a high-dimensional state variable, so it is difficult to compute the optimal policy. We give a fluid approximation that allows us to come up with policies with performance guarantees. Our fluid approximation is a linear program, where we ensure that the probability distribution of the prices charged for a product at a particular time period stochastically dominates those at earlier time periods. The latter condition serves as a proxy to the requirement that the prices that we charge for a product have to be monotonically increasing over time. Even though our fluid approximation only imposes a dominance constraint between the probability distributions of the prices, we will extract a policy from the fluid approximation such that the charged prices are monotone. We use the decision variable $x_{jt}^\ell$ to capture the probability that we charge price level $\ell$ for product $j$ at time period $t$, as well as the decision variable $y_{jt}^\ell$ to capture the probability that price level $\ell$ is the lower bound on the price for product $j$ after making the pricing decisions at time period $t$. Using the vectors of decision variables $\xvec = (x_{jt}^\ell: j \in \Jcal,~\ell \in \Ncal,~t \in \Tcal)$  and $\yvec = (y_{jt}^\ell: j \in \Jcal,~\ell \in \Ncal,~t \in \Tcal)$, we consider the problem 
\begin{align}
Z_\lp^* = \max_{(\xvec,\yvec) \in [0,1]^{2 \times |\Jcal| \times n \times T}}
\Bigg\{ 
&\sum_{t \in \Tcal} \sum_{j \in \Jcal} \sum_{\ell \in \Ncal} r_j^\ell \ts \lambda_{jt}^\ell \ts x_{jt}^\ell ~:~
\sum_{t \in \Tcal} \sum_{j \in \Jcal} \sum_{\ell \in \Ncal} a_{ij} \ts \lambda_{jt}^\ell \ts x_{jt}^\ell \leq c_i \qquad \forall \ts i \in \Lcal,
\nonumber
\\
&\qquad x_{jt}^\ell \leq y_{jt}^\ell \qquad \forall \ts j \in \Jcal,~\ell \in \Ncal,~t \in \Tcal,
\phantom{\sum}
\nonumber
\\
& \qquad
\sum_{k=\ell}^n y_{j,t-1}^k \leq \sum_{k=\ell}^n y_{jt}^k \qquad \forall \ts j \in \Jcal,~\ell \in \Ncal,~t \in \Tcal \setminus \{1\},
\nonumber
\\
& \qquad \sum_{\ell \in \Ncal} y_{jt}^\ell = 1 \qquad \forall \ts j \in \Jcal,~t \in \Tcal \Bigg\}.
\tag{\sf \small Fluid Approximation}
\label{eqn:fluid}
\end{align}

\vspace{-0.5mm}

In the first constraint, we ensure that the total expected capacity consumption of a resource does not exceed its initial capacity. In the second constraint, noting that the price charged for a product has to be monotonically increasing over time, we use the fact that if we charge price level~$\ell$ for product $j$ at time period $t$, then the lower bound on the price of product $j$ after making the pricing decisions at time period $t$ has to be price level~$\ell$. The sense of this constraint is less than or equal to, because price level $\ell$ also becomes the lower bound on the price of product $j$ after making the pricing decisions at time period $t$ when we do not make product $j$ available at time period $t$, but charge price level $\ell$ for product $j$ at an earlier time period. The third constraint follows because the lower bounds on the price of a product have to be monotonically increasing over the time periods, so if the lower bound on the price of product $j$ at time period $t-1$ is price level $\ell$ or larger, then  the lower bound on the price of product $j$ at time period~$t$ must be price level $\ell$ or larger. In the fourth constraint, we ensure that the lower bound on the price of product $j$ at time period $t$ has to follow a valid probability distribution. In the objective function, we compute the total expected revenue over the selling horizon. In the next proposition, we show that the optimal objective value of the \ref{eqn:fluid} is an upper bound on the optimal total expected revenue.

\vspace{-1mm}

\begin{prop}[Upper Bound] 
\label{pro:ub}
Noting that $Z_\lp^*$ is the optimal objective value of the \ref{eqn:fluid} and $\opt$ is the optimal total expected revenue, we have $Z_\lp^* \geq \opt$. 
\end{prop}

\vspace{-1mm}

The proof is in Appendix \ref{sec:ub}. In the proof, we use the decisions of the optimal policy to construct a feasible solution to the \ref{eqn:fluid} such that this solution provides an objective value of $\opt$ for the \ref{eqn:fluid}. In this case, the optimal objective value of the \ref{eqn:fluid} must be at least $\opt$. In particular, we define the Bernoulli random variable $\Xrm_{jt}^\ell$ such that $\Xrm_{jt}^\ell =1$ if and only if the optimal policy charges price level $\ell$ for product $j$ at time period $t$. Furthermore, we define the random variable $\Yrm_{jt}^\ell$ recursively as $\Yrm_{jt}^\ell = \Xrm_{jt}^\ell + (1-\sum_{k \in \Ncal} \Xrm_{jt}^k) \ts \Yrm_{j,t-1}^\ell$ with the boundary condition that $\Yrm_{j0}^1 =1$ and $\Yrm_{j0}^\ell = 0$ for all $\ell \in \Ncal \setminus \{1\}$. In this case, we show that the solution $(\xvecbar,\yvecbar)$ with $\xbar_{jt}^\ell = \mathbb E\{ \Xrm_{jt}^\ell\}$ and $\ybar_{jt}^\ell = \mathbb E\{ \Yrm_{jt}^\ell\}$ for all $j \in \Jcal$, $\ell \in \Ncal$ and $t \in \Tcal$ is feasible to the \ref{eqn:fluid} and provides an objective value of $\opt$ for the \ref{eqn:fluid}. Because the optimal objective value of the \ref{eqn:fluid} is an upper bound on the optimal total expected revenue, to establish a performance guarantee for our approximate policy, we will compare the total expected revenue of our approximate policy with the optimal objective value of the \ref{eqn:fluid}. If the total expected revenue of our approximate policy exceeds a certain fraction of the upper bound on the optimal total expected revenue, then the total expected revenue of our approximate policy exceeds the same fraction of the optimal total expected revenue as well.

An important question is how we can extract an approximate policy from the \ref{eqn:fluid} while making sure that the prices that we charge for a product are monotonically increasing over time. A simple approach for extracting policies from the fluid approximations is to use randomization based on an optimal solution to the fluid approximation. Using $(\xvecbar,\yvecbar)$ to denote an optimal solution to the \ref{eqn:fluid}, if we follow this approach for our fluid approximation, then our approximate policy would charge price level $\ell$ for product $j$ at time period $t$ with probability $\xbar_{jt}^\ell$, but randomizing the decisions of the approximate policy in this way can potentially end up with a price path for a product that is not monotonically increasing over the time periods.~One of our contributions is to come up with a strategy to sample price paths by using an optimal solution to the \ref{eqn:fluid}, while ensuring that these price paths are monotonically increasing over the time periods. We give this sampling strategy in the next section. Coming up with such a sampling strategy is not difficult by using common random numbers.~More~surprisingly, however, we will show that our sampling strategy comes up with at most two price paths for each product.~Using one of these price paths for each product, while not offering the product at each time period with a certain probability, we construct our approximate policy.

\newpage
~
\vspace{-16mm}
\section{Sampling Price Paths from the Fluid Approximation}
\label{sec:paths}

\vspace{0.0mm}

We give an approach to sample the prices for each product by using an optimal solution to the \ref{eqn:fluid} so that the prices for a product are monotonically increasing over the time periods. This sampling approach forms an important part of our approximate policy.~We use $(\xvecbar,\yvecbar)$ to denote an optimal solution to the \ref{eqn:fluid}. Letting $\text{\sf Unif}$ be the uniform random variable over the interval $[0,1)$, for product $j$, we define the random price path \mbox{$\Prmvecbar_j = (\Prmbar_{jt} : t \in \Tcal) \in \Ncal^T$} such that $\Prmbar_{jt} = \ell$ if and only if $\sum_{k=1}^{\ell-1} \ybar_{jt}^k \leq \text{\sf Unif} < \sum_{k=1}^\ell \ybar_{jt}^k$. By the fourth constraint in the \ref{eqn:fluid}, $\sum_{\ell \in \Ncal} \ybar_{jt}^\ell =1$, so the marginal distribution of the price levels in the random price path $\Prmvecbar_j$ satisfies $\mathbb P \{ \Prmbar_{jt} = \ell \} = \ybar_{jt}^\ell$. Furthermore, we use the same uniform random variable $\text{\sf Unif}$ to sample the price level for product $j$ at different time periods. Subtracting both sides of the third constraint from one, we have $\sum_{k=1}^{\ell-1} \ybar_{j,t-1}^k \geq \sum_{k=1}^{\ell-1} \ybar_{jt}^k$ for all $\ell \in \Ncal$ and \mbox{$t \in \Tcal \setminus \{1\}$}.~In this case, if we have $\Prmbar_{j,t-1} = \ell$ so that $\sum_{k=1}^{\ell-1} \ybar_{j,t-1}^k \leq \text{\sf Unif} < \sum_{k=1}^\ell \ybar_{j,t-1}^k$, then we also have \mbox{$\text{\sf Unif} \geq \sum_{k=1}^{\ell-1} \ybar_{j,t-1}^k  \geq \sum_{k=1}^{\ell-1} \ybar_{jt}^k$}, so we get $\Prmbar_{jt} \geq \ell$. Thus, the price levels in the random price path $\Prmvecbar_j$ are monotone, so $\Prmbar_{j1} \leq \Prmbar_{j2} \leq \ldots \leq \Prmbar_{jT}$ with probability one. At the top of Figure \ref{fig:paths}, we show the values of $\{\ybar_{j,t-1}^\ell : \ell \in \Ncal\}$ and $\{ \ybar_{jt}^\ell : \ell \in \Ncal\}$ for a case with $n=4$. Note that $\sum_{k=1}^\ell \ybar_{j,t-1}^k \geq \sum_{k=1}^\ell \ybar_{jt}^k$ for all $\ell = 1,\ldots,4$. For the realization of the uniform random variable $\text{\sf Unif}$ with dotted lines, we have $\Prmbar_{j,t-1} = 2$ and $\Prmbar_{jt} = 3$ in the figure. 

\vspace{0.1mm}

There are $O(nT)$ possible realizations of the random price path $\Prmvecbar_j$. In particular, the realization of $\Prmbar_{jt}$ does not change as long as the uniform random variable $\text{\sf Unif}$ takes a value in the interval  $[\sum_{k=1}^{\ell-1} \ybar_{jt}^k , \sum_{k=1}^\ell \ybar_{jt}^k)$. For all $\ell \in \Ncal$ and $t \in \Tcal$, we collect the end points of such intervals to obtain the set of points $\{ \sum_{k=1}^\ell \ybar_{jt}^k : \ell \in \Ncal,~t \in \Tcal\} \cup \{0\}$, drop the duplicates and sort the remaining ones in increasing order to obtain the set of points $\{\nubar_j^q : q = 0,1,\ldots,m\}$ with $0 = \nubar_j^0 < \nubar_j^1 < \ldots < \nubar_j^m = 1$ and $m = O(nT)$. In this case, the realization of none of the price levels in the random price path $\Prmvecbar_j = (\Prmbar_{jt} : t \in \Tcal)$ changes as long as the uniform random variable $\text{\sf Unif}$ takes a value in one of the intervals in the collection \mbox{$\{[\nubar_j^{q-1},\nubar_j^q): q = 1,\ldots,m\}$}, which implies that the number of possible realizations of the random price path $\Prmvecbar_j$ is as large as the number of intervals in the collection \mbox{$\{[\nubar_j^{q-1},\nubar_j^q): q = 1,\ldots,m\}$}. At the bottom of Figure \ref{fig:paths}, we show the set of points $\{ \nubar_j^q : q = 0,1,\ldots,m\}$ for a case with \mbox{$\Tcal = \{t-1,t\}$}.~As long as the uniform random variable $\text{\sf Unif}$ takes a value in the interval $[\nu_j^3,\nu_j^4)$, for example, we have $\Prmbar_{j,t-1} = 2$ and $\Prmbar_{jt} = 3$ in the figure. We use $\{ \pvecbar_j^q : q \in \Mcal\}$ with $|\Mcal| = O(nT)$ to denote the possible realizations of the random price path $\Prmvecbar_j$. For each $q \in \Mcal$, we refer to $\pvecbar_j^q$ simply as a price path for product $j$. The price levels in the price path $\pvecbar_j^q$ are $(\pbar_{jt}^q : t \in \Tcal) \in \Ncal^T$. By our construction, each of these price paths is monotone, so $\pbar_{j,t-1}^q \leq \pbar_{jt}^q$. 

\vspace{0.1mm}

\begin{figure}
\begin{center}
\includegraphics{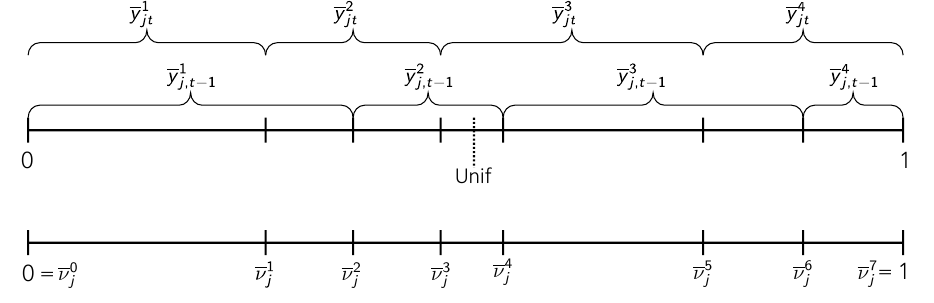}
\caption{Construction of the price path by using an optimal solution to the fluid approximation.}
\label{fig:paths}
\end{center}
\end{figure}

Using the collection of price paths $\{\pvecbar_j^q \!:q \in\! \Mcal\}$ for product $j$, we give an alternative representation of the \ref{eqn:fluid}. We use the vector of decision variables $\zvec = (z_j^q : j \in \Jcal,~q \in \mathcal M)$, where $z_j^q$ is the probability of offering price path $q$ for product $j$. Using $\ind{\cdot}$ to denote the indicator function, we consider the linear program given by 
\begin{align}
\max_{(\xvec,\yvec,\zvec) \in [0,1]^{|\Jcal| \ts (2 \times n \times T + |\Mcal|)}}
\Bigg\{ 
&\sum_{t \in \Tcal} \sum_{j \in \Jcal} \sum_{\ell \in \Ncal} r_j^\ell \ts \lambda_{jt}^\ell \ts x_{jt}^\ell ~:~
\sum_{t \in \Tcal} \sum_{j \in \Jcal} \sum_{\ell \in \Ncal} a_{ij} \ts \lambda_{jt}^\ell \ts x_{jt}^\ell \leq c_i \qquad \forall \ts i \in \Lcal,
\nonumber
\\
&\qquad x_{jt}^\ell \leq y_{jt}^\ell \qquad \forall \ts j \in \Jcal,~\ell \in \Ncal,~t \in \Tcal,
\phantom{\sum_n^n}
\nonumber
\\
& \qquad
y_{jt}^\ell = \sum_{q\in \Mcal} \ind{\pbar_{jt}^q = \ell} \ts z_j^q \qquad \forall \ts j \in \Jcal,~\ell \in \Ncal,~t \in \Tcal,
\nonumber
\\
& \qquad \sum_{q \in \mathcal M} z_j^q = 1 \qquad \forall \ts j \in \Jcal \Bigg\}.
\label{eqn:path_fluid}
\end{align}
The interpretation of the decision variables $x_{jt}^\ell$ and $y_{jt}^\ell$ in the problem above is the same as that in the \ref{eqn:fluid}. The first and second constraints in the problem above are the same as those in the \ref{eqn:fluid}. In the third constraint above, the probability that price level $\ell$ is the lower bound on the price for product $j$ after making the pricing decisions at time period $t$ is equal to the probability of choosing a price path for product $j$ that charges price level $\ell$ at time period~$t$.~In the fourth constraint, we ensure that the total probability of choosing a price path for product~$j$ is equal to one. Recall that we construct the collection of price paths $\{ \pvecbar_j^q : q \in \Mcal\}$ for product~$j$ by using an optimal solution to the \ref{eqn:fluid}. In the next proposition, we show that problem (\ref{eqn:path_fluid}) has the same optimal objective value as the \ref{eqn:fluid}. Thus, the decision variables $(z_j^q : q \in \Mcal)$ allow us to choose probabilities with which we should be following each price path for product $j$, while preserving the optimal objective value of the \ref{eqn:fluid}.

\vspace{-1mm}

\begin{prop}[Equivalence of Fluid Approximations]
\label{pro:equivalence}
If $(\xvecbar,\yvecbar,\zvecbar)$ is an optimal solution to problem $(\ref{eqn:path_fluid})$, then $(\xvecbar,\yvecbar)$ is an optimal solution to the \ref{eqn:fluid}.
\end{prop}

\vspace{-1mm}

We give the proof in Appendix \ref{sec:equivalence}. There are two parts in the proof. In the first part, considering an optimal solution $(\xvecbar,\yvecbar,\zvecbar)$ to problem (\ref{eqn:path_fluid}), we show that the solution $(\xvecbar,\yvecbar)$ is feasible to the \ref{eqn:fluid}. In this part, to be able to show that the solution $(\xvecbar,\yvecbar)$ satisfies the third constraint in the \ref{eqn:fluid}, we critically use the fact that the price paths $\{ \pvecbar_j^q : q \in \Mcal\}$ are monotone so that we have $\ind{\pbar_{j,t-1}^q \geq \ell} \leq \ind{\pbar_{jt}^q \geq \ell}$ for all $\ell \in \Ncal$. In the second part, recalling $\Prmvecbar_j$ is the random price vector that we constructed at the beginning of this section, considering an optimal solution $(\xvecbar,\yvecbar)$ to the \ref{eqn:fluid}, we show that the solution $(\xvecbar,\yvecbar,\zvecbar)$ with $\zbar_j^q = \mathbb P \{ \Prmvecbar_j = \pvecbar_j^q\}$ is feasible to problem (\ref{eqn:path_fluid}). In this part, to be able to show that the solution $(\xvecbar,\yvecbar,\zvecbar)$ satisfies the third constraint in problem (\ref{eqn:path_fluid}), we critically use the fact that our construction of the random price path $\Prmvecbar_j = (\Prmbar_{jt} : t \in \Tcal)$ ensures that we have $\mathbb P \{ \Prmbar_{jt} = \ell \} = \ybar_{jt}^\ell$.

By Proposition \ref{pro:equivalence}, we can obtain an optimal solution to the \ref{eqn:fluid} by using problem (\ref{eqn:path_fluid}). Furthermore, the optimal objective values of the two problems match. In problem (\ref{eqn:path_fluid}), the decision variables $(z_j^q : j \in \Jcal,~q \in \Mcal)$ capture the probability of following each price path for each product. One of the useful properties of problem (\ref{eqn:path_fluid}) is that an extreme point solution to this problem puts strictly positive probabilities on at most two price paths for each product. Therefore, noting that the optimal objective value of problem (\ref{eqn:path_fluid}) is equal to that of the \ref{eqn:fluid}, intuitively speaking, we can recover an optimal solution to the \ref{eqn:fluid} by using at most two price paths for each product. In the next proposition, we give this result.

\vspace{-1mm}

\begin{prop}[Extreme Points]
\label{pro:extreme}
Letting $(\xvecbar,\yvecbar,\zvecbar)$ be an extreme point solution to problem $(\ref{eqn:path_fluid})$, there are two price paths $\Fs,\Ss \in \Mcal$ for each product $j$ such that $\zbar_j^q = 0$ for all $q \in \Mcal \setminus \{\Fs,\Ss\}$.  
\end{prop}

\vspace{-1mm}

The proof is in Appendix \ref{sec:extreme}. In the proof, we use the additional decision variables \mbox{$(w_j : j \in \Jcal)$} to express the first constraint in (\ref{eqn:path_fluid})  as $\sum_{j \in \Jcal} a_{ij} \ts w_j \leq c_i$ for all $i \in \Lcal$ and $\sum_{t \in \Tcal} \sum_{\ell \in \Ncal} \lambda_{jt}^\ell \ts x_{jt}^\ell = w_j$ for all $j \in \Jcal$. In this case, if we fix the values of the decision variables $(w_j : j \in \Jcal)$, then the set of feasible solutions for problem (\ref{eqn:path_fluid}) decomposes by the products. For fixed value of $w_j$, we relate the set of feasible solutions corresponding to product $j$ to the set of feasible solutions to a variant of the knapsack problem. Using non-trivial properties of the set of feasible solutions to the knapsack problem, we show that at most two of the decision variables $(z_j^q : q \in \Mcal)$ take strictly positive values in an extreme point solution. In the notation in Proposition \ref{pro:extreme}, the superscripts $\Fs$ and $\Ss$ simply stand for the \underline{\it first\phantom{p}\!\!\!} and \underline{\it second\phantom{p}\!\!\!} price paths. In our approximate policies, for each product~$j$, we carefully choose one of the price paths $\Fs$ or $\Ss$. At each time period $t$, for each product $j$, we either charge the price given by the chosen price path or do not make the product available. The choice between charging the price given by the chosen price path and not making the product available for purchase is a probabilistic choice according to carefully chosen probabilities. Because the price paths $\{\pvecbar_j^q : q \in \Mcal\}$ are monotone, the prices charged by our approximate policies are monotone over time. We give performance guarantees for approximate policies of this form.

\section{Approximate Policy}
\label{sec:policy}
\vspace{-2mm}

We construct an approximate policy by using an optimal solution to problem (\ref{eqn:path_fluid}). Throughout this section, we use $(\xvecbar,\yvecbar,\zvecbar)$ to denote an optimal extreme point solution to problem (\ref{eqn:path_fluid}). Noting that the optimal objective value of the \ref{eqn:fluid} is $Z_\lp^*$, by Proposition \ref{pro:equivalence}, the optimal objective value of problem (\ref{eqn:path_fluid}) is $Z_\lp^*$ as well. By Proposition \ref{pro:extreme}, for each product $j$, there are two price paths \mbox{$\Fs, \Ss \in \Mcal$}, such that $\zbar_j^\Fs \geq 0$, $\zbar_j^\Ss \geq 0$ and $\zbar_j^q = 0$ for all $q \in \Mcal \setminus \{\Fs,\Ss\}$. We define a few pieces of notation. We use $\Rbar_j$ to denote the total expected revenue provided by product $j$ in the optimal objective value of problem (\ref{eqn:path_fluid}). In particular, we have $\Rbar_j = \sum_{t \in \Tcal} \sum_{\ell \in \Ncal} r_j^\ell \ts \lambda_{jt}^\ell \ts \xbar_{jt}^\ell$. Similarly, we use $\Cbar_j$ to denote the total expected capacity consumption of product $j$ in the optimal solution to problem (\ref{eqn:path_fluid}). In other words, we have $\Cbar_j = \sum_{t \in \Tcal} \sum_{\ell \in \Ncal} \lambda_{jt}^\ell \ts \xbar_{jt}^\ell$. Noting the objective function in (\ref{eqn:path_fluid}), we have $Z_\lp^* = \sum_{j \in \Jcal} \Rbar_j$, whereas noting the first constraint in  (\ref{eqn:path_fluid}), we have $\sum_{j \in \Jcal} a_{ij} \ts \Cbar_j \leq c_i$ for all $i \in \Lcal$. We use $\Rbar_j$ and $\Cbar_j$ to characterize the total expected revenue and capacity consumption from product $j$ in the optimal solution to problem (\ref{eqn:path_fluid}). On the other hand, to characterize the total expected revenue and capacity consumption from product $j$ due to price path $\pvecbar_j^q$ in the optimal solution to problem (\ref{eqn:path_fluid}), we also define  
\begin{align}
\reve_j^q = \sum_{t \in \Tcal} \sum_{\ell \in \Ncal} \ind{\pbar_{jt}^q = \ell} \ts r_j^\ell \ts \lambda_{jt}^\ell \ts \frac{\xbar_{jt}^\ell}{\ybar_{jt}^\ell},
\qquad
\capa_j^q = \sum_{t \in \Tcal} \sum_{\ell \in \Ncal} \ind{\pbar_{jt}^q = \ell} \ts \lambda_{jt}^\ell \ts \frac{\xbar_{jt}^\ell}{\ybar_{jt}^\ell}.
\label{eqn:path_rev}
\end{align}

\vspace{-0.5mm}

We can recover the total expected revenue and capacity consumption from product~$j$ by using the corresponding quantities from the different price paths. In particular, we have the identities $\Rbar_j = \sum_{q \in \Mcal} \zbar_j^q \ts \reve_j^q$ and $\Cbar_j = \sum_{q \in \Mcal} \zbar_j^q \ts \capa_j^q$. To see that the first identity holds, by (\ref{eqn:path_rev}), we have $\sum_{q \in \Mcal} \zbar_j^q \ts \reve_j^q = \sum_{t \in \Tcal} \sum_{\ell \in \Ncal}  r_j^\ell \ts \lambda_{jt}^\ell \ts \frac{\xbar_{jt}^\ell}{\ybar_{jt}^\ell} \ts \sum_{q \in \Mcal} \ind{\pbar_{jt}^q = \ell} \ts \ts \zbar_j^q = \sum_{t \in \Tcal} \sum_{\ell \in \Ncal} r_j^\ell \ts \lambda_{jt}^\ell \ts \xbar_{jt}^\ell = \Rbar_j$, where the second equality holds because $(\xvecbar,\yvecbar,\zvecbar)$ satisfies the third constraint in problem (\ref{eqn:path_fluid}). We can follow the same argument to see that the second identity holds. Viewing $\reve_j^q$ as the total expected revenue from product $j$ due to price path $\pvecbar_j^q$, we interpret this quantity as follows. Following the price path $\pvecbar_j^q = (\pbar_{jt}^q :t \in \Tcal)$, we charge price level $\ell$ for product $j$ at time period $t$ if and only if $\pbar_{jt}^q  = \ell$. If we decide to charge price level $\ell$ for product $j$ at time period $t$, then we make the product available for purchase with probability $\xbar_{jt}^\ell / \ybar_{jt}^\ell$. If we make the product available, then we make a sale for the product with probability $\lambda_{jt}^\ell$, in which case, we obtain a revenue of $r_j^\ell$. We do not make the product available for purchase with probability $1 - \xbar_{jt}^\ell / \ybar_{jt}^\ell$. By the second constraint in (\ref{eqn:path_fluid}), we have $\xbar_{jt}^\ell / \ybar_{jt}^\ell \leq 1$, so we can indeed use the quantity $\xbar_{jt}^\ell / \ybar_{jt}^\ell \leq 1$ as a probability. We proceed to giving a specification of our approximate policy. 

{\bf \underline{Specification of the Approximate Policy}:}
\\
\indent
For each product $j$, recalling that $\zbar_j^q = 0$ for all $q \in \Mcal \setminus \{\Fs,\Ss\}$, we index the two price paths $\Fs, \Ss$ such that $\reve_j^\Fs / \max\{ \capa_j^\Fs , \Cbar_j\}  \geq \reve_j^\Ss / \max\{ \capa_j^\Ss , \Cbar_j\}$. In our approximate policy, we follow the prices in the price path $\pvecbar_j^\Fs$ for product $j$. In particular, we set $\thetabar_{jt} = \sum_{\ell \in \Ncal} \ind{\pbar_{jt}^\Fs = \ell} \frac{\xbar_{jt}^\ell}{\ybar_{jt}^\ell}$, which is simply the value of the ratios $\{ \xbar_{jt}^\ell/ \ybar_{jt}^\ell : \ell \in \Ncal\}$ corresponding to the price charged in price path $\pvecbar_j^\Fs$ for product~$j$ at time period $t$. Letting $\gamma \in (0,1)$ be a tuning parameter, our approximate policy makes its decisions as follows. At time period $t$, if we do not have enough remaining capacity for the resources to make product~$j$ available, then we do not make product $j$ available for purchase. Otherwise, we make product $j$ available with probability $\gamma \ts \thetabar_{jt} \ts \Cbar_j / \max\{ \capa_j^\Fs , \Cbar_j\}$. If we make product $j$ available, then we charge the price $\pbar_{jt}^\Fs$. The product availability decisions at different time periods are independent of each other. We will specify the choice of the tuning parameter, which will be arbitrarily close to one as the resource capacities get large.

\vspace{-0.5mm}

The choice of the price path to follow for product $j$ is driven by the ratio $\reve_j^\Fs / \max\{ \capa_j^\Fs , \Cbar_j\}$, which we intuitively interpret as the total expected revenue from product~$j$ on price path $\pvecbar_j^\Fs$ for each unit of total expected capacity consumption. We ideally would like to make product $j$ available for purchase at time period $t$ with probability $\thetabar_{jt}$, but if the total expected capacity consumption $\capa_j^\Fs$ of product~$j$ due to price path $\pvecbar_j^\Fs$ exceeds the total expected capacity consumption $\Cbar_j$ of product~$j$, then we dial down the availability probability. In the next theorem, letting $\apx$ be the total expected revenue of the approximate policy, as well as using $L = \max_{j \in \Jcal} \sum_{i \in \Lcal} a_{ij}$ for the maximum number of resources used by a product and $c_{\min} = \min_{i \in \Lcal} c_i$ for the minimum resource capacity, we give a performance guarantee for the approximate policy.

\vspace{-2mm}

\begin{thm}
[Performance Guarantee]
\label{thm:perf}
There exists a choice of the tuning parameter to ensure that the total expected revenue of the approximate policy satisfies
\begin{align*}
\frac{\apx}{\opt} ~\geq~ \frac{\apx}{Z_\lp^*} ~\geq~
\max \Bigg\{ \frac{1}{8L} \ts , \ts \frac{1}{2} - \sqrt{\frac{\log c_{\min}}{2 \ts c_{\min}}} - \frac{L}{c_{\min}} \Bigg\}.
\end{align*}
\end{thm}

\vspace{-1.7mm}

We give the proof in Appendix \ref{sec:perf}. In practice, the number of resources may be large, but the number of resources used by a product is usually uniformly bounded. In airline revenue management applications, for example, the resources correspond to flight legs and the products correspond to itineraries. While the number of flight legs in an airline network can reach hundreds, the number of flight legs in an itinerary rarely exceeds two, so $L=2$. When $L$ is uniformly bounded, the first expression in the maximum operator in the theorem provides a constant-factor performance guarantee. As the resource capacities get large, the second expression in the maximum operator in the theorem becomes arbitrarily close to $\frac 12$, so our approximate policy provides a performance guarantee arbitrarily close to $\frac 12$ under large resource capacities.

\vspace{-2mm}

%In the proof of Theorem \ref{thm:perf}, we establish the second inequality, in which case, the first inequality follows by Proposition \ref{pro:ub}. 
{\bf \underline{Tightness of the Performance Guarantee}:}
\\
\indent The performance guarantee of $\frac 12$ in Theorem \ref{thm:perf} is tight in the sense that we can give a problem instance such that the optimal objective value of the \ref{eqn:fluid} exceeds the optimal total expected revenue by a factor arbitrarily close to two as the resource capacities get large. In particular, we consider a problem instance with a single resource and a single product, each indexed by $\Lcal = \{1\}$ and $\Jcal = \{1\}$. The capacity of the resource is $c_1 = C$. There are two possible price levels for the product with the associated revenues $r_1^1 = \frac 1C $ and $r_1^2= 1$. There are $C^2$ time periods in the selling horizon. If we charge the first price level, then we make a sale for the product at all time periods with probability one, so $\lambda_{1t}^1 = 1$ for all $t=1,\ldots,C^2$. If we charge the second price level, then we make a sale for the product with probability one at the first time period and with probability zero at other time periods, which is to say that we have $\lambda_{11}^2 = 1$ and $\lambda_{1t}^2 = 0$ for all $t=2,\ldots,C^2$. We proceed to computing the optimal total expected revenue.

Because we are constrained to follow a monotone price path, if we charge the second price level at the first time period, then we cannot switch to the first price level at a later time period. Thus, noting that $r_1^2=1$ and $\lambda_{11}^2=1$, if we charge the second price level at the first time period, then the optimal total expected revenue is one. On the other hand, because $\lambda_{1t}^2 = 0$ for $t \neq 1$, if we charge the second price level at any time period other than the first time period, then we do not make a sale. Therefore, if we charge the first price level at the first time period, then there is no reason to switch to the second price level at a later time period. In this case, noting that $c_1 = C$, $r_1^1 = \frac 1C$ and $\lambda_{1t}^1 = 1$ for all $t = 1,\ldots,C^2$, if we charge the first price level at the first time period, then the optimal total expected revenue is $C \times \frac 1C = 1$. Thus, it follows that we have $\opt = 1$.

We construct the solution $(\xvechat,\yvechat)$ to the \ref{eqn:fluid} as $\xhat_{1t}^1 = \yhat_{1t}^1 = \frac{1}{1+C}$ and \mbox{$\xhat_{1t}^2 = \yhat_{1t}^2 = \frac{C}{1+C}$} for all $t=1,\ldots,C^2$. We verify that this solution is feasible to the \ref{eqn:fluid}. Noting the value of $\lambda_{1t}^\ell$ for $\ell =1,2$ and $t=1,\ldots,C^2$, we have $\sum_{t=1}^{C^2} \sum_{\ell=1}^2 \lambda_{1t}^\ell \ts \xhat_{1t}^\ell= \frac{1}{1+C} + \frac{C}{1+C} + (C^2-1) \ts \frac{1}{1+C} = C = c_1$, so the first constraint in the \ref{eqn:fluid} holds. The second and third constraints in the \ref{eqn:fluid} immediately hold because $\xhat_{1t}^\ell = \yhat_{1t}^\ell$ and $\yhat_{1,t-1}^\ell = \yhat_{1t}^\ell$. Finally, the fourth constraint in the \ref{eqn:fluid} follows as $\yhat_{1t}^1 + \yhat_{1t}^2 = \frac{1}{1+C} + \frac{C}{1+C} = 1$. The objective value of the \ref{eqn:fluid} at the solution  $(\xvechat,\yvechat)$ is $\sum_{t=1}^{C^2} \sum_{\ell =1}^2 r_1^\ell \ts \lambda_{1t}^\ell \ts \xhat_{1t}^\ell = \frac{1}{C} \ts \frac{1}{1+C} + \frac{C}{1+C} + (C^2-1) \ts \frac{1}{C} \ts \frac{1}{1+C} = \frac{2C}{1+C}$, so we get $Z_\lp^* \geq \frac{2C}{1+C}$. Therefore, we have $\frac{Z_\lp^*}{\opt} \geq \frac{2C}{C+1}$, so the right side of the inequality is arbitrarily close to two as the capacity of the resource gets large, as desired. Using the last inequality along with the performance guarantee in Theorem \ref{thm:perf}, we have $\frac{1}{2} - \sqrt{\frac{\log C}{2 \ts C}} - \frac{1}{C} \leq \frac{\apx}{Z_\lp^*} \leq \frac{\opt}{Z_\lp^*} \leq \frac{C+1}{2C}$, which implies that both the approximate and optimal policies obtain half of the optimal objective value of the \ref{eqn:fluid} as the resource capacities get large. In other words, even if we consider the optimal policy, we only obtain at most half of the optimal objective value of the \ref{eqn:fluid}, so the approximate policy is the best we can hope for in this sense. In the next section, nevertheless, we slightly modify our approximate policy to ensure that the approximate policy obtains the full optimal objective value of the \ref{eqn:fluid} as the resource capacities get large under an additional assumption on the optimal solution to the \ref{eqn:fluid}.

\section{Approximate Policy with Ex-Post Performance Guarantee}
\label{sec:expost_policy}

The performance guarantee for our approximate policy in Theorem \ref{thm:perf} depends only on the problem primitives $L$ and $c_{\min}$ and it gets arbitrarily close to $\frac 12$ as the capacities of the resources get large, even though the practical performance of our approximate policy is substantially better than what is indicated by this performance guarantee. In this section, we give an approximate policy with a performance guarantee that depends on the optimal solution to the \ref{eqn:fluid}, so we can calculate the performance guarantee of the approximate policy after we solve the \mbox{\ref{eqn:fluid}}. We refer to this approximate policy as the ex-post approximate policy, because we can calculate its  performance guarantee after solving the \ref{eqn:fluid}. The performance guarantee of the \mbox{ex-post approximate policy} gets arbitrarily close to one as the capacities of the resources get large, as long as the optimal solution to the \ref{eqn:fluid} satisfies certain properties that we will make precise.
We give a specification of the ex-post approximate policy. Letting $(\xvecbar,\yvecbar,\zvecbar)$ be an extreme point optimal solution to problem (\ref{eqn:path_fluid}), for each product $j$, recall that there are two price paths $\Fs,\Ss \in \Mcal$ such that \mbox{$\zbar_j^\Fs \geq 0$}, $\zbar_j^\Ss \geq 0$ and $\zbar_j^q = 0$ for all $q \in \Mcal \setminus \{ \Fs, \Ss\}$. We continue using $\Rbar_j = \sum_{t \in \Tcal} \sum_{\ell \in \Ncal} r_j^\ell \ts \lambda_{jt}^\ell \ts \xbar_{jt}^\ell$ and \mbox{$\Cbar_j = \sum_{t \in \Tcal} \sum_{\ell \in \Ncal} \lambda_{jt}^\ell \ts \xbar_{jt}^\ell$}, as well as $\reve_j^q$ and $\capa_j^q$ as defined in (\ref{eqn:path_rev}). By the discussion just after (\ref{eqn:path_rev}), we have $\Rbar_j = \sum_{q \in \Mcal} \zbar_j^q \ts \reve_j^q$ and $\Cbar_j = \sum_{q \in \Mcal} \zbar_j^q \ts \capa_j^q$.

{\bf \underline{Specification of the Ex-Post Approximate Policy}:}
\\
\indent
We index the two price paths $\Fs, \Ss$ such that $\reve_j^\Fs  \geq \reve_j^\Ss$. In our approximate policy, we follow the prices in the price path $\pvecbar_j^\Fs$ for product $j$. We set $\thetabarp_{jt} = \sum_{\ell \in \Ncal} \ind{\pbar_{jt}^\Fs = \ell} \frac{\xbar_{jt}^\ell}{\ybar_{jt}^\ell}$ to capture the value of the ratios $\{ \xbar_{jt}^\ell/ \ybar_{jt}^\ell : \ell \in \Ncal\}$ corresponding to the price charged in price path $\pvecbar_j^\Fs$ for product~$j$ at time period $t$. Using $\gamma \in (0,1)$ to denote a tuning parameter, our ex-post approximate policy makes its decisions as follows. At time period $t$, if we do not have enough remaining resource capacities to make product $j$ available, then we do not make product $j$ available for purchase. Otherwise, using $[a]^+ = \max\{a,0\}$, letting $\Deltabar = \max_{i \in \Lcal} \sum_{j \in \Jcal} a_{ij} [\capa_j^\Fs - \capa_j^\Ss]^+$, we make product $j$ available with probability $\gamma \ts \thetabarp_{jt} \ts c_{\min} / (c_{\min} + \Deltabar)$. If we make product $j$ available, then we charge the price $\pbar_{jt}^\Fs$. The ex-post approximate policy  may choose a different price path to follow than the original approximate policy. Also, the two approximate policies may make a product available with different probabilities. For $a, b \in \Mcal$ with $\zbar_j^a > 0$ and $\zbar_j^b> 0$, because $\Cbar_j = \zbar_j^a  \ts \capa_j^a + \zbar_j^b \ts \capa_j^b$, if $\capa_j^a \geq \capa_j^b$, then we have  $\capa_j^a \geq \Cbar_j \geq \capa_j^b$. Similarly, if $\reve_j^a \geq \reve_j^b$, then we have  $\reve_j^a \geq \Rbar_j \geq \reve_j^b$. It is possible to have $\capa_j^a \geq \Cbar_j \geq \capa_j^b$, $\reve_j^a \geq \Rbar_j \geq \reve_j^b$ and $\reve_j^a / \capa_j^a \leq \reve_j^b / \Cbar_j$, in which case, the ex-post and original approximate policies would follow different price paths. Despite our best efforts, we do not have a unified policy that achieves the performance guarantees of both approximate policies.

%If, however, $\capa_j^\Fs \leq \capa_j^\Ss$ for all $j \in \Jcal$, then the ex-post policy makes product $j$ available at time period $t$ with probability $\gamma \ts \thetabar_{jt}$. Noting that $\Cbar_j = \zbar_j^\Fs \ts \capa_j^\Fs + \zbar_j^\Ss \ts \capa_j^\Ss$, if $\capa_j^\Fs \leq \capa_j^\Ss$ for all $j \in \Jcal$, then we have $\capa_j^\Fs \leq \Cbar_j$, in which case, we have $\max\{ \Cbar_j , \capa_j^\Fs \} = \Cbar_j$, so the original approximate policy also makes product $j$ available at time period $t$ with probability $\gamma \ts \thetabar_{jt}$. Thus, if $\capa_j^\Fs \leq \capa_j^\Ss$ for all $j \in \Jcal$, then the products are available with the same probabilities.

In the next theorem, using $\apxp$ to denote the total expected revenue of the ex-post approximate policy, we give a performance guarantee. Comparing the performance guarantee in  Theorem \ref{thm:expost} with the one in Theorem \ref{thm:perf}, the performance guarantee below depends on $\Deltabar$, which, in turn, depends on the optimal solution to the \ref{eqn:fluid}. The proof is in Appendix \ref{sec:expost}.

\vspace{-1.5mm}

\begin{thm}
[Ex-Post Approximate Policy]
\label{thm:expost}
There exists a choice of the tuning parameter to ensure that the total expected revenue of the ex-post approximate policy satisfies
\begin{align*}
\frac{\apxp}{\opt} ~\geq~ \frac{\apxp}{Z_\lp^*} ~\geq~
1 - \sqrt{\frac{2\log c_{\min}}{ c_{\min}}} - \frac{L + \Deltabar}{c_{\min}}.
\end{align*}
\end{thm}

\vspace{-1.5mm}

Thus, as the capacities of the resources get large, if the difference between the total expected capacity consumptions of the two price paths for each product remains bounded, then the ex-post approximate policy is asymptotically optimal. While the proofs of Theorems~\ref{thm:perf}  and \ref{thm:expost} have similar outlines, there is a critical difference in the result that we achieve in the two theorems. Using $\availp(\gamma)$ to denote a lower bound on the probability that all of the resources used by a product are available at a time period under the ex-post approximate policy, we can establish that the total expected revenue of the ex-post approximate policy satisfies \mbox{$\apxp \geq \gamma \ts \availp(\gamma) \ts Z_\lp^* \ts c_{\min} / (c_{\min} + \Deltabar)$}. On the other hand, using $\avail(\gamma)$ to denote a lower bound on the probability that all of the resources used by a product are available at a time period under the original approximate policy, if we do not allow dependence on the quantity $\Deltabar$ when lower bounding the total expected revenue of the policy, then we can only establish that the total expected revenue of the original approximate policy satisfies $\apx \geq \frac 12 \ts \gamma \ts \avail(\gamma) \ts Z_\lp^*$. Note the half on the right side of the last inequality.

{\bf \underline{Problem Instance Yielding Tight Performance Guarantee}:}
\\
\indent In Section \ref{sec:policy}, we give a problem instance such that the optimal objective value of the \ref{eqn:fluid} exceeds the optimal total expected revenue by a factor arbitrarily close to two as the resource capacities get large. Thus, by comparing the total expected revenue of an approximate policy with the optimal objective value of the \ref{eqn:fluid}, we cannot establish that any approximate policy is asymptotically optimal as the resource capacities get large. We reconcile this problem instance with the performance guarantee in Theorem \ref{thm:expost}. Noting that we have \mbox{$1 \geq \frac{\opt}{Z_\lp^*} \geq \frac{\apxp}{Z_\lp^*} \geq 1 - \sqrt{\frac{2\log c_{\min}}{ c_{\min}}} - \frac{L + \Deltabar}{c_{\min}}$} by Theorem \ref{thm:expost}, 
if the value of $(L+\Deltabar)/c_{\min}$ for this problem instance got arbitrarily close to zero as the resource capacities get large, then the optimal objective value of the \ref{eqn:fluid} would be arbitrarily close to the optimal total expected revenue as the resource capacities get large. We demonstrate that the value of $(L+\Deltabar)/c_{\min}$ does not diminish as the resource capacities get large for this problem instance. Recall that there is one resource with a capacity of $C$ and one product with two price levels. There are $C^2$ time periods. 

In Appendix \ref{sec:unbounded_diff}, we show that the optimal objective value of the \ref{eqn:fluid} for our problem instance is $Z_\lp^* = \frac{2\ts C}{1+C}$ with the corresponding optimal solution $(\xvecbar,\yvecbar)$ with $\xbar_{1t}^1 = \ybar_{1t}^1 = \frac{1}{1+C}$ and \mbox{$\xbar_{1t}^2 = \ybar_{1t}^2 = \frac{C}{1+C}$} for all $t=1,\ldots,C^2$. Noting that we have $\sum_{k=1}^\ell \ybar_{1t}^\ell \in \{\frac{1}{1+C} , 1\}$ for all $\ell=1,2$ and $t =1,\ldots,C^2$, following the construction at the beginning of Section \ref{sec:paths}, we have two possible realizations of the random price path $\Prmvecbar_1$. The first path charges the first price level at all time periods, whereas the second price path charges the second price level at all time periods. Setting $\zbar_1^1 = \frac{1}{1+C}$ and $\zbar_1^2 = \frac{C}{1+C}$, as well as $\xbar_{1t}^1 = \ybar_{1t}^1 = \frac{1}{1+C}$ and \mbox{$\xbar_{1t}^2 = \ybar_{1t}^2 = \frac{C}{1+C}$} for all $t=1,\ldots,C^2$, the solution $(\xvecbar,\yvecbar,\zvecbar)$ is feasible to problem (\ref{eqn:path_fluid}) and yields an objective value of $\frac{2\ts C}{1+C}$, which is equal to the optimal objective value of the \ref{eqn:fluid}. Because the optimal objective value of problem (\ref{eqn:path_fluid}) is equal to that of the \ref{eqn:fluid}, this solution is optimal to problem (\ref{eqn:path_fluid}). The first price path charges the first price level at all time periods with $r_1^1 = \frac{1}{C}$ and $\lambda_{1t}^1 =1$, so we have $\reve_1^1 = \frac{1}{C} \ts C^2 = C$ and $\capa_1^1 = C^2$. Similarly, the second price path charges the second price level at all time periods with $r_1^2 = 1$, $\lambda_{1t}^2=1$ for $t=1$ and $\lambda_{1t}^2 =0$ for $t \neq 1$, so we have $\reve_1^2 = 1$ and $\capa_1^2 = 1$. Because $\reve_1^1 \geq \reve_1^2$, we have $\Fs = 1$ and $\Ss = 2$, so $\Deltabar = \capa_1^1 - \capa_1^2 = C^2 - 1$. We have $\frac{1}{c_{\min}} \ts (L+\Deltabar)  = \frac{1}{C} \ts (1 + C^2 - 1) = C$, which does not diminish as the resource capacities get large.

\section{Pricing Under Promotion Fatigue Constraints}
\label{sec:fatigue}

Part of our notation follows the one under price monotonicity constraints. The set of resources is~$\Lcal$.~The initial capacity of resource $i$ is $c_i$. The set of products is $\Jcal$. To capture the resources used by product $j$, we use the vector $\avec_j = (a_{ij} : i \in \Lcal) \in \{0,1\}^{|\Lcal|}$, where $a_{ij} = 1$ if and only if product $j $ uses resource $i$. We can offer a product at a promoted and a regular price level.~We use $\Ncal = \{\dis ,\reg \}$ to capture the promoted and regular price levels. If we charge price level~$\ell$ for product~$j$ and make a sale for this product, then we obtain a revenue of $r_j^\ell$. We expect to have $r_j^\dis \leq r_j^\reg$, so that the revenue corresponding to the promoted price level is smaller than that corresponding to the regular price level, but our results do not rely on this assumption. The set of time periods in the selling horizon is $\Tcal = \{1,\ldots,T\}$. There is at most one customer arrival at each time period. If we charge price level $\ell$ for product $j$ at time period $t$, then we make a sale for the product with probability $\lambda_{jt}^\ell$.~At each time period, we decide which products to make available for purchase and what prices to charge for the offered products. We can promote the products infrequently. Over any interval of $K$ time periods, we can offer a product at the promoted price at most once.

To capture the state of the system at the beginning of a generic time period, we use the vector $\wvec = (w_i : i \in \Lcal) \in \mathbb Z_+^{|\Lcal|}$, where $w_i$ is the remaining capacity of resource $i$. Furthermore, we use the vector $\zvec =  (z_j : j \in \Jcal) \in \mathbb Z_+^{|\Jcal|}$, where $z_j$ is the number of time periods elapsed since we offered product $j$ at the promoted price. We use the pair $(\wvec,\zvec)$ as the state. For the decisions at a generic time period, we use the vector $\xvec = ( x_j^\ell : j \in \Jcal,~\ell \in \Ncal)$, where $x_j^\ell = 1$ if and only if we charge price level $\ell$ for product $j$. Also, we use the vector $\yvec = (y_j : j \in \Jcal) \in \mathbb Z_+^{|\Jcal|}$, where $y_j$ is the number of time periods elapsed since promoting product $j$ after the pricing decisions at the current time period. We use the pair $(\xvec,\yvec)$ as the decisions. For the state $(\wvec,\zvec)$, the set of feasible decisions is 

\vspace{-8.75mm}

\begin{align}
\Fcal_t(\wvec,\zvec) = \Bigg\{& (\xvec,\yvec) \in \{0,1\}^{2 \times |\Jcal|} \times \mathbb Z_+^{|\Jcal|} : \sum_{j \in \Jcal} \sum_{\ell \in \Ncal} a_{ij} \ts \lambda_{jt}^\ell \ts x_j^\ell \leq w_i~~\forall \ts i \in \Lcal,~~~\sum_{\ell \in \Ncal} x_j^\ell \leq 1~~\forall \ts j \in \Jcal,
\nonumber
\\
&
x_j^\dis \leq \ind{z_j \geq K}~~\forall \ts j \in \Jcal,~~~ y_j = x_j^\dis + (1 - x_j^\dis) \ts (z_j + 1)~~\forall \ts j \in \Jcal
 \Bigg\}.
 \label{eqn:feas_fatigue}
\end{align}

\vspace{-3.75mm}

\indent The first two constraints in (\ref{eqn:feas_fatigue}) are identical to the first two constraints in (\ref{eqn:feas_monotone}). In the third constraint, if the number of time periods elapsed since we offered product $j$ at the promoted price is less than $K$, then we cannot offer product $j$ at the promoted price. In the fourth constraint, if we offer product $j$ at the promoted price, then we reset the number of time periods elapsed since we offered product $j$ at the promoted price to one. Otherwise, we increment the number of time periods elapsed since we offered product $j$ at the promoted price by one. In the initial state of the system, we use $\cvec = (c_i : i \in \Lcal)$ to capture the initial resource capacities. We proceed with the assumption that we can offer a product at the promoted price at the first time period, but if we do so, then we cannot offer the product at the promoted price until time period $K+1$. Thus, letting $\evec = (e_j : j \in \Jcal) \in \mathbb Z_+^{|\Jcal|}$ be the vector of all ones, the initial state is $(\cvec , K \evec)$. We can compute the optimal policy by using the same dynamic program in (\ref{eqn:dp}), but the state vector $(\wvec,\zvec)$, decision vector $(\xvec,\yvec)$ and set of feasible decisions $\Fcal_t(\wvec,\zvec)$ are as described earlier in this section. In this case, the optimal total expected revenue is given by  $\opt = J_1(\cvec , K \evec)$. We give an approximate policy with a performance guarantee by using a fluid approximation. 

\vspace{-1mm}

{\bf \underline{Fluid Approximation Under Promotion Fatigue Constraints}:}
\\
\indent We give a fluid approximation for pricing under promotion fatigue constraints. Following our fluid approximation, we describe our approach for sampling price paths according to an optimal solution to the fluid approximation in such a way that each sampled price path satisfies promotion fatigue constraints. Finally, we construct an approximate policy from the sampled price paths and give a performance guarantee for our approximate policy. Our approach for sampling price paths according to an optimal solution to the fluid approximation under promotion fatigue constraints differs from that under price monotonicity constraints. Once we sample price paths, however, our construction of the approximate policy, as well as the proof of the performance guarantee, follows from an outline similar to the one under price monotonicity constraints. To give our fluid approximation under promotion fatigue, we use the decision variable $x_{jt}^\ell$ to capture the probability that we charge price level $\ell$ for product $j$ at time period $t$ and make the product available. Furthermore, we use the decision variable $y_{jt}^\ell$ to capture the probability that we charge price level~$\ell$ for product $j$ at time period $t$. Noting the definitions of the two decision variables, the probability that we do not make product $j$ available at time period $t$ is given by $\sum_{\ell \in \Ncal} y_{jt}^\ell - \sum_{\ell \in \Ncal} x_{jt}^\ell$. Using the vectors of decision variables $\xvec = (x_{jt}^\ell: j \in \Jcal,~\ell \in \Ncal,~t \in \Tcal)$  and $\yvec = (y_{jt}^\ell: j \in \Jcal,~\ell \in \Ncal,~t \in \Tcal)$, setting $ a \wedge b = \min\{a,b\}$, we consider the fluid approximation given by 
\begin{align}
\max_{(\xvec,\yvec) \in [0,1]^{4 \times |\Jcal| \times T}}
\Bigg\{ 
&\sum_{t \in \Tcal} \sum_{j \in \Jcal} \sum_{\ell \in \Ncal} r_j^\ell \ts \lambda_{jt}^\ell \ts x_{jt}^\ell ~:~
\sum_{t \in \Tcal} \sum_{j \in \Jcal} \sum_{\ell \in \Ncal} a_{ij} \ts \lambda_{jt}^\ell \ts x_{jt}^\ell \leq c_i \qquad \forall \ts i \in \Lcal,
\nonumber
\\
&\qquad x_{jt}^\ell \leq y_{jt}^\ell \qquad \forall \ts j \in \Jcal,~\ell \in \Ncal,~t \in \Tcal,
\phantom{\sum}
\nonumber
\\
& \qquad
\!\!\!\!\!\!\!\!\! \sum_{\tau = t}^{(t+K-1) \wedge T} \!\!\!\!\!\!\!\! y_{j\tau}^\dis \leq 1 \qquad \forall \ts j \in \Jcal,~t \in \Tcal ,
\nonumber
\\
& \qquad \sum_{\ell \in \Ncal} \ts y_{jt}^\ell = 1 \qquad \forall \ts j \in \Jcal,~t \in \Tcal \Bigg\}.
\label{eqn:fluid_fatigue}
\end{align}

In the third constraint, we ensure that the total expected number of times that we offer product $j$ at the promoted price over any duration of $K$ time periods is at most one. In the fluid approximation, we enforce  promotion fatigue constraints in expectation, but we will construct an approximate policy that imposes promotion fatigue constraints with probability one. The other constraints above are identical to their counterparts in our earlier fluid approximation. We can show that the optimal objective value of the fluid approximation in (\ref{eqn:fluid_fatigue}) is an upper bound on the optimal total expected revenue under promotion fatigue constraints. This result is the analogue of Proposition \ref{pro:ub} under promotion fatigue. The proof is based on using the decisions of the optimal policy to construct a feasible solution to (\ref{eqn:fluid_fatigue}). We turn to using an optimal solution to the fluid approximation in (\ref{eqn:fluid_fatigue}) to sample price paths such that each of the price paths satisfies promotion fatigue constraints. We sample the price paths such that the probability that we charge price level~$\ell$ for product $j$ at time period $t$ also matches the optimal solution to the fluid approximation.

{\bf \underline{Sampling Price Paths from the Fluid Approximation}:}
\\
\indent
We use $(\xvecbar,\yvecbar)$ to denote an optimal solution to problem (\ref{eqn:fluid_fatigue}). Setting $\sbar_{jt} = \sum_{\tau = 1}^t \ybar_{j\tau}^\dis$ with the convention that $\sbar_{j0} = 0$, for each product $j$ and time period $t$, we define a subset of the interval $[0,1)$ as follows. By the fourth constraint in (\ref{eqn:fluid_fatigue}), we have $\sbar_{jt} - \sbar_{j,t-1} = \ybar_{jt}^\dis \leq 1$. Therefore, the interval $(\sbar_{j,t-1} , \sbar_{jt}]$ can include at most one integer. If the interval $(\sbar_{j,t-1} , \sbar_{jt}]$ does not include any integers so that $k \leq  \sbar_{j,t-1} \leq \sbar_{jt} < k+1$ for some integer $k$, then we define $\Ucal_{jt} = [ \sbar_{j,t-1} - k , \sbar_{jt} - k)$. By the last chain of inequalities, we have $0 \leq \sbar_{j,t-1} - k \leq \sbar_{jt} - k < 1$, so we get $\Ucal_{jt} \subseteq [0,1)$. On the other hand, if the interval $(\sbar_{j,t-1} , \sbar_{jt}]$ contains an integer so that $\sbar_{j,t-1} < k \leq \sbar_{jt}$ for some integer $k$, then we define $\Ucal_{jt} =[0 , \sbar_{jt} - k) \cup \ts[\sbar_{j,t-1} - k + 1 , 1)$. By the last chain of inequalities, as well as the fact that $\sbar_{jt} - \sbar_{j,t-1} \leq 1$, we have $0 \leq \sbar_{jt} - k \leq \sbar_{j,t-1} - k +1 < 1$, so we get $\Ucal_{jt} \subseteq [0,1)$. Also, because we have $\sbar_{jt} - \sbar_{j,t-1} \leq 1$, note that the intervals $[0 , \sbar_{jt} - k)$ and $ [\sbar_{j,t-1} - k + 1 , 1)$ are disjoint. Letting $\text{\sf Unif}$ be the uniform random variable over the interval $[0,1)$, for product $j$, we define the random price path $\Prmvecbar_j = (\Prmbar_{jt} : t \in \Tcal) \in \{ \dis , \reg\}^T$ such that $\Prmbar_{jt} = \dis$ if $\text{\sf Unif} \in \Ucal_{jt}$, whereas $\Prmbar_{jt} = \reg$ if $\text{\sf Unif} \not \in \Ucal_{jt}$. In the next lemma, we show that the random price path satisfies promotion fatigue constraints with probability one and characterize the marginal distribution of the price levels. 

\vspace{-1mm}

\begin{lem}[Feasible Paths] Considering the random price path $\Prmvecbar_j$, for time periods $ t,\kappa \in \Tcal$ that satisfy $t < \kappa \leq t + K - 1$, we have $\ind{\Prmbar_{jt} = \dis} + \ind{\Prmbar_{j\kappa} = \dis} \leq 1$. Also, we have $\mathbb P \{ \Prmbar_{jt} = \dis \} = \ybar_{jt}^\dis$. 

\end{lem}

\vspace{-2mm}

\noindent{\it Proof:} To show that $\ind{\Prmbar_{jt} = \dis} + \ind{\Prmbar_{j\kappa} = \dis} \leq 1$, we check that the sets $\Ucal_{jt}$ and $\Ucal_{j\kappa}$ are disjoint. Because $t < \kappa \leq t + K - 1$, by the third constraint in (\ref{eqn:fluid_fatigue}), we have $\sbar_{j\kappa} - \sbar_{j,t-1} = \sum_{\tau = t}^\kappa \ybar_{j\tau}^\dis \leq 1$, so the interval $(\sbar_{j,t-1} , \sbar_{j\kappa}]$ can have at most one integer. First, assume that the interval  $(\sbar_{j,t-1} , \sbar_{j\kappa}]$ does not have an integer, so $k \leq \sbar_{j,t-1} \leq \sbar_{j\kappa} < k+1$ for some $k \in \mathbb Z_+$. Thus, we have $\Ucal_{jt} = [ \sbar_{j,t-1} - k , \sbar_{jt} - k)$~and $\Ucal_{j\kappa} = [ \sbar_{j,\kappa-1} - k , \sbar_{j\kappa} - k)$. Because $\kappa-1 \geq t$, we have $\sbar_{j,\kappa-1} = \sum_{\tau = 1}^{\kappa - 1} \ybar_{j\tau}^\dis \geq \sum_{\tau = 1}^t \ybar_{j\tau}^\dis  = \sbar_{jt}$, which~yields  $\sbar_{jt} - k \leq \sbar_{j,\kappa-1} - k$, so $\Ucal_{jt} \cap \ts \Ucal_{j\kappa} = \varnothing $. Second, assume that the interval $(\sbar_{j,t-1} , \sbar_{j\kappa}]$ has an integer, so $\sbar_{j,t-1} < k \leq \sbar_{j\kappa}$ for some $k \in \mathbb Z_+$. In the first case, assume that $\sbar_{j,t-1} < k \leq \sbar_{jt} \leq \sbar_{j,\kappa-1} \leq \sbar_{j\kappa}$. Thus, we have  $\Ucal_{jt} =[0 , \sbar_{jt} - k) \cup \ts[\sbar_{j,t-1} - k + 1 , 1)$ and $\Ucal_{j\kappa} = [ \sbar_{j,\kappa-1} - k , \sbar_{j\kappa} - k)$. Noting that $\sbar_{jt}  \leq \sbar_{j,\kappa-1}$ and $\sbar_{j\kappa} - \sbar_{j,t-1} \leq 1$, we have  $ \sbar_{jt} - k \leq \sbar_{j,\kappa-1} - k \leq \sbar_{j,\kappa} - k \leq \sbar_{j,t-1} - k +1$, so  $\Ucal_{jt} \cap \Ucal_{j\kappa} = \varnothing$. In the second and third cases, each of which corresponds to  $\sbar_{j,t-1} \leq \sbar_{jt} < k \leq \sbar_{j,\kappa-1} \leq \sbar_{j\kappa}$ and \mbox{$\sbar_{j,t-1} \leq \sbar_{jt} \leq \sbar_{j,\kappa-1}< k \leq \sbar_{j\kappa}$}, we can use the same argument to show that $\Ucal_{jt} \cap \Ucal_{j\kappa} = \varnothing$. We turn to showing that $\mathbb P \{ \Prmbar_{jt} = \dis \} = \ybar_{jt}^\dis$. If  $\Ucal_{jt} = [ \sbar_{j,t-1} - k , \sbar_{jt} - k)$, then the length of the interval is $\sbar_{jt} - \sbar_{j,t-1}$. If $\Ucal_{jt} =[0 , \sbar_{jt} - k) \cup \ts[\sbar_{j,t-1} - k + 1 , 1)$, then the total length of the two intervals is also $\sbar_{jt} - \sbar_{j,t-1}$. In either case, we get $\mathbb P \{ \Prmbar_{jt} = \dis \} = \mathbb P \{ \text{\sf Unif} \in \Ucal_{jt}\} = \sbar_{jt} - \sbar_{j,t-1} = \ybar_{jt}^\dis$. \qed

By the lemma above, the marginal distribution of the price levels in the random price path $\Prmvecbar_j$ satisfies $\mathbb P \{ \Prmbar_{jt} = \dis \} = \ybar_{jt}^\dis$. Furthermore, the random price path $\Prmvecbar_j$ satisfies promotion fatigue constraints with probability one. We note that there are $O(T)$ possible realizations of the random price path $\Prmvecbar_j$. In particular, the realization of $\Prmbar_{jt}$ does not change as long as the uniform random variable $\text{\sf Unif}$ takes a value in the set $\Ucal_{jt}$. For some integer $k$, we have either $\Ucal_{jt} = [ \sbar_{j,t-1} - k , \sbar_{jt} - k)$ or $\Ucal_{jt} =[0 , \sbar_{jt} - k) \cup \ts[\sbar_{j,t-1} - k + 1 , 1)$. In the former case, we set  $L_{jt} = \sbar_{j,t-1} -k$ and $U_{jt} = \sbar_{jt} - k$, whereas in the latter case, we set $L_{jt} = \sbar_{jt} -k$ and $U_{jt} = \sbar_{j,t-1} - k+1$ to capture the non-trivial end points of these intervals. Focusing on product $j$, we collect these end points to obtain the set of points $\{ L_{jt} : t \in \Tcal \} \cup \{ U_{jt} : t \in \Tcal\} \cup \{0,1\}$, drop the duplicates and sort the remaining ones in increasing order to obtain the set of points $\{ \nubar_j^q : q = 0,1,\ldots,m\}$ with $0 = \nubar_j^0 < \nubar_j^1 < \ldots < \nubar_j^m = 1$ and $m = O(T)$.~The realization of none of the price levels in the random price path $\Prmvecbar_j = (\Prmbar_{jt} : t \in \Tcal)$ changes as long as the uniform random variable $\text{\sf Unif}$ takes a value in one of the intervals in the collection \mbox{$\{[\nubar_j^{q-1},\nubar_j^q): q = 1,\ldots,m\}$} and there are $O(T)$ intervals in the last collection, establishing that there are indeed $O(T)$ possible realizations of the random price path $\Prmvecbar_j$.

We use $\{\pvecbar_j^q : q \in \Mcal\}$ with $|\Mcal| = O(T)$ to denote the possible realizations of the random price path $\Prmvecbar_j$, where $\pvecbar_j^q = (\pbar_{jt}^q : t \in \Tcal) \in \{\dis,\reg\}^T$ captures the prices in price path $q$. Consider problem (\ref{eqn:path_fluid}) with the understanding that the set of possible price levels is $\Ncal = \{ \dis,\reg\}$ and the set of possible price paths $\{\pvecbar_j^q : q \in \Mcal\}$ is as given in the previous sentence. By the same reasoning in Proposition \ref{pro:equivalence}, we can show that if $(\xvecbar,\yvecbar,\zvecbar)$ is an optimal solution to problem $(\ref{eqn:path_fluid})$ after modifying this problem as described earlier in this paragraph, then $(\xvecbar,\yvecbar)$ is an optimal solution to the fluid approximation in (\ref{eqn:fluid_fatigue}). Also, by the same reasoning in Proposition \ref{pro:extreme}, we can show that if $(\xvecbar,\yvecbar,\zvecbar)$ is an extreme point optimal solution to problem $(\ref{eqn:path_fluid})$, once again, after modifying this problem as described earlier in this paragraph, then there are two price paths $\Fs,\Ss \in \Mcal$ for each product $j$ such that $\zbar_j^q = 0$ for all $q \in \Mcal \setminus \{\Fs,\Ss\}$. Thus, Propositions \ref{pro:equivalence} and \ref{pro:extreme} extend to promotion fatigue constraints. Our approximate policy under promotion fatigue constraints closely mirrors the one under price monotonicity constraints.  We define $\Rbar_j$ and $\Cbar_j$ as in just before (\ref{eqn:path_rev}), as well as $\reve_j^q$ and $\capa_j^q$ as in (\ref{eqn:path_rev}), with the understanding that $\Ncal = \{ \dis,\reg\}$ and $(\xvecbar,\yvecbar)$ is an optimal solution to problem (\ref{eqn:fluid_fatigue}). We follow the approximate policy precisely as in Section \ref{sec:policy}. By the same reasoning in Theorem \ref{thm:perf}, we can show that this approximate policy has a performance guarantee of $\max \Big\{ \frac{1}{8L} \ts , \ts \frac{1}{2} - \sqrt{\frac{\log c_{\min}}{2 \ts c_{\min}}} - \frac{L}{c_{\min}} \Big\}$ under promotion fatigue constraints. Also, we follow the ex-post approximate policy precisely as in Section \ref{sec:expost_policy}. By the same reasoning in Theorem \ref{thm:expost}, we can show that this ex-post approximate policy has a performance guarantee of $1 - \sqrt{\frac{2\log c_{\min}}{ c_{\min}}} - \frac{L + \Deltabar}{c_{\min}}$ under promotion fatigue constraints. 

\section{Fluid Approximations Under Inter-Temporal Constraints}
\label{sec:gen_const}

Under both price monotonicity and promotion fatigue constraints, we started with a natural fluid approximation. By using an optimal solution to the natural fluid approximation, we constructed a collection of price paths. In this case, our approximate policies followed one of the price paths in the collection that we constructed. It turns out that we can follow a similar outline for a pricing problem with inter-temporal constraints other than price monotonicity and promotion fatigue constraints. Under other inter-temporal price constraints, we may still construct a natural fluid approximation and use an optimal solution to the natural fluid approximation to construct a collection of price paths, in which case, we can give an approximate policy by following one of the price paths in the collection. Under other inter-temporal price constraints, however, the number of price paths in the collection is not necessarily polynomial in the input size, whereas under the price monotonicity and promotion fatigue constraints, the numbers of price paths in our collections are polynomial in the input size. Consider a pricing problem with general inter-temporal price constraints. We use the binary vector $\yvec = (y_{jt}^\ell : j \in \Jcal,~\ell \in \Ncal,~t \in \Tcal) \in \{0,1\}^{|\Jcal| \times n \times T}$ to capture the prices charged for the products over the selling horizon, where $y_{jt}^\ell = 1$ if and only if we charge price level $\ell$ for product $j$ at time period $t$.  The set of feasible price paths for product~$j$ is given by the generic polytope \mbox{$\Pcal_j = \{ \yvec_j \in [0,1]^{n \times T} : \sum_{t \in \Tcal} \sum_{\ell \in \Ncal} \Psi_{vjt}^\ell \ts y_{jt}^\ell \leq B_v~\forall \ts v \in \Vcal\}$}, where we use the vector \mbox{$\yvec_j = (y_{jt}^\ell : \ell \in \Ncal,~t \in \Tcal)$} and the set $\Vcal$ includes generic indices. Thus, if the prices that a policy charges for product $j$ are given by the vector $\yvec_j = (y_{jt}^\ell : \ell \in \Ncal,~t \in \Tcal)$ so that $y_{jt}^\ell = 1$ if and only if the policy charges price level $\ell$ for product $j$ at time period $t$, then the price path for product $j$ has to satisfy \mbox{$\yvec_j \in \Pcal_j \cap \{0,1\}^{n \times T}$}. We shortly give the polytope $\Pcal_j$ when, for example, we have price monotonicity or promotion fatigue constraints. As in our development for price monotonicity and promotion fatigue constraints, we allow the possibility of not making a product available for purchase. In this way, under price monotonicity constraints, for example,  we can shut off the demand for a product by not making it available rather than charging a large price, as charging a large price has implications on what prices we can charge at subsequent time periods under price monotonicity constraints. Our goal is to decide which products to make available and what prices to charge for the available products so that we maximize the total expected revenue, while making sure that the prices that we charge for product $j$ satisfy the feasibility constraints characterized by the polytope $\Pcal_j$. Using the decision variables \mbox{$\xvec = (x_{jt}^\ell: j \in \Jcal,~\ell \in \Ncal,~t \in \Tcal)$}  and \mbox{$\yvec = (y_{jt}^\ell: j \in \Jcal,~\ell \in \Ncal,~t \in \Tcal)$}, we consider the fluid approximation 
\begin{align}
\max_{(\xvec,\yvec) \in [0,1]^{2 \times |\Jcal| \times n \times T}}
\Bigg\{ 
& \sum_{t \in \Tcal} \sum_{j \in \Jcal} \sum_{\ell \in \Ncal} r_j^\ell \ts \lambda_{jt}^\ell \ts x_{jt}^\ell ~:~
\sum_{t \in \Tcal} \sum_{j \in \Jcal} \sum_{\ell \in \Ncal} a_{ij} \ts \lambda_{jt}^\ell \ts x_{jt}^\ell \leq c_i \qquad \forall \ts i \in \Lcal,
\nonumber
\\
& \qquad \qquad x_{jt}^\ell \leq y_{jt}^\ell \quad \forall \ts j \in \Jcal,~\ell \in \Ncal,~t \in \Tcal,~~~ \yvec_j \in \Pcal_j \quad \forall \ts j \in \Jcal
\Bigg\}.
\label{eqn:fluid_gen}
\end{align}

By the same reasoning as in Proposition \ref{pro:ub}, we can show that the optimal objective value of problem (\ref{eqn:fluid_gen}) is an upper bound on the optimal total expected revenue. To show this result, we can construct a feasible solution to problem (\ref{eqn:fluid_gen}) by using the decisions of the optimal policy. The decision variables in problem (\ref{eqn:fluid_gen}) only characterize the price for each product at each time period, so they do not characterize a price path that satisfies the inter-temporal constraints. Under price monotonicity or promotion fatigue constraints, by using a uniform random variable over the interval $[0,1)$, we could construct a random price path $\Prmvecbar_j$ for each product $j$ such that the random price path satisfies price monotonicity or promotion fatigue constraints with probability one. By using the possible realizations of this random price path, we could equivalently formulate the fluid approximation as in problem (\ref{eqn:path_fluid}). In an optimal solution to problem (\ref{eqn:path_fluid}), for each product $j$, at most two components of the decision variables $(z_j^q : q \in \Mcal)$ take strictly positive values, in which case, by following the price path corresponding to one of these decision variables, we constructed an approximate policy with a performance guarantee. It is not at all clear whether we can use  problem (\ref{eqn:fluid_gen}) to come up with a collection of price paths that satisfy the feasibility constraint.

We proceed to showing that we can indeed use problem (\ref{eqn:fluid_gen}) to construct a collection of price paths such that each of these price paths satisfies the constraints imposed on feasible price paths, as long as~the polytope $\Pcal_j$ satisfies a certain assumption. In particular, we assume that the extreme~points~of the polytope $\Pcal_j$ have integer values. When we give the form of the polytope $\Pcal_j$ under price monotonicity or promotion fatigue constraints, we show that the extreme points of the polytope do  take integer values. Let $\{ \alphavecbar_j^q : q \in \Mcal\}$ be the collection of the extreme points of $\Pcal_j$. By our assumption of extreme points with integer values, we have $\alphavecbar_j^q \in \{0,1\}^{n \times T}$. In this case, we can express any point $\yvec_j \in \Pcal_j$ as  $\sum_{q \in \Mcal} z_j^q \ts \alphavecbar_j^q$ for some $\zvec_j = (z_j^q : q \in \Mcal) \in \mathbb R_+^{|\Mcal|}$ with $\sum_{q \in \Mcal} z_j^q = 1$. Using the decision variables $\zvec = (z_j^q : j \in \Jcal,~q \in \Mcal)$, denoting the components of the vector $\alphavecbar_j^q$ as $(\alphabar_{jt}^q(\ell) : \ell \in \Ncal,~t \in \Tcal)$, the constraint $\yvec_j \in \Pcal_j$ for all $j \in \Jcal$ is equivalent to $y_{jt}^\ell = \sum_{q \in \Mcal} \alphabar_{jt}^q(\ell) \ts z_j^q$ for all $j \in \Jcal$, $\ell \in \Ncal$, $t \in \Tcal$ and $\sum_{q \in \Mcal} z_j^q = 1$ for all $j \in \Jcal$. Thus, problem (\ref{eqn:fluid_gen}) is equivalent to 
\begin{align}
\max_{(\xvec,\yvec,\zvec) \in [0,1]^{|\Jcal| \ts (2 \times n \times T + |\Mcal|)}}
\Bigg\{ 
&\sum_{t \in \Tcal} \sum_{j \in \Jcal} \sum_{\ell \in \Ncal} r_j^\ell \ts \lambda_{jt}^\ell \ts x_{jt}^\ell ~:~
\sum_{t \in \Tcal} \sum_{j \in \Jcal} \sum_{\ell \in \Ncal} a_{ij} \ts \lambda_{jt}^\ell \ts x_{jt}^\ell \leq c_i \qquad \forall \ts i \in \Lcal,
\nonumber
\\
&\qquad x_{jt}^\ell \leq y_{jt}^\ell \qquad \forall \ts j \in \Jcal,~\ell \in \Ncal,~t \in \Tcal,
\phantom{\sum_n^n}
\nonumber
\\
& \qquad
y_{jt}^\ell = \sum_{q\in \Mcal} \alphabar_{jt}^q(\ell) \ts z_j^q \qquad \forall \ts j \in \Jcal,~\ell \in \Ncal,~t \in \Tcal,
\nonumber
\\
& \qquad \sum_{q \in \mathcal M} z_j^q = 1 \qquad \forall \ts j \in \Jcal \Bigg\}.
\label{eqn:path_fluid_gen}
\end{align}

Identifying $\alphabar_{jt}^q(\ell) \in \{0,1\}$ in the problem above with $\ind{\pbar_{jt}^q = \ell}$ in (\ref{eqn:path_fluid}), the problem above has the same form as problem (\ref{eqn:path_fluid}). We can view $\alphavecbar_j^q = (\alphabar_{jt}^q(\ell) : \ell \in \Ncal,~t \in \Tcal)$ as a price path, where we charge price level $\ell$ for product $j$ at time period $t$ if and only if $\alphabar_{jt}^q(\ell) = 1$. Because $\alphavecbar_j^q$ is an extreme point of $\Pcal_j$, we have $\alphavecbar_j^q \in \Pcal_j$, so the price path corresponding to $\alphavecbar_j^q$ satisfies the \mbox{inter-temporal} price constraints. Because problems (\ref{eqn:path_fluid}) and (\ref{eqn:path_fluid_gen}) have the same structure, Propositions \ref{pro:equivalence} and \ref{pro:extreme} hold for problem (\ref{eqn:path_fluid_gen}) as well. In this case, letting $(\xvecbar,\yvecbar,\zvecbar)$ be an extreme point solution to problem $(\ref{eqn:path_fluid_gen})$, there are two price paths $\Fs,\Ss \in \Mcal$ for each product $j$ such that $\zbar_j^q = 0$ for all $q \in \Mcal \setminus \{\Fs,\Ss\}$. In this case, we can construct approximate policies with performance guarantees by following one of the price paths corresponding to $\alphavecbar_j^\Fs$ or $\alphavecbar_j^\Ss$ for product $j$. 

{\bf \underline{Feasible Price Paths Under Price Monotonicity and Promotion Fatigue}:}
\\
\indent We focus on the form of $\Pcal_j$ for price monotonicity and promotion fatigue constraints. Under  price monotonicity constraints, we can express the set of feasible price paths for product $j$ by using \mbox{$\Pcal_j = \{ \yvec_j \in [0,1]^{n \times T} : \sum_{k=\ell}^n y_{j,t-1}^k \leq \sum_{k=\ell}^n y_{jt}^k~\forall \ts \ell \in \Ncal,~t \in \Tcal \setminus \{1\},~~\sum_{\ell \in \Ncal} y_{jt}^\ell = 1 ~\forall \ts t \in \Tcal\}$}, where the first constraint ensures that if we charge price level $\ell$ or higher at time period $t-1$, then we charge price level $\ell$ or higher at time period $t$ as well, whereas the second constraint ensures that we pick one price level at time period $t$. Recall that we facilitate not making product $j$ available in (\ref{eqn:fluid_gen}) and (\ref{eqn:path_fluid_gen}) by using the decision variables $(x_{jt}^\ell : \ell \in \Ncal,~t \in \Tcal)$. By the second constraint in $\Pcal_j$, we have $\sum_{k=\ell}^n y_{jt}^\ell =1 - \sum_{k=1}^{\ell-1} y_{jt}^\ell$, so the first constraint in $\Pcal_j$ is \mbox{$\sum_{k=\ell}^n y_{j,t-1}^k + \sum_{k=1}^{\ell-1} y_{jt}^k \leq 1$}, in which case, we can show that the polyhedron $\Pcal_j$ is characterized by an interval matrix, which is known to be totally unimodular. We give the details in Appendix \ref{sec:tum_mon}. 
Under promotion fatigue constraints, on the other hand, we can express the set of feasible price paths for product~$j$ by using \mbox{$\Pcal_j = \{ \yvec_j \in [0,1]^{n \times T} : \sum_{\tau = t}^{(t+K-1) \wedge T}  y_{j\tau}^\dis \leq 1~  \forall \ts t \in \Tcal, ~~ \sum_{\ell \in \Ncal} \ts y_{jt}^\ell = 1 ~ \forall \ts t \in \Tcal \}$}, where the first constraint ensures that we can charge the promoted price at most once over any duration of $K$ time periods at most once. Because $\Ncal = \{ \dis,\reg\}$, the second constraint is equivalent to $y_{jt}^\dis\leq 1$, in which case, we can show that the polyhedron $\Pcal_j$ is characterized by an interval matrix appended by the identity matrix, which is known to be totally unimodular. We give the details in Appendix \ref{sec:tum_fatigue}. By the discussion in this section, we may not need to guess the fluid approximation under general inter-temporal price constraints. We can construct a polytope to describe what it means to have a feasible price path under such constraints. If the polytope has integer extreme points, then we can construct a fluid approximation and extract feasible price paths from this fluid approximation. 

\section{Computational Experiments}
\label{sec:exp}

We give computational experiments on synthetic datasets, as well as datasets based on a real-world hotel, to test the effectiveness of our approximate policies.

\vspace{-2mm}

\subsection{Synthetic Datasets}

\vspace{-2mm}

We describe our experimental setup and benchmark policies, followed by our computational results.~We focus on price monotonicity constraints.

{\bf \underline{Experimental Setup}:} We consider an airline network with one hub and three spokes. There is a flight leg from each spoke to the hub and from the hub to each spoke, so there are six flight legs. There is an itinerary that connects every origin-destination pair, so there are 12 itineraries. The itineraries from a spoke to another spoke connect at the hub, whereas the itineraries from a spoke to the hub or from the hub to a spoke are direct. Flight legs correspond to resources and itineraries correspond to products. For each product $j$ and time period $t$, we sample $\zeta_{jt}$ from the uniform distribution over the interval $[0,1]$. Setting $\theta_{jt} = \zeta_{jt} / \sum_{k \in \Jcal} \zeta_{kt}$, the customer arriving at time period $t$ is interested in purchasing product $j$ with probability $\theta_{jt}$. There are 40 price levels for each product. For each product $j$, the revenue corresponding to price level $\ell$ is $r_j^\ell = \frac 12 \ts \ell$ for $\ell = 1,\ldots,40$. Thus, the revenues associated with the different price levels are uniformly placed over the interval $[0.5, 20]$. For each product $j$ and time period $t$, we sample $\alpha_{jt}$ from the uniform distribution over the interval $[0.1, 0.5]$ to capture the price sensitivity of the demand, in which case, if we charge price level $\ell$ for product $j$, then a customer interested in purchasing product $j$ at time period $t$ makes a purchase with probability $\gamma_{jt}^\ell = \exp( - \alpha_{jt} \ts \frac 12 \ts \ell )$. Thus, if we charge price level $\ell$ for product $j$ at time period $t$, then we make a sale for the product with probability $\lambda_{jt}^\ell = \theta_{jt} \ts \gamma_{jt}^\ell$. 

We vary the number of time periods $T$ in the selling horizon. To come up with the capacities for the resources, we solve a dynamic program to maximize the total expected revenue from each product while satisfying price monotonicity constraints. In particular, we solve the dynamic program $\nu_{jt}(\ell) = \max_{k \in \{\ell,\ldots,n\}} \lambda_{jt}^\ell \ts r_j^\ell + \nu_{j,t+1}(k)$ with the boundary condition that $\nu_{j,T+1} = 0$. Letting \mbox{$\{ \overline \ell_{jt} : t \in \Tcal\}$} be the optimal trajectory of price levels starting with $\overline \ell_{j1} = 1$, if we ignore the resource capacities and follow the price trajectory to maximize the total expected revenue for each product, then the total expected demand for product $j$ is $\Lambda_j = \sum_{t \in \Tcal} \sum_{\ell \in \Ncal}  \ind{\overline \ell_{jt} = \ell} \ts \lambda_{jt}^\ell$, so the total expected demand for the capacity of resource $i$ is $D_i = \sum_{j \in \Jcal} a_{ij} \ts \Lambda_j$. We set the capacity of resource $i$ as $c_i = \ru{D_i / \rho}$, where we vary the parameter $\rho$ to control the tightness of the resource capacities. 

Varying \mbox{$T \in \{500,1000, 2000\}$} and $\rho \in \{1.2, 1.4, 1.6, 1.8, 2\}$, we get 15 parameter configurations. We generate a test problem for each parameter configuration using the approach above.

{\bf \underline{Benchmark Policies\phantom{p}\!\!\!}:} We test six benchmarks, four based on our approximate policies and two based on using price paths that are agnostic to resource capacities. 

{\it \underline{Benchmarks Based on Approximate Policies}.} Our first benchmark is the approximate policy in Section~\ref{sec:policy}, but we make product $j$ available at time period $t$ with probability $\thetabar_{jt} \ts \frac{\Cbar_j}{\max\{ \capa_j^\Fs , \Cbar_j\}}$, so we set $\gamma = 1$. We refer to this benchmark as APT, standing for approximate policy with availability threshold.  The practical performance of our approximate policy is slightly better with $\gamma = 1$.~Our second benchmark, referred to as AP1, is the approximate policy in Section \ref{sec:policy}, but we make product $j$ available at time period $t$  with probability $\thetabar_{jt}$. Our third benchmark is the  ex-post approximate policy in Section \ref{sec:expost_policy}, but we make product $j$ available at time period $t$ with probability $\thetabarp_{jt} \frac{c_{\min}}{c_{\min} + \Deltabar}$. Our fourth benchmark is the ex-post approximate policy, but we make product $j$ available at time period $t$ with probability $\thetabarp_{jt}$. We refer to the third and fourth benchmarks as EPT and EP1.

%Our first benchmark is the approximate policy in Section~\ref{sec:policy}. In the approximate policy, we choose the tuning parameter $\gamma$ to get a performance guarantee of $\frac{1}{2} - \sqrt{\frac{\log c_{\min}}{2 \ts c_{\min}}} - \frac{L}{c_{\min}}$. We refer to this benchmark as APT, standing for approximate policy with tuning parameter.  Setting $\gamma$ away from one is necessary to give a performance guarantee, but the practical performance of our approximate policy can be stronger when we set $\gamma = 1$. Our second benchmark is the approximate policy in Section \ref{sec:policy} with $\gamma = 1$. We refer to this benchmark as AP1. Our third benchmark is the  ex-post approximate policy in Section \ref{sec:expost_policy}. We refer to this benchmark as EPT, standing for ex-post approximate policy with tuning parameter. Our fourth benchmark, which we refer to as EP1, is the ex-post approximate policy with $\gamma =1$.

{\it \underline{Benchmarks Based on Capacity Agnostic Prices}.} Our fifth benchmark is based on computing a price path for each product under price monotonicity constraints without taking the resource capacities into consideration. As done when calibrating the capacities of the resources in our experimental setup, we solve the dynamic program  $\nu_{jt}(\ell) = \max_{k \in \{\ell,\ldots,n\}} \lambda_{jt}^k \ts r_j^k + \nu_{j,t+1}(k)$ with the boundary condition that $\nu_{j,T+1} = 0$. Letting  $\{ \overline \ell_{jt} : t \in \Tcal\}$ be the optimal trajectory of price levels starting with $\overline \ell_{j0} = 1$, we charge the price level $\overline \ell_{jt}$ for product $j$ at time period $t$, as long as we have resource capacities to serve a request for product $j$. Otherwise, we do not make product $j$ available. We refer to this benchmark as UCP, standing for unlimited capacity price paths. 
Our sixth benchmark adjusts these price paths to take resource capacities into consideration.  Using the trajectory of price levels $\{ \overline \ell_{jt} : t \in \Tcal\}$ as earlier in this paragraph, we set \mbox{$\overline \Lambda_{jt} = \sum_{\ell \in \Ncal} \ind{\ell = \overline \ell_{jt}} \ts \lambda_{jt}^\ell$} and $\Rbar_{jt} = \sum_{\ell \in \Ncal} \ind{\ell = \overline \ell_{jt}} \ts r_j^\ell \ts \lambda_{jt}^\ell$ to, respectively, denote the demand probability and expected revenue from product $j$ at time period $t$ under the trajectory of price levels. Using the decision variable $x_{jt}$ to capture the probability of making product $j$ available at time period~$t$, we solve the linear program $\max_{\xvec \in [0,1]^{|\Jcal| \times T}}\{ \sum_{t \in \Tcal} \sum_{j \in \Jcal} \Rbar_{jt}\ts x_{jt} : \sum_{t \in \Tcal} \sum_{j \in \Jcal} \ts a_{ij} \ts  \overline \Lambda_{jt} \ts x_{jt} \leq c_i ~~\forall \ts i \in \Lcal\}$, where we have the vector $\xvec = (x_{jt} : j \in \Jcal,~t \in \Tcal)$. In the linear program, we find the probability of making each product available at each time period to maximize the total expected revenue, while adhering to the resource capacities. Letting $\xvecbar$ be an optimal solution, we make product $j$ available at time period $t$ with probability $\xbar_{jt}$. If we make product $j$ available at time period $t$, then we charge price level $\overline \ell_{jt}$. We refer to this benchmark as CAP, standing for capacity aware price paths. 

\begin{table}
\begin{center}
\scriptsize
\begin{tabular}{|c|cccccc|ccccc|}
\hline
Param. &  \multicolumn{6}{c|}{Total Expected Revenue} & \multicolumn{5}{c|}{Percent Gap with AP1} \\
$(\rho,T)$ & APT & AP1 & EPT & EP1 & UCP & CAP & APT & EPT & EP1 & UCP & CAP \\
\hline
\hline
$(1.2,~\ts500)$	&	92.56	&	92.56	&	89.66	&	91.58	&	81.91	&	82.15	&	0.00	&	3.13	&	1.06	&	11.51	&	11.25	\\
$(1.2,1000)$	&	94.61	&	94.61	&	89.69	&	92.71	&	83.99	&	84.18	&	0.01	&	5.21	&	2.01	&	11.23	&	11.03	\\
$(1.2,2000)$	&	96.40	&	96.40	&	91.03	&	94.03	&	84.21	&	85.16	&	0.00	&	5.57	&	2.45	&	12.64	&	11.65	\\
\hline
$(1.4,~\ts500)$	&	91.41	&	91.45	&	88.56	&	91.09	&	73.22	&	75.53	&	0.04	&	3.16	&	0.40	&	19.94	&	17.41	\\
$(1.4,1000)$	&	94.05	&	94.08	&	89.81	&	92.91	&	74.63	&	77.16	&	0.04	&	4.54	&	1.24	&	20.68	&	17.98	\\
$(1.4,2000)$	&	95.95	&	95.96	&	92.39	&	95.09	&	74.35	&	78.13	&	0.01	&	3.72	&	0.91	&	22.52	&	18.59	\\
\hline
$(1.6,~\ts500)$	&	90.74	&	90.77	&	89.29	&	90.56	&	66.43	&	71.05	&	0.04	&	1.64	&	0.24	&	26.82	&	21.73	\\
$(1.6,1000)$	&	93.12	&	93.29	&	89.93	&	93.22	&	67.53	&	72.24	&	0.18	&	3.60	&	0.07	&	27.61	&	22.57	\\
$(1.6,2000)$	&	95.45	&	95.52	&	93.46	&	95.18	&	67.28	&	73.48	&	0.07	&	2.16	&	0.35	&	29.57	&	23.07	\\
\hline
$(1.8,~\ts500)$	&	90.36	&	90.40	&	88.00	&	89.92	&	60.90	&	67.73	&	0.04	&	2.65	&	0.53	&	32.63	&	25.08	\\
$(1.8,1000)$	&	93.18	&	93.18	&	90.77	&	92.58	&	62.14	&	69.02	&	0.00	&	2.59	&	0.65	&	33.32	&	25.93	\\
$(1.8,2000)$	&	95.15	&	95.22	&	91.51	&	94.86	&	61.63	&	70.08	&	0.07	&	3.89	&	0.38	&	35.27	&	26.40	\\
\hline
$(2.0,~\ts500)$	&	89.97	&	89.98	&	87.09	&	89.29	&	56.99	&	65.40	&	0.01	&	3.21	&	0.76	&	36.67	&	27.32	\\
$(2.0,1000)$	&	92.84	&	92.85	&	88.35	&	92.12	&	57.60	&	66.61	&	0.01	&	4.85	&	0.78	&	37.96	&	28.26	\\
$(2.0,2000)$	&	94.95	&	95.00	&	91.67	&	94.53	&	57.34	&	67.70	&	0.05	&	3.50	&	0.49	&	39.64	&	28.73	\\
\hline
\hline
Average	&	93.38	&	93.42	&	90.08	&	92.64	&	68.68	&	73.71	&	0.04	&	3.56	&	0.82	&	26.53	&	21.13	\\
\hline
\end{tabular}
\caption{Total expected revenues obtained by the benchmarks for the synthetic datasets.}
\label{tab:exp_syn}
\end{center}
\vspace{-7mm}
\end{table}

{\bf \underline{Computational Results\phantom{p}\!\!\!}:} We give our computational results in Table \ref{tab:exp_syn}. The first column gives the parameter configuration by using the pair $(\rho,T)$. The second to seventh columns give the total expected revenues obtained by our benchmarks. The optimal objective value of the \ref{eqn:fluid} is an upper bound on the optimal total expected revenue, so we report the total expected revenues of all benchmarks by normalizing with the upper bound. The benchmarks APT and AP1 use our approximate policy, EPT and EP1 use our ex-post approximate policy, UCP and CAP use price paths computed without taking resource capacities into consideration. Noting that AP1 is one of the strongest benchmarks, the eighth to twelfth columns give the percent gap between the total expected revenues of AP1 and the remaining five benchmarks. We estimate the total expected revenues of all benchmarks using Monte Carlo simulation.

Our results indicate that our approximate policy performs quite well. On average, AP1 obtains 93.42\% of the upper bound on the optimal total expected revenue. The capacities of the resources are larger for the test problems with a larger number of time periods in the selling horizon. In alignment with the performance guarantees for our approximate and ex-post approximate policies, APT, AP1, EPT and EP1 obtain larger fractions of the upper bound on the optimal total expected revenue when the value of $T$ is large so that there is a large number of time periods in the selling horizon. The performance of UCP and CAP is poor especially when the value of $\rho$ is large so that the resource capacities are tight, which is not surprising as these benchmarks use price trajectories computed without using resource capacities. Our approximate and ex-post approximate policies also provide larger improvements over UCP and CAP when the resource capacities are tight. 

\vspace{-3mm}

\subsection{Hotel Datasets}

\vspace{-3mm}

We give computational experiments based on a publicly available dataset that provides bookings from an urban hotel; see \cite{Ka19}.

\vspace{-0.275mm}

{\bf \underline{Experimental Setup}:} The dataset gives the bookings at an urban hotel over a year. Each row in the dataset corresponds to a booking, giving the date when the booking was made,  start date and number of days for the stay, as well as the price per stay day paid. We focus on the bookings made up to nine weeks in advance of the start date of the stay and for fewer than seven nights,  dropping all other bookings, yielding 11 bookings per day on average. We refer to the number of days between the booking date and stay date as the lead time of a booking. Letting~$N_\tau$ be the number of bookings with a lead time of $\tau$ days in the dataset, we estimate the fraction of bookings with a lead time of $\tau$ days as $\pi_\tau = N_\tau / \sum_{s = 1}^{9 \times 7} N_s$. Similarly, letting $M_\kappa$ be the number of bookings for $\kappa$ nights in the dataset, we estimate the fraction of bookings for $\kappa$ nights as $\zeta_\kappa = M_\kappa / \sum_{s=1}^7 M_s$. 
In our benchmarks, we consider stays in the hotel over one week. Therefore, there are seven resources, each corresponding to a different night. A stay can start and end on any day in the week, yielding $\frac 12 \times 7 \times 8 = 28$ products. There are $K$ different price levels for each product, where we vary the parameter $K$. If product $j$ corresponds to a stay of $q$ days, then the revenue corresponding to price level $\ell$ for product $j$ is  $r_j^\ell = (40 + \frac{160}{K-1} \ts (\ell-1)) \times q$, so the price per stay day takes values in the interval $[40,200]$, which is in alignment with the prices in the dataset. We divide the selling horizon into nine segments, so that the customers booking in different segments have different price sensitivities. If we charge price level $\ell$ for a product that involves $q$ nights of stay, then a customer arriving in segment $s$ books with probability $1 / (1 + \exp(\beta_s + \alpha_s \ts(40 + \frac{160}{K-1} \ts (\ell-1)) \times q))$. In Appendix \ref{sec:hotel_data}, we explain our approach for estimating the parameters $\{ (\beta_s , \alpha_s) : s = 1,\ldots,9\}$. 

\vspace{-0.275mm}

Whenever a booking inquiry occurs, we sample the start day of the stay and number of nights for the stay from the empirical distributions obtained from the dataset as described at the beginning of the previous paragraph. In this case, if a customer arriving in segment $s$ is interested in booking some product $j$ that involves $q$ nights of stay and we charge price level $\ell$ for this product, then the customer makes the booking with probability $1 / (1 + \exp(\beta_s + \alpha_s \ts(40 + \frac{160}{K-1} \ts (\ell-1) ) \times q))$. Otherwise, the customer leaves without a booking. We continue focusing on price monotonicity constraints so that the customers making bookings early in the selling horizon are charged lower prices. To come up with the room capacities available on each night, we use the same approach as in our computational experiments with synthetic datasets. Thus, the total expected demand for the capacity on each night under monotone prices that maximize the total expected revenue exceeds the room capacity on that night by a factor of $\rho$. We vary the parameter $\rho$. Varying $K \in \{10,20,40\}$ and $\rho \in \{1.2, 1.4, 1.6, 1.8, 2\}$, we get 15 parameter configurations. For each parameter configuration, we have a test problem constructed by using the approach discussed so far in this section.

\begin{table}
\begin{center}
\scriptsize
\begin{tabular}{|c|cccccc|ccccc|}
\hline
Param. &  \multicolumn{6}{c|}{Total Expected Revenue} & \multicolumn{5}{c|}{Percent Gap with AP1} \\
$(\rho,K)$ & APT & AP1 & EPT & EP1 & UCP & CAP & APT & EPT & EP1 & UCP & CAP \\
\hline
\hline
$(1.2,10)$	&	97.84	&	97.88	&	96.46	&	97.30	&	84.08	&	88.36	&	0.04	&	1.45	&	0.59	&	14.10	&	9.73	\\
$(1.2,20)$	&	97.82	&	97.84	&	96.95	&	97.61	&	83.57	&	87.85	&	0.02	&	0.91	&	0.23	&	14.58	&	10.22	\\
$(1.2,40)$	&	97.87	&	97.87	&	97.41	&	97.64	&	83.41	&	87.91	&	0.00	&	0.47	&	0.23	&	14.77	&	10.17	\\
\hline
$(1.4,10)$	&	97.85	&	97.86	&	95.96	&	97.23	&	74.07	&	81.54	&	0.01	&	1.95	&	0.64	&	24.32	&	16.68	\\
$(1.4,20)$	&	97.72	&	97.74	&	96.88	&	97.54	&	73.38	&	80.82	&	0.02	&	0.88	&	0.20	&	24.92	&	17.30	\\
$(1.4,40)$	&	97.67	&	97.68	&	97.07	&	97.54	&	73.14	&	80.99	&	0.01	&	0.63	&	0.14	&	25.12	&	17.09	\\
\hline
$(1.6,10)$	&	97.32	&	97.45	&	95.50	&	97.09	&	66.81	&	76.37	&	0.14	&	2.01	&	0.37	&	31.45	&	21.64	\\
$(1.6,20)$	&	97.52	&	97.53	&	96.72	&	97.32	&	66.02	&	75.60	&	0.01	&	0.83	&	0.21	&	32.31	&	22.49	\\
$(1.6,40)$	&	97.54	&	97.55	&	97.14	&	97.43	&	65.74	&	75.73	&	0.00	&	0.42	&	0.12	&	32.61	&	22.37	\\
\hline
$(1.8,10)$	&	97.38	&	97.40	&	95.87	&	97.19	&	61.08	&	72.21	&	0.01	&	1.57	&	0.22	&	37.28	&	25.86	\\
$(1.8,20)$	&	97.29	&	97.34	&	96.83	&	97.26	&	60.33	&	71.49	&	0.05	&	0.52	&	0.08	&	38.03	&	26.56	\\
$(1.8,40)$	&	97.37	&	97.38	&	97.09	&	97.32	&	59.98	&	71.54	&	0.02	&	0.30	&	0.06	&	38.40	&	26.54	\\
\hline
$(2.0,10)$	&	97.14	&	97.22	&	95.80	&	97.05	&	56.63	&	68.95	&	0.08	&	1.46	&	0.17	&	41.75	&	29.07	\\
$(2.0,20)$	&	97.29	&	97.29	&	96.73	&	97.20	&	55.80	&	68.18	&	0.00	&	0.57	&	0.09	&	42.65	&	29.92	\\
$(2.0,40)$	&	97.25	&	97.27	&	96.84	&	97.22	&	55.42	&	68.12	&	0.03	&	0.44	&	0.05	&	43.03	&	29.97	\\
\hline
\hline
Average	&	97.53	&	97.55	&	96.62	&	97.33	&	67.96	&	77.04	&	0.03	&	0.96	&	0.23	&	30.35	&	21.04	\\
\hline
\end{tabular}
\caption{Total expected revenues obtained by the benchmarks for the hotel datasets.}
\label{tab:exp_hotel}
\end{center}
\vspace{-7mm}
\end{table}

{\bf \underline{Computational Results\phantom{p}\!\!\!}:} We give our computational results in Table \ref{tab:exp_hotel}. The layout of this table is identical to that of Table \ref{tab:exp_syn} other than that we use the pair $(\rho,K)$ to index our test problems, where $K$ is the number of different price levels for a product. The results in Table \ref{tab:exp_hotel} are aligned with those in Table \ref{tab:exp_syn}. On average, AP1 obtains 97.55\% of the upper bound on the optimal total expected revenue. The performance of APT, AP1, EPT and EP1 is close to each other, but using a tuning parameter of one gives a slight but consistent improvement. The performance of UCP and CAP is especially poor when the resource capacities are tight.

\vspace{-3mm}

\section{Conclusions}
\label{sec:conc}

\vspace{-3mm}

Our work opens several directions for research. In our model, the demand for a product depends on its price. It would be useful to study the case where the price charged for a product affects the demands for other products. Our efforts in this direction were not fruitful and this extension appears to be non-trivial and needs new set of tools. Also, we give a unifying framework that extends to general inter-temporal price constraints. It would be interesting to study  policies under other specific inter-temporal price constraints. Lastly, one can study other fluid approximations that may provide tighter upper bounds, potentially yielding stronger performance guarantees.

\bibliographystyle{ormsv080}
\renewcommand*{\bibfont}{\footnotesize}
\renewcommand*{\bibfont}{\normalfont\footnotesize\linespread{1}\selectfont}
\setlength{\bibsep}{0.5pt}
\bibliography{references}

\clearpage

\ECSwitch

\begin{APPENDICES}

\vspace{-10mm}

\begin{center}
\large
\underline{Electronic Companion}
\\
Inter-Temporal Price Constraints in Dynamic Pricing: Performance Guarantees Under Price Monotonicity and Promotion Fatigue
\end{center}

\titleformat{\section}{\normalsize \sf \bfseries}{\thesection}{1em}{}

\section{Proof of Proposition \ref{pro:ub}}
\label{sec:ub}

\vspace{-3mm}

We define the Bernoulli random variable $\Xrm_{jt}^\ell$ such that $\Xrm_{jt}^\ell =1$ if and only if the optimal policy charges price level $\ell$ for product $j$ at time period $t$. Furthermore, we define the random variable $\Yrm_{jt}^\ell$ recursively as $\Yrm_{jt}^\ell = \Xrm_{jt}^\ell + (1-\sum_{k \in \Ncal} \Xrm_{jt}^k) \ts \Yrm_{j,t-1}^\ell$ with the boundary condition that \mbox{$\Yrm_{j0}^1 =1$} and \mbox{$\Yrm_{j0}^\ell = 0$} for all $\ell \in \Ncal \setminus \{1\}$. Note that we have $\Yrm_{jt}^\ell = 1$ if and only if price level~$\ell$ is the lower bound for the price of product $j$ after we make the pricing decisions at time period $t$. We define the solution~$(\xvecbar,\yvecbar)$ for the \ref{eqn:fluid} such that $\xbar_{jt}^\ell = \mathbb E\{ \Xrm_{jt}^\ell \}$ and \mbox{$\ybar_{jt}^\ell = \mathbb E\{ \Yrm_{jt}^\ell\}$}.~We verify that the solution $(\xvecbar,\yvecbar)$~is feasible to the \ref{eqn:fluid}.~We define the Bernoulli random variable $\Brm_{jt}^\ell$ such that $\Brm_{jt}^\ell = 1$ if and only if the customer arriving at time period $t$ purchases product $j$ at price level $\ell$ under the optimal policy. We have $\mathbb E\{ \Brm_{jt}^\ell \ts | \ts \Xrm_{jt}^\ell = 1\} = \lambda_{jt}^\ell$ and \mbox{$\mathbb E\{ \Brm_{jt}^\ell \ts | \ts \Xrm_{jt}^\ell = 0\} = 0$}, so \mbox{$\mathbb E\{ \Brm_{jt}^\ell \} = \mathbb E\{ \Brm_{jt}^\ell \ts | \ts \Xrm_{jt}^\ell = 1\} \ts \mathbb P \{ \Xrm_{jt}^\ell = 1\} = \lambda_{jt}^\ell \ts \xbar_{jt}^\ell$}. The customer arriving at time period $t$ purchases product $j$ if and only if $\sum_{\ell \in \Ncal} \Brm_{jt}^\ell = 1$. Consumption of the capacity of resource~$i$ under the optimal policy cannot exceed its initial capacity, so  \mbox{$\sum_{t \in \Tcal} \sum_{j \in \Jcal} a_{ij} \ts \sum_{\ell \in \Ncal} \Brm_{jt}^\ell \leq c_i$} with probability one. Using $\mathbb E\{ \Brm_{jt}^\ell \} = \lambda_{jt}^\ell \ts \xbar_{jt}^\ell$, taking expectations in the last inequality yields $\sum_{t \in \Tcal} \sum_{j \in \Jcal} \sum_{\ell \in \Ncal} a_{ij} \ts \lambda_{jt}^\ell \ts \xbar_{jt}^\ell \leq c_i$, verifying the first constraint in the \ref{eqn:fluid}.

At each time period, the optimal policy either makes product $j$ unavailable or charges one price for product $j$, so $\sum_{\ell \in \Ncal} \Xrm_{jt}^\ell \leq 1$, in which case, we get \mbox{$\Yrm_{jt}^\ell = \Xrm_{jt}^\ell + (1-\sum_{k \in \Ncal} \Xrm_{jt}^k) \ts \Yrm_{j,t-1}^\ell \geq \Xrm_{jt}^\ell$}, so taking expectations in the last chain of inequalities yields $\ybar_{jt}^\ell \geq \xbar_{jt}^\ell$, verifying the second constraint in the \ref{eqn:fluid}. We have $\sum_{\ell \in \Ncal} \Yrm_{jt}^\ell = \sum_{\ell \in \Ncal} \Xrm_{jt}^\ell + (1 - \sum_{\ell \in \Ncal} \Xrm_{jt}^\ell) \ts \sum_{\ell \in \Ncal} \Yrm_{j,t-1}^\ell$ by the definition of $\Yrm_{jt}^\ell$. Therefore, if we have $\sum_{\ell \in \Ncal} \Yrm_{j,t-1}^\ell = 1$, then we also have $\sum_{\ell \in \Ncal} \Yrm_{jt}^\ell  =1$. By the boundary condition in the definition of $\Yrm_{jt}^\ell$, we have $\sum_{\ell \in \Ncal} \Yrm_{j0}^\ell =1$, which implies that $\sum_{\ell \in \Ncal} \Yrm_{jt}^\ell =1$ for all $t \in \Tcal$. Taking expectations in the last equality yields $\sum_{\ell \in \Ncal} \ybar_{jt}^\ell = 1$, verifying the fourth constraint in the \ref{eqn:fluid}. Using the definition of $\Yrm_{jt}^\ell$ once more, we have the equality $\sum_{k=\ell}^n \Yrm_{jt}^k = \sum_{k =\ell}^n \Xrm_{jt}^k + (1-\sum_{k \in \Ncal} \Xrm_{jt}^k) \ts \sum_{k=\ell}^n \Yrm_{j,t-1}^k$. We refer to this equality as the preservation of partial sums. Consider three cases. First, if $\sum_{k \in \Ncal} \Xrm_{jt}^k =0$, then the preservation of partial sums yields $\sum_{k=\ell}^n \Yrm_{jt}^k = \sum_{k=\ell}^n \Yrm_{j,t-1}^k$. Second, if $\Xrm_{jt}^q = 1$ with $q \geq \ell$, then we have $\sum_{k =\ell}^n \Xrm_{jt}^k =1$ and $\sum_{k \in \Ncal} \Xrm_{jt}^k = 1$, in which case, the preservation of partial sums yields $\sum_{k=\ell}^n \Yrm_{jt}^k =1$. Because $\sum_{k \in \Ncal} \Yrm_{j\tau}^k=1$ for all $\tau \in \Tcal$ by the discussion in this paragraph, we have the chain of inequalities $\sum_{k=\ell}^n \Yrm_{jt}^k =1 = \sum_{k \in \Ncal} \Yrm_{j,t-1}^k \geq \sum_{k=\ell}^n \Yrm_{j,t-1}^k$. Third, if $\Xrm_{jt}^q = 1$ with $q < \ell$, then $\sum_{k=\ell}^n \Xrm_{jt}^k = 0$.

To address the third case, we use the fact that the prices charged by the optimal policy are monotonically increasing over time. Because $\Xrm_{jt}^q=1$, the optimal policy charges price level $q$ for product $j$ at time period $t$, so noting that $q<\ell$, the optimal policy cannot charge price levels $\{\ell,\ldots,n\}$ for product $j$ at time periods $\{1,\ldots,t-1\}$. Therefore, we have $\sum_{k=\ell}^n \Xrm_{j\tau}^k=0$ for all $\tau = 1,\ldots,t-1$. Noting that we also have $\sum_{k=\ell}^n \Xrm_{jt}^k=0$, by the preservation of partial sums, we have $\sum_{k=\ell}^n \Yrm_{j\tau}^k = (1-\sum_{k \in \Ncal} \Xrm_{j\tau}^k) \ts \sum_{k=\ell}^n \Yrm_{j,\tau-1}^k \leq \sum_{k=\ell}^n \Yrm_{j,\tau-1}^k$ for all $\tau = 1,\ldots,t$. Because $\Xrm_{jt}^q = 1$ with \mbox{$q < \ell$}, it must be the case that $\ell \geq 2$, so using the boundary condition in the definition of $\Yrm_{jt}^\ell$, we get $\sum_{k=\ell}^n \Yrm_{j0}^k = 0$. In this case, having  $\sum_{k=\ell}^n \Yrm_{j\tau}^k \leq \sum_{k=\ell}^n \Yrm_{j,\tau-1}^k$ for all $\tau = 1,\ldots,t$ implies that $\sum_{k=\ell}^n \Yrm_{j\tau}^k =0$ for all $\tau = 1,\ldots,t$. In all of the three cases, we have $\sum_{k=\ell}^n \Yrm_{jt}^k \geq \sum_{k=\ell}^n \Yrm_{j,t-1}^k$, so taking expectations in this inequality yields $\sum_{k=\ell}^n \ybar_{jt}^k \geq \sum_{k=\ell}^n \ybar_{j,t-1}^k$, verifying the third constraint in the \ref{eqn:fluid}. Thus, the solution $(\xvecbar,\yvecbar)$ is feasible to the \ref{eqn:fluid}. By the definition of the random variable $\Brm_{jt}^\ell$, the optimal total expected revenue is $\opt = \sum_{t \in \Tcal} \sum_{j \in \Jcal} \sum_{\ell \in \Ncal} r_j^\ell \ts \mathbb E\{ \Brm_{jt}^\ell \}$, but noting that $\mathbb E\{ \Brm_{jt}^\ell\} = \lambda_{jt}^\ell \ts \xbar_{jt}^\ell$, we get $\opt = \sum_{t \in \Tcal} \sum_{j \in \Jcal} \sum_{\ell \in \Ncal} r_j^\ell \ts \lambda_{jt}^\ell \ts \xbar_{jt}^\ell$. The last expression is the objective value of the \ref{eqn:fluid} evaluated at the solution $(\xvecbar,\yvecbar)$. Thus, the solution $(\xvecbar,\yvecbar)$ is feasible to the \ref{eqn:fluid} and provides an objective value of $\opt$ for this problem, so the optimal objective value of the \ref{eqn:fluid} is at least $\opt$. \qed

In the proof of Proposition \ref{pro:ub}, note that we use the monotonicity of the prices charged by the optimal policy only when dealing with the third constraint in the \ref{eqn:fluid}.

\section{Proof of Proposition \ref{pro:equivalence}}
\label{sec:equivalence}

\vspace{-2mm}

Letting $(\xvecbar,\yvecbar,\zvecbar)$ be an optimal solution to problem (\ref{eqn:path_fluid}), we verify that the solution $(\xvecbar,\yvecbar)$ is feasible to the \ref{eqn:fluid}. The first two constraints in (\ref{eqn:path_fluid}) are identical to the first two constraints~in the \ref{eqn:fluid}, so we only check the last two constraints in the \ref{eqn:fluid}. The price path $\pvecbar_j^q$ is monotone over time, which implies that if this price path charges price level $k$ or larger at time period $t-1$, then it must also charge price level $k$ or larger at time period $t$, so we have $\ind{\pbar_{j,t-1}^q \geq k} \leq \ind{\pbar_{jt}^q \geq k}$. Using the fact that $(\xvecbar,\yvecbar,\zvecbar)$ satisfies the third constraint in (\ref{eqn:path_fluid}), we get $\sum_{k=\ell}^n \ybar_{jt}^k  = \sum_{q \in \Mcal} \zbar_j^q \ts \sum_{k=\ell}^n \ind{\pbar_{jt}^q = k} = \sum_{q \in \Mcal} \zbar_j^q \ts \ind{\pbar_{jt}^q \geq \ell} $ for all $t \in \Tcal$.~Thus, the last chain of equalities yields $\sum_{k=\ell}^n \ybar_{j,t-1}^k  = \sum_{q \in \Mcal} \zbar_j^q \ts {\bf 1}_{(\pbar_{j,t-1}^q \geq \ell)} \leq \sum_{q \in \Mcal} \zbar_j^q \ts \ind{\pbar_{jt}^q \geq \ell}  = \sum_{k=\ell}^n \ybar_{jt}^k$, verifying the third constraint in the \ref{eqn:fluid}. Using the fact that $(\xvecbar,\yvecbar,\zvecbar)$ satisfies the third constraint in (\ref{eqn:path_fluid}) once more, we get $\sum_{\ell \in \Ncal} \ybar_{jt}^\ell = \sum_{q \in \Mcal} \zbar_j^q \sum_{\ell \in \Ncal} \ind{\pbar_{jt}^q = \ell} = \sum_{q \in \Mcal} \zbar_j^q = 1$, where the second equality holds because $\sum_{\ell \in \Ncal} \ind{\pbar_{jt}^q = \ell} =1$ and the third equality holds by the fourth constraint in (\ref{eqn:path_fluid}).  Thus, the fourth constraint in the \ref{eqn:fluid} holds. Because the solution $(\xvecbar,\yvecbar)$ is feasible to the \ref{eqn:fluid} and this problem shares the same objective function with problem (\ref{eqn:path_fluid}), the optimal objective value of the \ref{eqn:fluid} is at least as large as that of problem (\ref{eqn:path_fluid}).  

On the other hand, letting $(\xvecbar, \yvecbar)$ be an optimal solution to the \ref{eqn:fluid}, we construct the random price path $\Prmvecbar_j$ as discussed at the beginning of Section \ref{sec:paths}. We verify that the solution $(\xvecbar,\yvecbar,\zvecbar)$ with $\zbar_j^q = \mathbb P \{ \Prmvecbar_j = \pvecbar_j^q\}$ is feasible to problem (\ref{eqn:path_fluid}). Once again, we only check the last two constraints in problem (\ref{eqn:path_fluid}). By the construction of the random price path $\Prmvecbar_j$ as discussed at the beginning of Section \ref{sec:paths}, we have $\mathbb P \{ \Prmbar_{jt} = \ell \} = \ybar_{jt}^\ell$, in which case, we get the chain of equalities given by  $\sum_{q \in \Mcal} \ind{\pbar_{jt}^q = \ell} \ts \zbar_j^q = \sum_{q \in \Mcal} \ind{\pbar_{jt}^q = \ell} \ts \mathbb P \{ \Prmvecbar_j = \pvecbar_j^q\} = \mathbb P \{ \Prmbar_{jt} = \ell\} = \ybar_{jt}^\ell$, where the second equality holds by computing the probability that the random price path $\Prmvecbar_j$ charges price level $\ell$ at time period~$t$ by considering all its realizations with price level $\ell$ at time period~$t$, verifying the third constraint in (\ref{eqn:path_fluid}). Furthermore, we have $\sum_{q \in \Mcal} \zbar_j^q = \sum_{q \in \Mcal}  \mathbb P \{ \Prmvecbar_j = \pvecbar_j^q\} = 1$ because $\{ \pvecbar_j^q :q \in \Mcal\}$ is the set of all possible realizations of the random price path $\Prmvecbar_j$, verifying  the fourth constraint in (\ref{eqn:path_fluid}). In this case, because the solution $(\xvecbar,\yvecbar,\zvecbar)$ is feasible to problem (\ref{eqn:path_fluid}) and this problem shares the same objective function with the \ref{eqn:fluid}, the optimal objective value of problem (\ref{eqn:path_fluid}) is at least as large as that of the \ref{eqn:fluid}. \qed

By Proposition \ref{pro:equivalence}, not only are the optimal objective values of the \ref{eqn:fluid} and problem (\ref{eqn:path_fluid}) equal, we can construct an optimal solution to one problem by using the other.

\vspace{-2mm}

\section{Proof of Proposition \ref{pro:extreme}}
\label{sec:extreme}

\vspace{-2mm}

The proof uses a sequence of lemmas. 
Using the vectors $\xvec = (x_1,\ldots,x_n)$ and $\zvec = (z_1,\ldots,z_m)$, for constants $b \geq 0$, $f_i \geq 0$ and $a_{ij} \geq 0$ for all $i=1,\ldots,n$, $j=1,\ldots,m$, consider the polytope
\begin{align}
\Pcal = \Bigg\{ (\xvec,\zvec) \in \mathbb R_+^{n+m} ~:~ \sum_{i=1}^n f_i \ts x_i = b,~~ x_i \leq \sum_{j=1}^m a_{ij} \ts z_j ~\forall \ts i=1,\ldots,n,~~\sum_{j=1}^m z_j = 1 \Bigg\}.
\label{eqn:poly}
\end{align}
We establish the following sequence of results. If $(\xvecbar,\zvecbar)$ is an extreme point of the polytope $\Pcal$ with $\sum_{j=1}^m a_{ij} \ts \zbar_j > 0$ for all $i=1,\ldots,n$, then there are at most two strictly positive components of the vector $\zvecbar$. This result is the key. Leveraging this result, it turns out that if $(\xvecbar,\zvecbar)$ is an extreme point of the polytope $\Pcal$, then there can be at most two strictly positive components of the vector~$\zvecbar$. In other words, we do not have to impose the condition that $\sum_{j=1}^m a_{ij} \ts z_j > 0$ for all $i = 1,\ldots,n$. Once we establish these results for the generic polytope $\Pcal$, we turn to the extreme point solutions to problem (\ref{eqn:path_fluid}). We use the additional decision variables \mbox{$(w_j : j \in \Jcal)$} to express the first constraint in (\ref{eqn:path_fluid})  as $\sum_{j \in \Jcal} a_{ij} \ts w_j \leq c_i$ for all $i \in \Lcal$ and $\sum_{t \in \Tcal} \sum_{\ell \in \Ncal} \lambda_{jt}^\ell \ts x_{jt}^\ell = w_j$ for all $j \in \Jcal$. If we fix the values of the decision variables $(w_j : j \in \Jcal)$, then the set of feasible solutions for problem (\ref{eqn:path_fluid}) decomposes by the products. Furthermore, the set of feasible solutions corresponding to each product has the same structure as the polytope $\Pcal$, in which case, the result in Proposition \ref{pro:extreme} will follow. 

\vspace{-2mm}

\begin{lem}[Polytope]
\label{lem:nonzero_extreme}
If $(\xvecbar,\zvecbar)$ is an extreme point of the polytope $\Pcal$ with $\sum_{j=1}^m a_{ij} \ts \zbar_j > 0$ for all $i=1,\ldots,n$, then there exist $\Fs, \Ss \in \{1,\ldots,m\}$ such that $\zbar_j = 0$ for all $j \in \{1,\ldots,m\} \setminus \{\Fs, \Ss\}$. 
\end{lem}

\vspace{-2mm}

\noindent{\it Proof:} Associating the slack decision variables $\uvec = (u_1,\ldots,u_n)$, we write the second constraint in the polytope $\Pcal$ as the equality constraint $x_i + u_i = \sum_{j=1}^m a_{ij} \ts z_j$.  Once we write the polytope $\Pcal$ with such an equality constraint,  we use $(\xvecbar,\zvecbar,\uvecbar)$ with $\ubar_i = \sum_{j=1}^m a_{ij} \ts \zbar_j - \xbar_i$ to denote the extreme point in the statement of the lemma. There are $n+2$ equality constraints and $2n + m$ decision variables in the definition of the polytope $\Pcal$ with the equality constraints, so there can be at most $n+2$ strictly positive components of the vector $(\xvecbar,\zvecbar,\uvecbar)$. We count the number of strictly positive components of the vector $(\xvecbar,\zvecbar,\uvecbar)$. Because $\sum_{j=1}^m a_{ij} \ts \zbar_j > 0$, to satisfy the constraint $\xbar_i + \ubar_i = \sum_{j=1}^m a_{ij} \ts \zbar_j$, at least one of the decision variables $\ubar_i$ and $\xbar_i$ has to be strictly positive. The number of such constraints in the definition of the polytope $\Pcal$ is $n$. Therefore, considering the vector $(\xvecbar,\uvecbar)$, at least $n$ of the components of this vector have to be strictly positive. In this case, noting that there can be at most $n+2$ strictly positive components of the vector $(\xvecbar,\zvecbar,\uvecbar)$, it follows that at most two of the components of the vector $\zvecbar$ can be strictly positive. \qed

The lemma above is the key result. We argue that the result in Lemma \ref{lem:nonzero_extreme} holds even if we do not impose the condition that $\sum_{j=1}^m a_{ij} \ts \zbar_j > 0$ for all $i=1,\ldots,n$. In particular, if $(\xvecbar,\zvecbar)$ is an extreme point of the polytope $\Pcal$, then there are at most two strictly positive components of the vector~$\zvecbar$, irrespective of the values of  $\sum_{j=1}^m a_{ij} \ts \zbar_j$ for $i=1,\ldots,n$. To see this result, we define \mbox{$\overline \Qcal = \{ i = 1,\ldots,n : \sum_{j=1}^m a_{ij} \ts \zbar_j > 0 \}$}. %, in which case, by the second constraint in (\ref{eqn:poly}), we have $\xbar_i = 0$ for $i \in \{1,\ldots,n\} \setminus \overline \Qcal$.
Using the vector $\yvec = (y_i : i \in \overline \Qcal)$, we consider the polytope given by $\overline \Pcal = \{ (\yvec,\zvec) \in \mathbb R_+^{|\overline \Qcal|+m}: \sum_{i \in \overline \Qcal} f_i \ts y_i = b,~ y_i \leq \sum_{j=1}^m a_{ij} \ts z_j ~\forall \ts i \in \overline \Qcal,~\sum_{j=1}^m z_j = 1\}$.~We can check that if $(\xvecbar,\zvecbar)$ is an extreme point of the polytope $\Pcal$, then $(\yvecbar,\zvecbar)$ with $\ybar_i = \xbar_i$ for all $i \in \overline \Qcal$ is an extreme point of the polytope $\overline \Pcal$. Also, we have $\sum_{j=1}^m a_{ij} \ts \zbar_j > 0$ for all $i \in \overline \Qcal$. Thus, we can use Lemma~\ref{lem:nonzero_extreme} for the extreme point $(\yvecbar,\zvecbar)$ of the polytope $\overline \Pcal$ to conclude that there are at most two strictly positive components of the vector $\zvecbar$, as desired. We give this result in the next lemma. 

\vspace{-2mm}

\begin{lem}[Generalized Polytope]
\label{lem:zero_extreme}
If $(\xvecbar,\zvecbar)$ is an extreme point of the polytope $\Pcal$, then there exist $\Fs, \Ss \in \{1,\ldots,m\}$ such that $\zbar_j = 0$ for all $j \in \{1,\ldots,m\} \setminus \{\Fs, \Ss\}$. 
\end{lem}

\vspace{-2mm}

\noindent{\it Proof:} Defining the set $\overline \Qcal$ and the polytope $\overline \Pcal$ as in the discussion just before the lemma, if we can verify that $(\yvecbar,\zvecbar)$ with $\ybar_i = \xbar_i$ for all $i \in \overline \Qcal$ is an extreme point of the polytope $\overline \Pcal$, then the result follows by the discussion just before the lemma. To get a contradiction, assume that $(\yvecbar,\zvecbar)$ is not an extreme point for the polytope $\overline \Pcal$, so there exist $(\yvechat,\zvechat), (\yvectilde,\zvectilde) \in \overline \Pcal$ with $\alpha \in (0,1)$ such that $(\yvecbar,\zvecbar) = \alpha \ts (\yvechat,\zvechat) + (1-\alpha) \ts (\yvectilde,\zvectilde)$. We define $\xvechat$ as $\xhat_i = \yhat_i$ for $i \in \overline \Qcal$ and $\xhat_i = 0$ for $i \not \in \overline \Qcal$. Using the fact that $(\yvechat,\zvechat) \in \overline \Pcal$, as well as $a_{ij} \geq 0$ so that $\sum_{j=1}^m a_{ij} \ts \zhat_j \geq 0$, we can check that $(\xvechat,\zvechat) \in \Pcal$. Similarly, we define $\xvectilde$ as $\xtilde_i = \ytilde_i$ for $i \in \overline \Qcal$ and $\xtilde_i = 0$ for $i \not \in \overline \Qcal$. By the same argument, we have $(\xvectilde,\zvectilde) \in \Pcal$. We have $\sum_{j=1}^m a_{ij} \ts \zbar_j = 0$ for $i \not \in \overline \Qcal$, so by the second constraint in (\ref{eqn:poly}), we get $\xbar_i = 0$ for $i \not \in \overline \Qcal$. In this case, we can check that having $(\yvecbar,\zvecbar) = \alpha \ts (\yvechat,\zvechat) + (1-\alpha) \ts (\yvectilde,\zvectilde)$ implies that $(\xvecbar,\zvecbar) = \alpha \ts (\xvechat,\zvechat) + (1-\alpha) \ts (\xvectilde,\zvectilde)$, contradicting the fact that $(\xvecbar,\zvecbar)$ is an extreme point of the polytope $\Pcal$. \qed

Our discussion so far focuses on the extreme points of the generic polytope $\Pcal$. We focus on the extreme point solutions to problem (\ref{eqn:path_fluid}). In problem (\ref{eqn:path_fluid}), we interpret $\sum_{t \in \Tcal} \sum_{\ell \in \Ncal} \lambda_{jt}^\ell \ts x_{jt}^\ell$ as the total expected sales for product $j$. Setting the total expected sales for product $j$ as $w_j$ for all $j \in \Jcal$, for fixed vector $\wvec = (w_j : j \in \Jcal)$, we define the polytope
\begin{align}
\Ycal(\wvec) ~=~ \Bigg\{& (\xvec,\zvec) \in \mathbb R_+^{|\Jcal| \ts (n \times T + |\Mcal|)} ~:~ 
 \sum_{t \in \Tcal}  \sum_{\ell \in \Ncal} \lambda_{jt}^\ell \ts x_{jt}^\ell = w_j \quad \forall \ts j \in \Jcal,
\nonumber
\\
&\qquad \qquad x_{jt}^\ell \leq  \sum_{q\in \Mcal} \ind{\pbar_{jt}^q = \ell} \ts z_j^q \quad \forall \ts j \in \Jcal,~\ell \in \Ncal,~t \in \Tcal, \quad
\sum_{q \in \mathcal M} z_j^q = 1 \quad \forall \ts j \in \Jcal \Bigg\}.
\label{eqn:dec_feasible}
\end{align}
In this case, we can express the set of feasible solutions to problem (\ref{eqn:path_fluid}) equivalently as the polytope $\Xcal = \{(\xvec,\wvec,\zvec) \in \mathbb R_+^{|\Jcal| \ts (n \times T + 1+|\Mcal|)}: \sum_{j \in \Jcal} a_{ij} \ts w_j \leq c_i~\forall \ts i \in \Lcal,~(\xvec,\zvec) \in \Ycal(\wvec) \}$. In particular, the first constraint in the definition of $\Ycal(\wvec)$, along with the constraint \mbox{$\sum_{j \in \Jcal} a_{ij} \ts w_j \leq c_i$} for all $i \in \Lcal$ in the definition of $\Xcal$, capture the first constraint in (\ref{eqn:path_fluid}), whereas the second and third constraints in the definition of $\Ycal(\wvec)$ capture the second, third and fourth constraints in (\ref{eqn:path_fluid}). Representing the set of feasible solutions to problem (\ref{eqn:path_fluid}) as the polytope $\Xcal$, the proof of Proposition \ref{pro:extreme} follows from the following two observations. First, we can check that if $(\xvecbar,\wvecbar,\zvecbar)$ is an extreme point of the polytope $\Xcal$, then $(\xvecbar,\zvecbar)$ is an extreme point of the polytope $\Ycal(\wvecbar)$. We can show this result by using a contradiction argument similar to the one in the proof of Lemma \ref{lem:zero_extreme}. We assume that $(\xvecbar,\wvecbar,\zvecbar)$ is an extreme point of $\Xcal$, but $(\xvecbar,\zvecbar$) is not an extreme point of the polytope $\Ycal(\wvecbar)$, in which case, we can express $(\xvecbar,\wvecbar,\zvecbar)$ as a non-trivial convex combination of two points in $\Xcal$. Because the outline is similar to the one in the proof of Lemma \ref{lem:zero_extreme}, we skip it. 
Second, the set of feasible solutions given by the polytope $\Ycal(\wvec)$ decomposes by the products. Fixing product $j$, identifying the decision variables $(x_{jt}^\ell : \ell \in \Ncal,~t \in \Tcal)$ and $(z_j^q : q \in \Mcal)$ in (\ref{eqn:dec_feasible}), respectively, with the decision variables \mbox{$(x_i : i =1,\ldots,n)$ and $(z_j : j =1,\ldots,m)$} in (\ref{eqn:poly}), as well as identifying the constants \mbox{$(\lambda_{jt}^\ell : \ell \in \Ncal,~t \in \Tcal)$} and \mbox{$(\ind{\pbar_{jt}^q = \ell} : \ell \in \Ncal,~t \in \Tcal,~q \in \Mcal)$} in (\ref{eqn:dec_feasible}) with the constants $(f_i : i =1,\ldots,n)$ and $(a_{ij} : i =1,\ldots,n,~j=1,\ldots,m)$ in (\ref{eqn:poly}), the polytope in (\ref{eqn:dec_feasible}) has the same structure as the polytope in (\ref{eqn:poly}). Thus, for any fixed $\wvecbar$, by Lemma \ref{lem:zero_extreme}, if $(\xvecbar,\zvecbar)$ is an extreme point of $\Ycal(\wvecbar)$, then there can be at most two strictly positive components of the vector $(\zbar_j^q : q \in \Mcal)$. We put these observations together to give an explicit proof for Proposition \ref{pro:extreme}.
%Putting the two observations in the previous paragraph together establishes Proposition \ref{pro:extreme}. We recap the main steps to give an explicit proof for Proposition \ref{pro:extreme}.

%We put these observations together to give a proof for Proposition \ref{pro:extreme}.

{\bf \underline{Proof of Proposition \ref{pro:extreme}}:}
\\
\indent The polytope $\Xcal$ defined just after (\ref{eqn:dec_feasible}) is equivalent to the set of feasible solutions for problem (\ref{eqn:path_fluid}).~In this case, by the first observation just after (\ref{eqn:dec_feasible}), if $(\xvecbar,\wvecbar,\zvecbar)$ is an extreme point of the polytope $\Xcal$, then $(\xvecbar,\zvecbar)$ is an extreme point of the polytope $\Ycal(\wvecbar)$. Furthermore, by the second observation just after (\ref{eqn:dec_feasible}), considering the extreme point $(\xvecbar,\zvecbar)$ for the polytope $\Ycal(\wvecbar)$, for each $j \in \Jcal$, there are at most two strictly positive components of the vector $(\zbar_j^q : q \in \Mcal)$. \qed

The polytope in (\ref{eqn:poly}) is similar to the knapsack polytope, but the upper bound on the decision variable $x_i$ is determined by a convex combination of the parameters $(a_{ij} : j = 1,\ldots,m)$.

\section{Proof of Theorem \ref{thm:perf}}
\label{sec:perf}

We give a proof for Theorem \ref{thm:perf}. The first inequality in the theorem follows by noting that \mbox{$Z_\lp^* \geq \opt$}, so we focus on showing the second inequality. We define the Bernoulli random variable $\Grm_{jt}$ such that $\Grm_{jt} = 1$ if and only if we have remaining resource capacities to make product $j$ available at time period $t$ under the approximate policy. Letting $\betabar_{jt} = \thetabar_{jt} \ts \Cbar_j / \max\{ \capa_j^\Fs , \Cbar_j\}$ for notational brevity, by the definition of the approximate policy, if we have the remaining resource capacities to make product~$j$ available for purchase at time period~$t$, then we make the product available for purchase with probability $\gamma \ts \betabar_{jt}$. For the approximate policy to collect revenue from product $j$ at time period $t$, we need to have remaining capacities for the resources to make the product available for purchase, we need to make the product available and we need to make a sale for the product at the price $\pbar_{jt}^\Fs$ chosen by the approximate policy. Therefore, the expected revenue of the approximate policy is given by \mbox{$\apx = \sum_{t \in \Tcal} \sum_{j \in \Jcal} \mathbb P \{ \Grm_{jt} = 1\} \ts \gamma \ts \betabar_{jt} \sum_{\ell \in \Ncal} \ind{\pbar_{jt}^\Fs = \ell} \ts \lambda_{jt}^\ell \ts r_j^\ell$}.~We proceed to lower bounding the total expected revenue of the approximate policy with a certain fraction of the optimal objective value of the \ref{eqn:fluid}.

Noting that $\mathbb P \{ \Grm_{jt} = 1\}$ is the probability that we have remaining resource capacities to make product $j$ available for purchase at time period $t$, for some function $\avail : [0,1] \rightarrow [0,1]$, we show that if we can lower bound this availability probability as $\mathbb P \{ \Grm_{jt} = 1\} \geq \avail(\gamma)$ as a function of the tuning parameter, then we can lower bound the total expected revenue of the approximate policy as $\apx \geq \frac 12 \ts \gamma \ts \avail(\gamma) \ts Z_\lp^*$.
We give a useful inequality to show this result. We have $\zbar_j^q = 0$ for all $q \in \Mcal \setminus \{ \Fs, \Ss\}$ by Proposition \ref{pro:extreme}. By the discussion at the beginning of Section~\ref{sec:policy}, we have the identity $\Cbar_j = \sum_{q \in \Mcal} \zbar_j^q \ts \capa_j^q$, so we obtain $\Cbar_j = \zbar_j^\Fs \ts \capa_j^\Fs +  \zbar_j^\Ss \ts \capa_j^\Ss$.~In this case, the last equality yields $\Cbar_j \geq \zbar_j^\Fs \ts \capa_j^\Fs$, which is equivalent to $\zbar_j^\Fs \leq \Cbar_j / \capa_j^\Fs$, yielding an upper bound on the probability of one of the price paths. Using the fact that $\zbar_j^\Fs \in [0,1]$, we write the last inequality as $\zbar_j^\Fs \leq \min\{ \Cbar_j / \capa_j^\Fs , 1\} = \Cbar_j / \max\{ \Cbar_j , \capa_j^\Fs\}$. Using the same argument, we also obtain the inequality $\zbar_j^\Ss \leq \Cbar_j / \max\{ \Cbar_j , \capa_j^\Ss\}$. 
In the next lemma, we build on these upper bounds on the probabilities $\zbar_j^\Fs$ and $\zbar_j^\Ss$ to show that we can lower bound the total expected revenue of the approximate policy by using a lower bound on the availability probability. 

\vspace{-1mm}

\begin{lem}[Lower Bound on Policy Performance]
\label{lem:apx_lp}
For some $\avail : [0,1] \rightarrow [0,1]$, if the availability probability satisfies $\mathbb P \{ \Grm_{jt} = 1\} \geq \avail(\gamma)$ for all $j \in \Jcal$ and $t \in \Tcal$, then we have
\begin{align*}
\apx \geq \frac 12 \ts \gamma \ts \avail(\gamma) \ts Z_\lp^*.
\end{align*}
\end{lem}

\vspace{-1mm}

\noindent{\it Proof:} By the discussion at the beginning of Section \ref{sec:policy}, we have $\Rbar_j = \zbar_j^\Fs \ts \reve_j^\Fs + \zbar_j^\Ss \ts \reve_j^\Ss$. Thus, the inequalities just before the lemma yield $\Rbar_j \leq \Cbar_j \ts ( \reve_j^\Fs / \max\{\Cbar_j , \capa_j^\Fs\} + \reve_j^\Ss / \max\{\Cbar_j , \capa_j^\Ss\})$.~We index the price paths $\Fs, \Ss$ such that $\reve_j^\Fs / \max\{ \capa_j^\Fs , \Cbar_j\}  \geq \reve_j^\Ss / \max\{ \capa_j^\Ss , \Cbar_j\}$, so we obtain \mbox{$\Rbar_j \leq 2 \ts \Cbar_j \reve_j^\Fs /  \max\{ \Cbar_j , \capa_j^\Fs\}$}. The expression for $\apx$ at the beginning of this section yields
\begin{align}
\apx 
~& =~ 
\sum_{t \in \Tcal} \sum_{j \in \Jcal} \mathbb P \{ \Grm_{jt} = 1\} \ts \gamma \ts \betabar_{jt} \sum_{\ell \in \Ncal} \ind{\pbar_{jt}^\Fs = \ell} \ts \lambda_{jt}^\ell \ts r_j^\ell
~\geq~
\gamma \ts \avail(\gamma) \sum_{t \in \Tcal} \sum_{j \in \Jcal} \betabar_{jt} \sum_{\ell \in \Ncal} \ind{\pbar_{jt}^\Fs = \ell} \ts \lambda_{jt}^\ell \ts r_j^\ell
\nonumber
\\
~&\stackrel{(a)} =~
\gamma \ts \avail(\gamma) \sum_{j \in \Jcal} \frac{\Cbar_j}{\max\{ \capa_j^\Fs , \Cbar_j\}} \sum_{t \in \Tcal}  \sum_{k \in \Ncal}  \ts \ind{\pbar_{jt}^\Fs = k} \ts  \frac{\xbar_{jt}^k}{\ybar_{jt}^k}  \sum_{\ell \in \Ncal} \ind{\pbar_{jt}^\Fs = \ell} \ts \lambda_{jt}^\ell \ts r_j^\ell
\nonumber
\\
~&\stackrel{(b)}=~
\gamma \ts \avail(\gamma) \sum_{j \in \Jcal} \frac{\Cbar_j}{\max\{ \capa_j^\Fs , \Cbar_j\}} \sum_{t \in \Tcal} \sum_{\ell \in \Ncal}  \ind{\pbar_{jt}^\Fs = \ell} \ts \lambda_{jt}^\ell \ts r_j^\ell \ts  \frac{\xbar_{jt}^\ell}{\ybar_{jt}^\ell}
~\stackrel{(c)}=~
\gamma \ts \avail(\gamma) \ts \sum_{j \in \Jcal} \frac{\Cbar_j}{\max\{ \capa_j^\Fs , \Cbar_j\}} \ts \reve_j^\Fs
\nonumber
\\
~&\geq~ 
\frac{1}{2} \gamma \ts \avail(\gamma) \ts \sum_{j \in \Jcal} \Rbar_j 
~\stackrel{(d)}=~ 
\frac{1}{2} \ts \gamma \ts \avail(\gamma) \ts Z_\lp^*,
\label{eqn:rev_lb}
\end{align}
where $(a)$ is by the definition of $\betabar_{jt}$, $(b)$ holds because  $\ind{\pbar_{jt}^\Fs = k}  \ts  \ind{\pbar_{jt}^\Fs = \ell} =1$ if and only if $k=\ell$, $(c)$ uses the definition of $\reve_j^\Fs$ and $(d)$ holds because we have $\sum_{j \in \Jcal} \Rbar_j = Z_\lp^*$ by the definition of $\Rbar_j$. \qed

If the availability probabilities are lower bounded by $\avail(\gamma)$, then we can lower bound $\apx / Z_\lp^*$ by $\frac 12 \ts \gamma \ts \avail(\gamma)$. We give a specific expression that lower bounds the availability probabilities.

{\bf \underline{Lower Bounding the Availability Probabilities}:}
\\
\indent We give an explicit expression for the function $\avail: [0,1] \rightarrow [0,1]$ such that \mbox{$\mathbb P \{ \Grm_{jt} = 1\} \geq \avail(\gamma)$}. We consider a  policy, which we refer to as the inventory agnostic policy, by following the setup for our approximate policy, but making a product available without checking whether we have remaining resource capacities. If a customer chooses to purchase a product without remaining resource capacities to satisfy the product request, then the customer leaves without a purchase. In particular, setting $\betabar_{jt} = \thetabar_{jt} \ts \Cbar_j / \max\{ \capa_j^\Fs , \Cbar_j\}$ with $\thetabar_{jt} = \sum_{\ell \in \Ncal} \ind{\pbar_{jt}^\Fs = \ell} \frac{\xbar_{jt}^\ell}{\ybar_{jt}^\ell}$, under the inventory agnostic policy, we always make product $j$ available at time period $t$ with probability $\gamma \ts \betabar_{jt}$ without checking the remaining resource capacities. If we make product $j$ available, then we charge the price $\pbar_{jt}^\Fs$ as in our approximate policy. If the customer chooses to purchase product $j$, then we collect the revenue from product $j$ corresponding to the price we charge only when we have   remaining capacities  to satisfy the product request. Under the inventory agnostic policy, for resource $i$ to have a demand for its capacity at time period $t$, we need to make a product that uses resource $i$ available for purchase and the customer needs to choose to purchase the product. Thus, defining the Bernoulli random variable $\Nrm_{it}$ with parameter $\sum_{j \in \Jcal} a_{ij} \ts \gamma \ts \betabar_{jt} \sum_{\ell \in \Ncal} \ind{\pbar_{jt}^\Fs=\ell} \ts \lambda_{jt}^\ell$, the demand for the capacity of resource $i$ at time period $t$ is given by the random variable $\Nrm_{it}$. Even if we have a demand for the capacity of resource $i$ at time period $t$, we may not consume the capacity of the resource because some other resource used by the requested product may not have  capacity. 

The key point is that the random variable $\Nrm_{it}$ upper bounds the capacity consumption of resource~$i$ at time period $t$ under our approximate policy. In particular, under the approximate policy, a product using resource $i$ consumes the capacity of the resource when all of the resources used by the product have remaining capacities. Under the inventory agnostic policy, however, a product using resource $i$ imposes a demand for the capacity of the resource regardless of the remaining capacity of other resources. Thus, having $\sum_{\tau = 1}^{t-1} \Nrm_{i\tau} < c_i$ is sufficient to have remaining capacity of resource $i$ at time period $t$ under our approximate policy. To have remaining resource capacities to make product $j$ available at time period $t$, we need to have remaining capacity for each resource used by product $j$. In this case, letting $\Acal_j = \{ i \in \Lcal: a_{ij} = 1\}$ to denote the set of resources used by product $j$, we have $\mathbb P \{ \Grm_{jt} =1 \} \geq \mathbb P \{ \sum_{\tau=1}^{t-1} \Nrm_{i\tau} < c_i~\forall \ts i \in \Acal_j\} \geq \mathbb P \{ \sum_{\tau \in \Tcal} \Nrm_{i\tau} < c_i~\forall \ts i \in \Acal_j\}$, which implies that a lower bound on the probability $\mathbb P \{ \sum_{\tau \in \Tcal} \Nrm_{i\tau} < c_i~\forall \ts i \in \Acal_j\}$ is also a lower bound on the probability $\mathbb P \{ \Grm_{jt} =1 \}$. We give a useful inequality to lower bound the former probability. The random variable $\Nrm_{it}$ is Bernoulli with parameter $\sum_{j \in \Jcal} a_{ij} \ts \gamma \ts \betabar_{jt} \sum_{\ell \in \Ncal} \ind{\pbar_{jt}^\Fs=\ell} \ts \lambda_{jt}^\ell$. We have 
\begin{align}
\!\!\sum_{t \in \Tcal} \betabar_{jt} \sum_{\ell \in \Ncal}  \ind{\pbar_{jt}^\Fs=\ell} \ts \lambda_{jt}^\ell
~&=~
\frac{\Cbar_j}{\max\{ \capa_j^\Fs , \Cbar_j\}} \sum_{t \in \Tcal}  \sum_{k \in \Ncal}  \ts \ind{\pbar_{jt}^\Fs = k} \ts  \frac{\xbar_{jt}^k}{\ybar_{jt}^k}  \sum_{\ell \in \Ncal} \ind{\pbar_{jt}^\Fs = \ell} \ts \lambda_{jt}^\ell
\nonumber
 \\
~&=~
\frac{\Cbar_j}{\max\{ \capa_j^\Fs , \Cbar_j\}} \sum_{t \in \Tcal}  \sum_{\ell \in \Ncal}      \ind{\pbar_{jt}^\Fs = \ell} \ts \lambda_{jt}^\ell \ts \frac{\xbar_{jt}^\ell}{\ybar_{jt}^\ell}
~=~
\frac{\Cbar_j}{\max\{ \capa_j^\Fs , \Cbar_j\}} \ts \capa_j^\Fs 
~\leq~\Cbar_j.\!\!\!\!
\label{eqn:avail_param}
\end{align}

In the chain of inequalities above, the first equality uses the definition of $\betabar_{jt}$, the second equality holds because  $\ind{\pbar_{jt}^\Fs = k}  \ts  \ind{\pbar_{jt}^\Fs = \ell} =1$ if and only if $k=\ell$ and the third equality uses the definition of $\capa_j^\Fs$ in  (\ref{eqn:path_rev}). Thus, setting $\alpha_{it} = \sum_{j \in \Jcal} a_{ij} \ts \gamma \ts \betabar_{jt} \sum_{\ell \in \Ncal} \ind{\pbar_{jt}^\Fs=\ell} \ts \lambda_{jt}^\ell$ to capture the parameter of the Bernoulli random variable $\Nrm_{it}$, by the chain of inequalities above, we get $\sum_{t \in \Tcal} \alpha_{it} \leq \gamma \ts \sum_{j \in \Jcal} a_{ij} \ts \Cbar_j \leq \gamma \ts  c_i$, where the last inequality follows from the definition of $\Cbar_j$ as discussed at the beginning of Section~\ref{sec:policy}. In this case, we get $\sum_{t \in \Tcal} \mathbb E\{ \Nrm_{it} \} = \sum_{t \in \Tcal} \alpha_{it} \leq \gamma \ts c_i$. Furthermore, because the inventory agnostic policy makes its product availability decisions without paying attention to the remaining resource capacities, the demands on the capacities of a resource at different time periods are independent. Therefore, we also get $\text{\sf Var}(\sum_{t \in \Tcal} \Nrm_{it}) = \sum_{t \in \Tcal} \alpha_{it} \ts (1-\alpha_{it}) \leq \sum_{t \in \Tcal} \alpha_{it} \leq \gamma \ts c_i$. In the next lemma, we use the upper bounds on the two moments of the random variable $\sum_{t \in \Tcal} \Nrm_{it}$ to give a lower bound on the availability probability. Recall that $L = \max_{j \in \Jcal} \sum_{i \in \Lcal} a_{ij}$ and $c_{\min} = \min_{i \in \Lcal} c_i$.    

\begin{lem}
[Lower Bound on Availability Probability] 
\label{lem:avail_prob}
For all $j \in \Jcal$ and $t \in \Tcal$, we can lower bound the availability probability for product $j$ at time period $t$ as
\begin{align*}
\mathbb P\{ \Grm_{jt} = 1\} ~\geq~ \max\Bigg\{ 1 - L \ts \gamma , 1 - L \exp\Bigg( - \frac{(1-\gamma)^2 \ts c_{\min}}{2}\Bigg) \Bigg\}.
\end{align*}
\end{lem}

\noindent{\it Proof:} By the discussion just before the lemma, we have $\sum_{t \in \Tcal} \mathbb E\{ \Nrm_{it} \} \leq \gamma \ts c_i$. Thus, by the Markov inequality, we get $\mathbb P\{ \sum_{t \in \Tcal} \Nrm_{it} \geq c_i\} \leq \sum_{t \in \Tcal} \mathbb E\{ \Nrm_{it} \} / c_i \leq \gamma$ for all $i \in \Lcal$. Noting that $|\Acal_j| \leq L$, using the union bound in the last inequality yields $\mathbb P \{ \sum_{t \in \Tcal} \Nrm_{it} \geq c_i \mbox{ for some $i \in \Acal_j$}\} \leq L \ts \gamma$, which is equivalent to $\mathbb P \{ \sum_{t \in \Tcal} \Nrm_{it} < c_i ~\forall \ts i \in \Acal_j\} \geq 1 - L \ts \gamma$. Also, we have the chain of inequalities %Using $\mathbb E\{ \Nrm_{it} \} \leq \gamma \ts c_i$ once more, we have 
\begin{align*}
\mathbb P \Bigg\{ \sum_{t \in \Tcal} \Nrm_{it} \geq c_i \Bigg\}
& \ts \stackrel{(a)}\leq \ts
\mathbb P \Bigg\{ \sum_{t \in \Tcal} [\Nrm_{it} - \mathbb E\{ \Nrm_{it} \}] \geq (1-\gamma) \ts c_i \Bigg\}
\ts \stackrel{(b)} \leq \ts 
\exp\Bigg(- \frac{\frac12 \ts (1-\gamma)^2 \ts c_i^2}{\text{\sf Var}(\sum_{t \in \Tcal} \Nrm_{it}) + \frac 13 \ts (1-\gamma) \ts c_i} \Bigg)
\\
& \stackrel{(c)} \leq \ts
\exp\Bigg(- \frac{\frac12 \ts (1-\gamma)^2 \ts c_i^2}{\gamma \ts c_i + \frac 13 \ts (1-\gamma) \ts c_i} \Bigg)
\ts \stackrel{(d)} \leq \ts
\exp\Bigg(- \frac{\frac12 \ts (1-\gamma)^2 \ts c_i^2}{c_i} \Bigg)
\ts \stackrel{(e)} \leq \ts
\exp\Bigg( - \frac{(1-\gamma)^2 \ts c_{\min}}{2}\Bigg),
\end{align*}
where $(a)$ uses the fact that $\mathbb E\{ \Nrm_{it} \} \leq \gamma \ts c_i$, $(b)$ is the one-sided Bernstein inequality, $(c)$ follows because we have $\text{\sf Var}(\sum_{t \in \Tcal} \Nrm_{it}) \leq \gamma \ts c_i$ by the discussion just before the lemma, $(d)$ uses the inequality $\gamma + \frac 13(1-\gamma) = \frac 23 \ts \gamma + \frac 13 \leq 1$ and $(e)$ follows by noting that $c_{\min} \leq c_i$. Using $g(\gamma)$ to denote the expression on the right side of the chain of inequalities above as a function of $\gamma$, using the union bound once more, we obtain $\mathbb P \{ \sum_{t \in \Tcal} \Nrm_{it} \geq c_i \mbox{ for some $i \in \Acal_j$}\} \leq L \ts g(\gamma)$, which is equivalent to having $\mathbb P \{ \sum_{t \in \Tcal} \Nrm_{it} < c_i ~\forall \ts i \in \Acal_j\} \geq 1 - L \ts g(\gamma)$. Noting also that  $\mathbb P \{ \sum_{t \in \Tcal} \Nrm_{it} < c_i ~\forall \ts i \in \Acal_j\} \geq 1 - L \ts \gamma$, as well as using the fact that $\mathbb P \{ \Grm_{jt} = 1\} \geq \mathbb P \{ \sum_{t \in \Tcal} \Nrm_{it} < c_i ~\forall \ts i \in \Acal_j\}$ by the discussion just before (\ref{eqn:avail_param}), we obtain $\mathbb P \{ \Grm_{jt} = 1\} \geq \mathbb P \{ \sum_{t \in \Tcal} \Nrm_{it} < c_i ~\forall \ts i \in \Acal_j\} \geq \max\{1 - L \ts \gamma , 1 - L \ts g(\gamma) \}$. \qed

In Lemma \ref{lem:apx_lp}, we lower bound the performance of the approximate policy. In Lemma \ref{lem:avail_prob}, we lower bound the availability probability. Using these results, we give a proof of Theorem \ref{thm:perf}.

{\bf \underline{Proof of Theorem \ref{thm:perf}\phantom{p}\!\!\!}:}
\\
\indent Consider the approximate policy with the tuning parameter $\gamma = \frac{1}{2L}$.  By Lemma \ref{lem:avail_prob}, if we set $\avail(\gamma) = 1 - L \ts \gamma$, then we have $\mathbb P \{ \Grm_{jt} = 1\} \geq \avail(\gamma)$. In this case, choosing the tuning parameter as $\gamma = \frac{1}{2L}$, we have $\avail(\gamma) = \frac{1}{2}$. Using this value of the tuning parameter in Lemma \ref{lem:apx_lp}, we obtain  $\frac{\apx}{Z_\lp^*} \geq \frac 12 \times \frac{1}{2L} \times \frac12 = \frac{1}{8L}$. Also, consider the approximate policy with the tuning parameter $\gamma = 1 - \sqrt{\frac{2 \log c_{\min}}{c_{\min}}}$. By Lemma \ref{lem:avail_prob}, if we set $\avail(\gamma) = 1 - L \ts \exp( -\frac 12 (1-\gamma)^2 \ts c_{\min})$, then we have $\mathbb P \{ \Grm_{jt} = 1\} \geq \avail(\gamma)$. In this case, choosing the tuning parameter as $\gamma = 1 - \sqrt{\frac{2 \log c_{\min}}{c_{\min}}}$ and noting that $(1-\gamma)^2 = \frac{2 \ts \log c_{\min}}{c_{\min}}$, we have $\avail(\gamma) = 1 - \frac{L}{c_{\min}}$. Using this value of the tuning parameter in Lemma \ref{lem:apx_lp}, we obtain $\frac{\apx}{Z_\lp^*} \geq \frac{1}{2} \ts (1 - \sqrt{\frac{2 \log c_{\min}}{c_{\min}}}) \ts (1 - \frac{L}{c_{\min}}) \geq \frac{1}{2} - \sqrt{\frac{\log c_{\min}}{2 \ts c_{\min}}} - \frac{L}{c_{\min}}$. \qed

There are two inequalities that drive the proof of Theorem \ref{thm:perf}. First, as given at the beginning of the proof of Lemma \ref{lem:apx_lp}, we have $\Rbar_j \leq 2 \ts \Cbar_j \reve_j^\Fs /  \max\{ \Cbar_j , \capa_j^\Fs\}$. Second, as given in (\ref{eqn:avail_param}), we have $\sum_{t \in \Tcal} \betabar_{jt} \sum_{\ell \in \Ncal}  \ind{\pbar_{jt}^\Fs=\ell} \ts \lambda_{jt}^\ell \leq \Cbar_j$. The first inequality ensures that the prices in the price path $\Fs$ result in large enough total expected revenue so that we can compare the total expected revenue of the approximate policy with the optimal objective value of the \ref{eqn:fluid}. The second inequality ensures that  the prices in the price path $\Fs$ result in small enough total expected capacity consumption so that we can bound the availability probabilities. The fraction $\Cbar_j / \max\{ \capa_j^\Fs , \Cbar_j\}$ in the probability of making product $j$ available for purchase balances the two goals.

\section{Proof of Theorem \ref{thm:expost}}
\label{sec:expost}

We define the Bernoulli random variable $\Grm_{jt}$ such that $\Grm_{jt} = 1$ if and only if we have remaining resource capacities to make product $j$ available at time period $t$ under the ex-post approximate policy. Letting $\zetabar_{jt} = \thetabarp_{jt} \ts c_{\min}/(c_{\min} + \Deltabar)$ with $\thetabarp_{jt} = \sum_{\ell \in \Ncal} \ind{\pbar_{jt}^\Fs = \ell} \frac{\xbar_{jt}^\ell}{\ybar_{jt}^\ell}$ for notational brevity, if there are remaining resource capacities to make product $j$ available for purchase at time period $t$, then the ex-post approximate policy makes the  product available with probability $\gamma \ts \zetabar_{jt}$. In this case, by the same argument at the beginning of Appendix \ref{sec:perf}, using $\apxp$ to denote the total expected revenue of the ex-post approximate policy, we have $\apxp = \sum_{t \in \Tcal} \sum_{j \in \Jcal} \mathbb P \{ \Grm_{jt} = 1\} \ts \gamma \ts \zetabar_{jt} \sum_{\ell \in \Ncal} \ind{\pbar_{jt}^\Fs = \ell} \ts \lambda_{jt}^\ell \ts r_j^\ell$. %In the next lemma, we show that if we have $\mathbb P \{ \Grm_{jt} = 1\} \geq \avail(\gamma)$ for all $j \in \Jcal$ and $t \in \Tcal$, then we can relate the total expected revenue of the approximate policy to the optimal objective value of the \ref{eqn:fluid}. This lemma is the analogue of Lemma \ref{lem:apx_lp}.
In the next lemma, we give an analogue of Lemma \ref{lem:apx_lp} for the ex-post approximate policy. 

\begin{lem}[Lower Bound on the Ex-Post Performance]
\label{lem:apx_lp_expost}
For some $\availp : [0,1] \rightarrow [0,1]$, if the availability probability satisfies $\mathbb P \{ \Grm_{jt} = 1\} \geq \availp(\gamma)$ for all $j \in \Jcal$ and $t \in \Tcal$, then we have  
\begin{align*}
\apxp \ts \geq \ts  \frac{c_{\min}}{c_{\min} + \Deltabar} \ts \gamma \ts \availp(\gamma) \ts Z_\lp^*.
\end{align*}
\end{lem}

\noindent{\it Proof:} Using the expression for the total expected revenue of the ex-post policy just before the lemma and lower bounding $\mathbb P \{ \Grm_{jt}=1\}$ with $\availp(\gamma)$, we have the chain of inequalities
\begin{align*}
\apxp 
~& =~ 
\sum_{t \in \Tcal} \sum_{j \in \Jcal} \mathbb P \{ \Grm_{jt} = 1\} \ts \gamma \ts \zetabar_{jt} \sum_{\ell \in \Ncal} \ind{\pbar_{jt}^\Fs = \ell} \ts \lambda_{jt}^\ell \ts r_j^\ell
~\geq~
\gamma \ts \availp(\gamma) \sum_{t \in \Tcal} \sum_{j \in \Jcal} \zetabar_{jt} \sum_{\ell \in \Ncal} \ind{\pbar_{jt}^\Fs = \ell} \ts \lambda_{jt}^\ell \ts r_j^\ell
\nonumber
\\
~&\stackrel{(a)} =~
 \frac{c_{\min}}{c_{\min} + \Deltabar} \ts \gamma \ts \availp(\gamma) \sum_{j \in \Jcal}  \sum_{t \in \Tcal}  \sum_{k \in \Ncal}  \ts \ind{\pbar_{jt}^\Fs = k} \ts  \frac{\xbar_{jt}^k}{\ybar_{jt}^k}  \sum_{\ell \in \Ncal} \ind{\pbar_{jt}^\Fs = \ell} \ts \lambda_{jt}^\ell \ts r_j^\ell
\nonumber
\\
~&=~
\frac{c_{\min}}{c_{\min} + \Deltabar} \ts \gamma \ts \availp(\gamma) \sum_{j \in \Jcal}  \sum_{t \in \Tcal} \sum_{\ell \in \Ncal}  \ind{\pbar_{jt}^\Fs = \ell} \ts \lambda_{jt}^\ell \ts r_j^\ell \ts  \frac{\xbar_{jt}^\ell}{\ybar_{jt}^\ell}
~=~
\frac{c_{\min}}{c_{\min} + \Deltabar} \ts \gamma \ts \availp(\gamma) \ts \sum_{j \in \Jcal} \reve_j^\Fs
\nonumber
\\
~&\stackrel{(b)}\geq~ 
\frac{c_{\min}}{c_{\min} + \Deltabar} \ts \gamma \ts \availp(\gamma) \ts \sum_{j \in \Jcal} \Rbar_j 
~=~ 
\frac{c_{\min}}{c_{\min} + \Deltabar} \ts \gamma \ts \availp(\gamma) \ts Z_\lp^*,
\end{align*}
where $(a)$ is by the definition of $\zetabar_{jt}$ and $(b)$ holds because $\Rbar_j = \zbar_j^\Fs \ts \reve_j^\Fs + \zbar_j^\Ss \ts \reve_j^\Ss$ with $\zbar_j^\Fs + \zbar_j^\Ss = 1$, so noting that we index the price paths such that $\reve_j^\Fs \geq \reve_j^\Ss$, we obtain $\reve_j^\Fs \geq \Rbar_j \geq \reve_j^\Ss$. \qed

The chain of inequalities in the proof is the analogue of (\ref{eqn:rev_lb}) for the ex-post approximate policy. %Note that the fraction $1/2$ on the right side of the inequality in Lemma \ref{lem:apx_lp} is replaced with $c_{\min} / (c_{\min} + \Deltabar)$ in Lemma \ref{lem:apx_lp_expost}. 
Similar to the development in Appendix \ref{sec:perf}, we consider an inventory agnostic policy that makes its decisions just as the ex-post approximate policy, but without checking whether we have remaining resource capacities. If a customer chooses to purchase a product without remaining capacities, then she leaves without a purchase.  Defining the Bernoulli random variable $\Nrm_{it}$ with parameter  $\sum_{j \in \Jcal} a_{ij} \ts \gamma \ts \zetabar_{jt} \sum_{\ell \in \Ncal} \ind{\pbar_{jt}^\Fs=\ell} \ts \lambda_{jt}^\ell$, the demand for the capacity of resource $i$ at time period $t$ under the inventory agnostic policy is given by the random variable $\Nrm_{it}$. The random variable $\Nrm_{it}$ upper bounds the capacity consumption of resource $i$ at time period $t$ under the ex-post approximate policy. Because \mbox{$\Cbar_j = \zbar_j^\Fs \capa_j^\Fs + \zbar_j^\Ss \ts \capa_j^\Ss$}, we get $\capa_j^\Fs = \Cbar_j + \zbar_j^\Ss \ts (\capa_j^\Fs - \capa_j^\Ss) \leq \Cbar_j + [\capa_j^\Fs - \capa_j^\Ss]^+$.
In this case, we obtain the chain of inequalities
\begin{align}
& \sum_{t \in \Tcal} \zetabar_{jt} \sum_{\ell \in \Ncal}  \ind{\pbar_{jt}^\Fs=\ell} \ts \lambda_{jt}^\ell
~=~
\frac{c_{\min}}{c_{\min} + \Deltabar} \sum_{t \in \Tcal}  \sum_{k \in \Ncal}  \ts \ind{\pbar_{jt}^\Fs = k} \ts  \frac{\xbar_{jt}^k}{\ybar_{jt}^k}  \sum_{\ell \in \Ncal} \ind{\pbar_{jt}^\Fs = \ell} \ts \lambda_{jt}^\ell
\nonumber
 \\
~&\qquad =~
\frac{c_{\min}}{c_{\min} + \Deltabar} \sum_{t \in \Tcal}  \sum_{\ell \in \Ncal}      \ind{\pbar_{jt}^\Fs = \ell} \ts \lambda_{jt}^\ell \ts \frac{\xbar_{jt}^\ell}{\ybar_{jt}^\ell}
~=~
\frac{c_{\min}}{c_{\min} + \Deltabar}\ts \capa_j^\Fs 
~\leq~\frac{c_{\min}}{c_{\min} + \Deltabar} \ts (\Cbar_j + [\capa_j^\Fs - \capa_j^\Ss]^+).
\label{eqn:avail_param_expost}
\end{align}
The chain of inequalities above is the analogue of (\ref{eqn:avail_param}) for the ex-post approximate policy. We have $\sum_{j \in \Jcal} a_{ij} \ts (\Cbar_j + [\capa_j^\Fs - \capa_j^\Ss]^+) \leq \sum_{j \in \Jcal} a_{ij} \ts \Cbar_j + \Deltabar \leq c_i + \Deltabar \leq c_i \ts (1 + \frac{\Deltabar}{c_{\min}})$, where the second inequality uses $\sum_{j \in \Jcal} a_{ij} \ts \Cbar_j \leq c_i$ by the discussion at the beginning of Section \ref{sec:policy}. Thus, setting $\alpha_{it} = \sum_{j \in \Jcal} a_{ij} \ts \gamma \ts \zetabar_{jt} \sum_{\ell \in \Ncal} \ind{\pbar_{jt}^\Fs=\ell} \ts \lambda_{jt}^\ell$ to capture the parameter of the Bernoulli random variable $\Nrm_{it}$, by the last chain of inequalities, as well as (\ref{eqn:avail_param_expost}), we get $\sum_{t \in \Tcal} \alpha_{it} \leq \gamma \ts \frac{c_{\min}}{c_{\min} + \Deltabar} \ts c_i \ts (1 + \frac{\Deltabar}{c_{\min}}) = \gamma \ts c_i$. We get $\sum_{t \in \Tcal} \mathbb E \{ \Nrm_{it} \} \leq \gamma \ts c_i$ and $\text{\sf Var}(\sum_{t \in \Tcal} \Nrm_{it}) = \sum_{t \in \Tcal} \alpha_{it} \ts (1 - \alpha_{it}) \leq \sum_{t \in \Tcal} \alpha_{it}$. In this case, we can use precisely the same argument in Lemma \ref{lem:avail_prob} 
to bound the availability probability $\mathbb P \{ \Grm_{jt} = 1\}$ for having remaining resource capacities to make product $j$ available at time period $t$ under the ex-post approximate policy by the same expression in Lemma \ref{lem:avail_prob}.

By the discussion so far, we lower bound the performance of the ex-post approximate policy, as well as the availability probabilities. Using these results, we give a proof of Theorem \ref{thm:expost}.

{\bf \underline{Proof of Theorem \ref{thm:expost}\phantom{p}\!\!\!}:}
\\
\indent Consider the ex-post approximate policy with the tuning parameter $\gamma = 1 - \sqrt{\frac{2 \log c_{\min}}{c_{\min}}}$. Noting the discussion just after (\ref{eqn:avail_param_expost}), if we set $\availp(\gamma) = 1 - L \ts \exp( -\frac 12 (1-\gamma)^2 \ts c_{\min})$, then we have \mbox{$\mathbb P \{ \Grm_{jt} = 1\} \geq \availp(\gamma)$}. In this case, choosing the tuning parameter as $\gamma = 1 - \sqrt{\frac{2 \log c_{\min}}{c_{\min}}}$ and noting that $(1-\gamma)^2 = \frac{2 \ts \log c_{\min}}{c_{\min}}$, we have $\availp(\gamma) = 1 - \frac{L}{c_{\min}}$. Using this value of the tuning parameter in Lemma \ref{lem:apx_lp_expost}, we obtain $\frac{\apxp}{Z_\lp^*} \geq \frac{c_{\min}}{c_{\min} + \Deltabar} (1 - \sqrt{\frac{2 \log c_{\min}}{c_{\min}}}) \ts (1 - \frac{L}{c_{\min}})$. To lower bound the right side of the last inequality, we observe that $\frac{c_{\min}}{c_{\min} + \Deltabar} \ts  (1 - \frac{L}{c_{\min}}) \geq (1 - \frac{\Deltabar}{c_{\min}}) \ts (1 - \frac{L}{c_{\min}}) \geq 1 - \frac{L + \Deltabar}{c_{\min}}$. In this case, we obtain $\frac{\apxp}{Z_\lp^*} \geq  (1 - \sqrt{\frac{2 \log c_{\min}}{c_{\min}}}) \ts (1 - \frac{L + \Deltabar}{c_{\min}}) \geq 1 - \sqrt{\frac{2 \log c_{\min}}{c_{\min}}} - \frac{L + \Deltabar}{c_{\min}}$. \qed

\vspace{-0mm}

\section{Optimal Solution to the Fluid Approximation in the Counterexample}
\label{sec:unbounded_diff}

\vspace{-1mm}

We construct an optimal solution to the \ref{eqn:fluid}. Noting that there is a single product, considering the decision variables $(x_{1t}^\ell : \ell = 1,2,~t=1,\ldots,C^2)$, because $\lambda_{1t}^2 = 0$ for $t \neq 1$, it is enough to work with the decision variables $(x_{1t}^1 : t=1,\ldots,C^2)$ and $x_{11}^2$. We can set the values of the remaining decision variables $(x_{1t}^2 : t=2,\ldots,C^2)$ to zero without changing the optimal objective value of the \ref{eqn:fluid}. In the \ref{eqn:fluid}, consider fixing the values of the decision variable $x_{11}^2$ at $\alpha$. In this case, by the second constraint, we have $y_{11}^2 \geq \alpha$, but noting the third constraint with $\ell = 2$, we get $y_{1t}^2 \geq y_{1,t-1}^2 \geq \alpha$ for all $t = \Tcal \setminus \{1\}$. Because $\lambda_{1t}^1 = 1$ for all $t \in \Tcal$ and $\lambda_{11}^2 = 1$, by the first constraint, we get $\sum_{t=1}^{C^2} x_{1t}^1 \leq C - \alpha$. In this case, we can obtain an optimal solution to the \ref{eqn:fluid} by solving the linear program 
\begin{align}
\alpha + \max_{(\xvec,\yvec) \in [0,1]^{3 \ts C^2}} \Bigg\{ \frac{1}{C} \sum_{t=1}^{C^2} x_{1t}^1 ~:~ & \sum_{t =1}^{C^2} x_{1t}^1 \leq C - \alpha,~
x_{1t}^1 \leq y_{1t}^1~~\forall \ts t=1,\ldots,C^2,
\nonumber
\\
& y_{1,C^2}^2 \geq y_{1,C^2-1}^2 \geq \ldots \geq y_{11}^2 \geq \alpha,~~
y_{1t}^1 + y_{1t}^2 = 1 ~~\forall \ts t = 1,\ldots,C^2 \Bigg\}.
\label{eqn:small_fluid}
\end{align}
The four constraints in (\ref{eqn:small_fluid}) are the four constraints in the \ref{eqn:fluid}. We add the constant $\alpha$ to the optimal objective value above because we fix the value of the decision variable $x_{11}^2$ at $\alpha$ and this decision variable appears with a coefficient of one in the objective function of the \ref{eqn:fluid}. By the fourth constraint in (\ref{eqn:small_fluid}), we have $y_{1t}^1 = 1 - y_{1t}^2$ and the decision variable $y_{1t}^1$ is an upper bound on the decision variable $x_{1t}^1$, so it is optimal to set the value of the decision variable $y_{1t}^2$ as small as possible so that the value of the decision variable $y_{1t}^1$ becomes as large as possible. Noting the third constraint above, we set $y_{1t}^2 = \alpha$ for all $t=1,\ldots,C^2$, in which case, by the second constraint above, we get $x_{1t}^1 \leq 1-\alpha$ for all $t=1,\ldots,C^2$. Thus, we can set the values of all decision variables $x_{1t}^1 = 1-\alpha$ for all $t=1,\ldots,C^2$ as long as the first constraint in (\ref{eqn:small_fluid}) allows doing so. Therefore, the optimal objective value of the maximization problem in (\ref{eqn:small_fluid}) is given by $\frac{1}{C} \min\{ C - \alpha , C^2 \ts (1-\alpha)\}$. Adding the constant $\alpha$ above, to obtain an optimal solution to the \ref{eqn:fluid}, we need to choose the value of $\alpha$ to maximize $f(\alpha) = \alpha + \frac{1}{C} \min\{ C - \alpha , C^2 \ts (1-\alpha)\}$ over all $\alpha \in [0,1]$. Because the minimum of two affine functions is piecewise linear and concave, the maximizer of the last function occurs either at a point of non-differentiability or at $\alpha = 0,1$. We have $f(0) = f(1) = 1$. To find a point of non-differentiability, setting $C - \alpha = C^2 \ts (1-\alpha)$ and solving for $\alpha$, we get $\alpha = C / (1+C)$. Because $f(\frac{C}{1+C}) = 2C/(1+C) \geq 1$, it is optimal to choose $\alpha = C/(1+C)$ yielding the optimal objective value of $2C/(1+C)$ for the \ref{eqn:fluid}.

\section{Feasible Price Paths Under Price Monotonicity Constraints}
\label{sec:tum_mon}

By the discussion at the end of Section \ref{sec:gen_const}, under price monotonicity constraints, we use the polytope $\Pcal_j = \{ \yvec_j \in [0,1]^{n \times T} : \sum_{k=\ell}^n y_{j,t-1}^k + \sum_{k=1}^{\ell-1} y_{jt}^k \leq 1~\forall \ts \ell \in \Ncal,~t \in \Tcal \setminus \{1\},~~\sum_{\ell \in \Ncal} y_{jt}^\ell = 1 ~\forall \ts t \in \Tcal\}$ to capture the set of feasible price paths for product $j$. We index the price level and time period pairs $\{ (\ell,t) : \ell \in \Ncal,~t \in \Tcal\}$ by using the integers $\{1,\ldots,n T\}$, so that the price level and time period pair~$(\ell,t)$ corresponds to the integer $\ell + (t-1) \ts n$. Thus, we use the vector of decision variables \mbox{$\zvec_j = (z_j^q : q = 1,\ldots,nT)$} instead of the vector of decision variables \mbox{$\yvec_j = (y_{jt}^\ell : \ell \in \Ncal,~t \in \Tcal)$}. In this case, considering the first constraint in the polytope $\Pcal_j$, capturing the price level and time period pair $(\ell,t)$ with the integer $\ell + (t-1) \ts n$, the decision variables $\{ y_{j,t-1}^k: k =\ell,\ldots,n\}$ correspond to the decision variables \mbox{$\{ z_j^q : q = \ell + (t-2) \ts n ,\ldots, n + (t-2) n\}$}, whereas the decision variables \mbox{$\{ y_{jt}^k : k=1,\ldots,\ell-1\}$} correspond to the decision variables $\{z_j^q : q = 1 + (t-1) \ts n ,\ldots, \ell-1 + (t-1) \ts n\}$. Thus, noting that $n +(t-2) \ts n = (t-1) \ts n$, the first constraint takes the form $\sum_{q=\ell + (t-2) n}^{\ell-1 + (t-1) n} z_j^q \leq 1$. Using a similar argument, considering the second constraint in the polytope $\Pcal_j$, this constraint takes the form $\sum_{q=1 + (t-1)n}^{n + (t-1) n} z_j^q =1$. By the preceding discussion, we can express the polytope $\Pcal_j$ as  $\Pcal_j = \{ \zvec_j \in [0,1]^{nT} : \sum_{q=\ell + (t-2) n}^{\ell-1 + (t-1) n} z_j^q \leq 1 ~\forall \ts \ell \in \Ncal,~t \in \Tcal \setminus \{1\},~~\sum_{q=1 + (t-1)n}^{n + (t-1) n} z_j^q =1 ~\forall \ts t \in \Tcal\}$.~In this representation of the polytope $\Pcal_j$, each constraint row for the polytope $\Pcal_j$ involves a sum of a number of consecutive decision variables in the vector $(z_j^q: q =1,\ldots,nT)$, which is to say that we can capture the constraint matrix for the polytope $\Pcal_j$ by using an interval matrix. By Corollary III.1.2.10 in \cite{NeWo88}, interval matrices are totally unimodular, so the extreme points of the polytope $\Pcal_j$ have integer values. 

\section{Feasible Price Paths Under Promotion Fatigue Constraints}
\label{sec:tum_fatigue}

Using the vector of decision variables $\yvec_j = (y_{jt}^\dis : t \in \Tcal)$, by the discussion at the end of Section \ref{sec:gen_const}, under promotion fatigue constraints, we can capture the set of feasible price paths for product~$j$ by using the polytope \mbox{$\Pcal_j = \{ \yvec_j \in [0,1]^T : \sum_{\tau = t}^{(t+K-1) \wedge T}  y_{j\tau}^\dis \leq 1~  \forall \ts t \in \Tcal, ~~  \ts y_{jt}^\dis \leq 1 ~ \forall \ts t \in \Tcal \}$}. In the first constraint in the polytope $\Pcal_j$, each constraint row involves a sum of a number of consecutive decision variables in the vector $(y_{jt}^\dis : t \in \Tcal)$, so we can capture the first constraint by using an interval matrix, which is totally unimodular. Considering the second constraint in the polytope $\Pcal_j$, each constraint row involves only one of the decision variables $(y_j^\dis : t \in \Tcal)$, so we can capture the second constraint by using an identity matrix. By Proposition III.1.2.1 in \cite{NeWo88}, appending the identity matrix to a totally unimodular matrix results in a totally unimodular matrix, so the extreme points of $\Pcal_j$ have integer values.

\section{Processing the Hotel Dataset}
\label{sec:hotel_data}

We discuss our approach for estimating the parameters $\{ (\beta_s , \alpha_s) : s =1,\ldots,9\}$ in our demand model. We preprocess the dataset as follows. 
The dataset includes the customers who made reservations, but does not include the customers who inquired about the charged price and chose not to make a reservation. We proceed with the assumption that there are a total of 40 customer arrivals on each day, which is significantly larger than the number of bookings on any day. We add no-purchase records into the dataset so that there are a total of 40 customer arrivals on each day. For each added no-purchase record, we assign a lead time and number of nights of stay that are sampled by using the fractions of bookings estimated from the dataset, as well as assign a quoted price per stay day equal to the average of the price per stay day in the dataset quoted on the date of the \mbox{no-purchase} record. This approach is used in several papers to augment the dataset, including, for example, \cite{GaMa21} and \cite{BeGa22}, where the authors add no-purchase records to datasets that only record the purchases from the customers. We worked with different numbers of added no-purchase records and our computational results remained qualitatively the same. Customers can book 63 days in advance with 40 customer arrivals per day, resulting in $T = 63 \times 40 = 2520$ time periods in the selling horizon. We use maximum likelihood estimation to estimate the parameters $\{ (\beta_s , \alpha_s) : s =1,\ldots,9\}$  in the purchase probability $1 / (1 + \exp(\beta_s + \alpha_s \ts(40 + \frac{160}{K-1} \ts (\ell-1)) \times q))$. When doing so, the likelihood function is separable by the different segments in the dataset, so we can estimate the parameters $(\beta_s , \alpha_s)$ separately for each segment, but noting that our dataset spans a year, there are multiple weeks in the dataset that correspond to the same segment. All of the booking records taking place in segment $s$ play a role when estimating the parameters $(\beta_s,\alpha_s)$. 

\end{APPENDICES}

\end{document}